\documentclass[12pt]{article}

\pdfoutput=1

\usepackage{amsmath, amssymb, amsthm, enumerate, fullpage, bm}
\usepackage{verbatim, enumitem, bbm, etoolbox}
\usepackage{latexsym}

\apptocmd{\lim}{\limits}{}{}

\theoremstyle{definition}
\newtheorem{thm}{Theorem}[section]
\newtheorem{theorem}[thm]{Theorem}

\newtheorem{conjecture}[thm]{Conjecture}
\newtheorem{lemma}[thm]{Lemma}

\numberwithin{subcase}{case}

\theoremstyle{definition}
\newtheorem{definition}[thm]{Definition}
\newtheorem{corollary}[thm]{Corollary}
\newtheorem{remark}[thm]{Remark}
\newtheorem{example}[thm]{Example}
\newtheorem{question}[thm]{Question}

\def\forkindep{\mathrel{\raise0.2ex\hbox{\ooalign{\hidewidth$\vert$\hidewidth\cr\raise-0.9ex\hbox{$\smile$}}}}}

\def\Ind{\setbox0=\hbox{$x$}\kern\wd0\hbox to 0pt{\hss$\mid$\hss}
	\lower.9\ht0\hbox to 0pt{\hss$\smile$\hss}\kern\wd0}

\def\Notind{\setbox0=\hbox{$x$}\kern\wd0\hbox to 0pt{\mathchardef
		\nn=12854\hss$\nn$\kern1.4\wd0\hss}\hbox to 0pt{\hss$\mid$\hss}\lower.9\ht0
	\hbox to 0pt{\hss$\smile$\hss}\kern\wd0}

\def\phi{\varphi}

\def\<{\langle}
\def\>{\rangle}

\makeatletter
\def\blfootnote{\xdef\@thefnmark{}\@footnotetext}
\makeatother

\begin{document}	

	\bibliographystyle{plain}
	
	\author{Danielle Ulrich\!\!\
	\thanks{Partially supported
by Laskowski's NSF grants DMS-1308546 and DMS-2154101.}\\
Department of Mathematics\\University of Maryland, College Park}
	\title{Amalgamation and Keisler's Order}
	\date{\today} 
	
	\blfootnote{2010 \emph{Mathematics Subject Classification:} 03C55.}
	\blfootnote{\emph{Key Words and Phrase:} Uncountable model theory, Keisler's order, ultraproducts.}
	
	\maketitle
	
	
\begin{abstract}
Malliaris and Shelah famously proved that Keisler's order $\trianglelefteq$ has infinitely many classes. In more detail, for each $2 \leq k < n < \omega$, let $T_{n, k}$ be the theory of the random $k$-ary $n$-clique free hypergraph. Malliaris and Shelah show that whenever $k+1 < k'$, then $T_{k+1, k} \not \trianglelefteq T_{k'+1, k'}$. However, their arguments do not separate $T_{k+1, k}$ from $T_{k+2, k+1}$, and the model-theoretic properties detected by their ultrafilters are difficult to evaluate in practice.

We uniformize the relevant ultrafilter constructions and obtain sharper model-theoretic bounds. As a sample application, we prove the following: suppose $3 \leq k < \aleph_0$, and $T$ is a countable low theory. Suppose that every independent system $(M_s: s \subsetneq k)$ of countable models of $T$ can be independently amalgamated. Then $T_{k, k-1} \not \trianglelefteq T$. In particular, for all $k < k'$, $T_{k+1, k} \not \trianglelefteq T_{k'+1, k'}$.
\end{abstract}

\section{Introduction}

Keisler proved the following fundamental theorem in \cite{Keisler}:

\begin{theorem}\label{KeislerOrigSecond}
	Suppose $T$ is a complete countable theory, and $\mathcal{U}$ is a $\lambda$-regular ultrafilter on $\mathcal{P}(\lambda)$, and $M_0, M_1 \models T$. Then $M_0^\lambda/\mathcal{U}$ is $\lambda^+$-saturated if and only if $M_1^\lambda/\mathcal{U}$ is. 
\end{theorem}

Motivated by this theorem, Keisler investigated the following  pre-ordering $\trianglelefteq$ on complete first-order theories; $\trianglelefteq$ is now called Keisler's order.

\begin{definition}
	Suppose $\mathcal{U}$ is a $\lambda$-regular ultrafilter on $\mathcal{P}(\lambda)$. Then say that $\mathcal{U}$ $\lambda^+$-saturates $T$ if for some or every $M \models T$, $M^\lambda/\mathcal{U}$ is $\lambda^+$-saturated.

	Given complete countable theories $T_0, T_1$, say that $T_0 \trianglelefteq_\lambda T_1$ if whenever $\mathcal{U}$ is a $\lambda$-regular ultrafilter on $\mathcal{P}(\lambda)$, if $\mathcal{U}$ $\lambda^+$-saturates $T_1$ then $\mathcal{U}$ $\lambda^+$-saturates $T_0$. Say that $T_0 \trianglelefteq T_1$ if $T_0 \trianglelefteq_\lambda T_1$ for all $\lambda$.

\end{definition}

We are interested in constructing dividing lines in Keisler's order. To be precise, a dividing line in Keisler's order is a set of complete countable theories which is downward-closed under $\trianglelefteq$. A principal dividing line in Keisler's order is a dividing line induced by a single ultrafilter $\mathcal{U}$; in other words, $\mathbf{T}$ is a principal dividing line if $\mathbf{T}$ is the set of all complete countable theories $\lambda^+$-saturated by some $\lambda$-regular ultrafilter on $\mathcal{P}(\lambda)$. In \cite{Optimals}, Mallairis and Shelah show that if there is a supercompact cardinal, then simplicity is a principal dividing line in Keisler's order. In \cite{InfManyClass}, they use the ultrafilters constructed in \cite{Optimals} to show (in ZFC) that Keisler's order has infinitely many classes, and in \cite{LowDividingLine}, I use the same ultrafilters to show that lowness is a principal dividing line in Keisler's order.\footnote{We use the original definition of Buechler \cite{Buechler}, so in particular low implies simple.} Finally, in \cite{NewSimpleTheory}, Malliaris and Shelah use these ultrafilters to show that Keisler's order is not linear in ZFC, improving earlier consistency results obtained independently by the authors \cite{InterpOrders} and myself \cite{KeislerNotLinear}.

In this paper, we give a uniform treatment of these ultrafilter constructions, and we investigate the model-theoretic properties detected by the ultrafilters of \cite{InfManyClass}. The main focus of our treatment is on generalized amalgamation properties of simple theories. In particular, we study the hypergraph examples $T_{n, k}$ of Hrushovski \cite{Hrush}: for $n > k \geq 2$, $T_{n, k}$ is the random $k$-ary $n$-clique free hypergraph. These were used by Malliaris and Shelah in \cite{InfManyClass} to show that Keisler's order has infinitely many classes, although they subtract $1$ from both indices in $T_{n, k}$. Each $T_{n, 2}$ has $SOP_2$, and thus is maximal in $\trianglelefteq$. In this paper, we will mainly be interested in the case when $k \geq 3$ and $n = k+1$. Malliaris and Shelah prove that for all $k < k'-1$, $T_{k+1, k} \not \trianglelefteq T_{k'+1, k'}$; we are able to show this holds for all $k < k'$, as expected.

We now give a more detailed overview of the paper. To begin, in Section~\ref{Combinatorics}, we refine the combinatorics used by Malliaris and Shelah to get infinitely many classes in Keisler's order \cite{InfManyClass}, and explain how this yields at once that $T_{k+1, k} \trianglelefteq T_{k'+1, k'}$ for all $k < k'$. We do not give full details since that would require importing a substantial amount of technology from \cite{Optimals} \cite{InfManyClass}, and we intend to prove a sharper result by different means.

In Sections \ref{KeislerNewIndSys}, \ref{ColoringSec} and \ref{ConsequencesSec}, we study several notions of generalized amalgamation, which are conjecturally equivalent.

 In Section~\ref{KeislerNewIndSys}, we define independent system of models, and what it means to amalgamate them. Specifically, suppose $T$ is a countable simple theory and $n < \omega$. Let $\mathcal{P}^-(n)$ denote the set of all proper subsets of $n$. Then $T$ has $\mathcal{P}^-(n)$-amalgamation of models if every independent system of models $(M_s: s \in \mathcal{P}^-(n))$ indexed by $\mathcal{P}^-(n)$ has a solution (note $n= \{0, 1, \ldots, n-1\}$).  This is similar to the notion of $n$-simplicity investigated by Kim, Kolesnikov and Tsuboi \cite{GenAmalg}, although we just look at independent systems of models, rather than boundedly closed sets. This is also related to the more technical notion of $<k$-type amalgamation, introduced in \cite{SP} (joint with Shelah). In fact, Theorem~\ref{GenAmalgImpliesTypeAmalg} states that $\mathcal{P}^-(k)$-amalgamation of models implies $<k$-type amalgamation.
%
%
%
%
%
%

In Section~\ref{ColoringSec}, we consider chromatic numbers of the comparability graphs of partial types. In more detail:

\begin{definition}
	Suppose $k \leq \aleph_0$ and $\mathcal{H} \subseteq [X]^{<k}$ is a hypergraph. Then let $\chi(\mathcal{H})$, the chromatic number of $\mathcal{H}$, be the least cardinal $\mu$ such that there is some coloring $c: X \to \mu$ such that for all $\alpha < \mu$, $\mathcal{H} \cap [c^{-1}(\alpha)]^{<k} = \emptyset$.
	
	Suppose $P$ is a partial order; for each $3 \leq k \leq \aleph_0$, let $\chi(P, k)$, the $k$-ary chromatic number of $P$, be the least cardinal $\mu$ such that there is some function $c: P \to \mu$ such that for all $u \in [P]^{<k}$, if $c \restriction_u$ is constant then $u$ has a lower bound in $P$. In other words, $\chi(P, k)$ is the chromatic number of $\mathcal{H} := \{u \in [P]^{<k}: u \mbox{ has no lower bound in } P\}$. 
	
	Suppose $\theta$ is a regular cardinal, $T$ is a countable simple theory, $M \models T$ and $M_0 \preceq M$ is countable. Then let $\Gamma^\theta_{M, M_0}$ be the partial order of all partial types $p(x)$ over $M$ of cardinality less than $\theta$, which do not fork over $M_0$; we order $\Gamma^\theta_{M, M_0}$ by reverse inclusion. $x$ can be a finite tuple here; alternatively, we should pass to the elimination of imaginaries $T^{eq}$.
	
	Suppose $\theta$ is a regular cardinal, $\mu = \mu^{<\theta}$, $\lambda \leq 2^\mu$, and $3 \leq k \leq \aleph_0$. Then say that $T$ has the $(\lambda, \mu, \theta, k)$-coloring property if whenever $M \models T$ with $|M| \leq \lambda$ and whenever $M_0 \preceq M$ is countable, then $\chi(\Gamma^{\theta}_{M, M_0}, k) \leq \mu$.
\end{definition}

We also define the $(\lambda, \mu, \theta, k)$-coloring property for $\lambda > 2^\mu$, but the right way to do this is slightly different.

In Section~\ref{ConsequencesSec}, we first recall some definitions from \cite{InterpOrders2Ulrich}.

\begin{definition}
$\Delta$ is a pattern on $I$ if $\Delta \subseteq [I]^{<\aleph_0}$ is closed under subsets. If $T$ is a complete first order theory, and $\phi(\overline{x}, \overline{y})$ is a formula of $T$, then say that $\phi(\overline{x}, \overline{y})$ admits $\Delta$ if we can choose $M \models T$ and $(\overline{a}_i: i \in I)$ from $M^{|\overline{y}|}$, such that for all $s \in [I]^{<\aleph_0}$, $M \models \exists \overline{x} \bigwedge_{i \in s} \phi(\overline{x}, \overline{a}_i)$ if and only if $s \in \Delta$.

\end{definition}

 In \cite{InterpOrders2Ulrich}, we isolated a particular pattern $\Delta_{k+1, k}$, such that $T_{k+1, k}$ is the $\trianglelefteq$-minimal theory admitting $\Delta_{k+1, k}$.

Corollary~\ref{SyntConsOfKeislerCor} states the following; (B) implies (C) is proved in \cite{SP}. 

\begin{corollary}\label{SyntCons1}
	Suppose $T$ is simple and $3 \leq k \leq \aleph_0$. Then (A) implies (B) implies (C) implies (D):
	
	\begin{itemize}
		\item[(A)] $T$ has $\mathcal{P}^-(k)$-amalgamation of models;
		\item[(B)] $T$ has $<k$-type amalgamation;
		\item[(C)] For every regular uncountable $\theta$, for every $\mu = \mu^{<\theta}$, and for every $\lambda \leq 2^\mu$, $T$ has the $(\lambda, \mu, \theta, k)$-coloring property, and moreover this statement continues to hold in every forcing extension;
		\item[(D)] $T$ does not admit $\Delta_{k'+1, k'}$ for any $k' < k$.
	\end{itemize}
\end{corollary}

\begin{conjecture}\label{Conj1}
	(A), (B), (C), (D) above are equivalent.
\end{conjecture}

In Section~\ref{Refinements}, we give some refinements of Corollary~\ref{SyntCons1} when $\lambda$ is small. In Section \ref{ForcingSec}, we review complete Boolean algebras and forcing iterations.

In Section~\ref{ReviewSec}, we recall the setup of full Boolean-valued models from \cite{BVModelsUlrich}. In particular, for every ultrafilter $\mathcal{U}$ on a complete Boolean algebra $\mathcal{B}$, and for every complete countable theory $T$, we define what it means for $\mathcal{U}$ to $\lambda^+$-saturate $T$. Keisler's order can be framed as follows: $T_0 \trianglelefteq_\lambda T_1$ if and only if for every complete Boolean algebra $\mathcal{B}$ with the $\lambda^+$-c.c. and for every ultrafilter $\mathcal{U}$ on $\mathcal{B}$, if $\mathcal{U}$ $\lambda^+$-saturates $T_1$, then $\mathcal{U}$ $\lambda^+$-saturates $T_0$; and $T_0 \trianglelefteq T_1$ if and only if $T \trianglelefteq_\lambda T_1$ for all $\lambda$. 

To motivate our strategy for attacking Keisler's order, consider Corollary~\ref{SyntCons1}(C) above. It turns out that the quantification over forcing extensions is necessary. Specifically, in joint work with Shelah \cite{SP}, we prove the following theorem; it is key in our analysis of $\leq_{SP}$-ordering. 

\begin{theorem}\label{SPTheorem}
Suppose $3 \leq k \leq \aleph_0$ and $\theta = \theta^{<\theta} > \aleph_0$, and $\lambda \geq \theta^{+ \omega}$ is a cardinal. Then after passing to a $\theta$-closed, $\theta^+$-c.c. forcing extension, we can arrange that $2^\theta \geq \lambda$, and further:

\begin{itemize}
	\item If $T$ is a countable simple theory with $<k$-type amalgamation, then $T$ has the $(\lambda, \theta, \theta, \aleph_0)$-coloring property;
	\item If $T$ admits $\Delta_{k'+1, k'}$ for some $k' < k$, then $T$ fails the $(\lambda, \theta, \theta, k)$-coloring property (and hence also the $(\lambda, \theta, \theta, \aleph_0)$-coloring property).
\end{itemize}
\end{theorem}

The forcing extension is not hard to describe. Namely, begin by adding $\lambda$-many $\theta$-Cohens. In this forcing extension, every theory $T$ which admits $\Delta_{k'+1, k'}$ for some $k' < \aleph_0$ will fail the $(\lambda, \theta, \theta, k')$-coloring property (and hence also the $(\lambda, \theta, \theta, \aleph_0)$-coloring property). Now suppose $T$ is a countable simple theory with $<k$-type amalgamation; we want to arrange that $T$ has the $(\lambda, \theta, \theta, \aleph_0)$-coloring property. Let $M \models T$ have $|M| \leq \lambda$, and let $M_0 \preceq M$ be countable. We need to arrange that $\chi(\Gamma^{\theta}_{M, M_0}, \aleph_0) \leq \theta$. The natural way to achieve this is by forcing over $\prod_{\alpha < \theta} \Gamma^\theta_{M, M_0}$, where the product is taken with $<\theta$-supports; indeed, this adds a sequence of $\theta$-many generic types over $M$ which do not fork over $M_0$, and we can color $p(x) \in \Gamma^{\theta}_{M, M_0}$ by the least $\alpha < \theta$ such that the $\alpha$'th generic type extends $p(x)$. To make this work we prove a certain iteration preservation theorem (Theorem 3.5 of \cite{SP}, or see Theorem~\ref{IterationPreserves1} of the present work).

In this way, we get in $ZFC$ some large poset $P$ which forces that the statement of Theorem~\ref{SPTheorem} holds. But now, instead of forcing over $P$, we instead take the Boolean algebra completion $\mathcal{B}(P)$ and attempt to construct sufficiently generic ultrafilters on it in $\mathbb{V}$. It turns out that the combinatorics of arranging $(\lambda, \mu, \theta, \aleph_0)$-colorability in the forcing extension are deeply intertwined with saturation by sufficiently generic ultrafilters in $\mathbb{V}$, and in particular, sufficiently generic ultrafilters on $\mathcal{B}(P)$ will (with some caveats) detect the division between $<k$-type amalgamation and admitting $\Delta_{k'+1, k'}$ for some $k' < k$.

In Section~\ref{KeislerNewForcingAmalg}, we review the relevant forcing iteration theorem from \cite{SP}, and extend it to the case $\theta = \aleph_0$. In particular, given a regular cardinal $\theta$, a cardinal $\mu = \mu^{<\theta}$ and a cardinal $3 \leq k \leq \aleph_0$, we define a class of forcing notions $\mathbb{P}_{\mu, \theta, k}$. Every $P \in \mathbb{P}_{\mu, \theta, k}$ has the $\mu^+$-c.c. and is $\theta$-closed, and $\mathbb{P}_{\mu, \theta, k}$ is closed under $<\theta$-support forcing iterations. We remark that essentially any iteration preservation theorem can be plugged into our proof to get a dividing line in Keisler's order.

In Section~\ref{KeislerNewUltCons}, we give the general ultrafilter construction.

\begin{definition}
	The cardinal $\tau$ is supercompact if for every set $X$, there is a normal, $\tau$-complete ultrafilter over $[X]^{<\tau}$. In more words, there is a $\tau$-complete ultrafilter $\mathcal{U}$ over $[X]^{<\tau}$, such that for every $x \in X$, $\{u \in [X]^{<\tau}: x \in u\} \in \mathcal{U}$, and further, if $f$ is a choice function on $[X]^{<\tau} \backslash \{\emptyset\}$ then $f$ is constant on some set in $\mathcal{U}$.
	
	$(\lambda,\mu, \theta, \tau, k)$ is a suitable sequence of cardinals if:
	\begin{itemize}
		\item $\tau \leq \theta$ and $\mu = \mu^{<\theta}$ and $\lambda > \mu$;
		\item $\theta$ is regular, and $\tau$ is either $\aleph_0$ or else supercompact.
	\end{itemize}
\end{definition}

For our construction, we take as input a suitable sequence $\mathbf{s} = (\lambda, \mu, \theta, \tau, k)$ of cardinals. We then build a long forcing iteration $P$ from $\mathbb{P}_{\mu, \theta, k}$ and a sufficiently generic $\tau$-complete ultrafilter $\mathcal{U}$ on the Boolean algebra completion $\mathcal{B}(P)$, and check which complete countable theories does $\mathcal{U}$ $\lambda^+$-saturate. 

In more detail, for every suitable sequence $\mathbf{s}$ and for every $3 \leq k \leq \theta$, we define two properties of theories, namely: the $\mathbf{s}$-extension property, and the smooth $\mathbf{s}$-extension property. Essentially, these detect ways in which we can locally solve saturation problems.  Theorem~\ref{SetTheoryDividingLines} states that there is some $P \in \mathbb{P}_{\mu, \theta, k}$ and some $\tau$-complete ultrafilter $\mathcal{U}$ on $\mathcal{B}(P)$, such that $\mathcal{U}$ $\lambda^+$-saturates every theory with the smooth $\mathbf{s}$-extension property, and does not $\lambda^+$-saturate any theory without the $\mathbf{s}$-extension property. In particular, there is a principal dividing line in Keisler's order somewhere between the $\mathbf{s}$-extension property and the smooth $\mathbf{s}$-extension property. We also show that this principal dividing line excludes every unsimple theory, and if $\tau = \aleph_0$ then it excludes every nonlow theory.

In Section~\ref{KeislerNewSat}, we prove the following:

\begin{theorem}For every suitable sequence $\mathbf{s}= (\lambda, \mu, \theta, \tau, k)$, and for every complete countable theory $T$, if  $T$ admits $\Delta_{k'+1, k'}$ for some $k' < k$, and if $\lambda \geq \mu^{+\omega}$, then $T$ fails the $\mathbf{s}$-extension property.
\end{theorem}

\begin{theorem}\label{satPortionGeneralFirst}
	Suppose $\mathbf{s} = (\lambda, \mu, \theta, \tau, k)$ is suitable. Suppose $T$ is simple, and either $\tau > \aleph_0$ or else $T$ is low. Suppose that every $P \in \mathbb{P}_{\mu, \theta, k}$ forces that $T$ has the $(\lambda, |\mu|, \theta, k)$-coloring property (we write $|\mu|$ because $\mu$ may be collapsed in the forcing extension). Then $T$ has the smooth $\mathbf{s}$-extension property.
\end{theorem}

In Section~\ref{KeislerNewConc}, we summarize our results. In particular, for every $3 \leq k \leq \aleph_0$, there is a dividing line in Keisler's order which contains every theory with $<k$-type amalgamation, and which excludes every theory which admits $\Delta_{k'+1, k'}$ for some $k' < k$. In particular, this dividing line includes every theory with $\mathcal{P}^-(k)$ amalgamation of models. Thus: $T_{k+1, k} \not \trianglelefteq T_{k'+1, k'}$ for all $k < k'$. We also list several conjectures and questions.

\noindent \textbf{Acknowledgements.} We would like to thank Vincent Guingona, Alexei Kolesnikov, Chris Laskowski and Jindrich Zapletal for several helpful discussions. We also note that much of the work here is inspired by joint work with Saharon Shelah on the $\leq_{SP}$-ordering \cite{SP}.

\section{A refined set mapping theorem}\label{Combinatorics}
In \cite{InfManyClass}, Malliaris and Shelah proved that for all $k+1 < k'$, $T_{k+1, k} \not \trianglelefteq T_{k'+1, k'}$. A combinatorial difficulty prevented the authors from concluding this for all $k < k'$; here, we sketch one way to overcome this difficulty.

The following combinatorial notion is key to the arguments of \cite{InfManyClass}; we refer the reader to Sections 44 through 46 of \cite{Combinatorics} for a history.

\begin{definition}
	Suppose $F: [\lambda]^k \to [\lambda]^{<\mu}$. Then $t \in [\lambda]^{n}$ is independent with respect to $F$ if for each $s \in [t]^k$, $F(s) \cap t \subseteq s$.
	
	Given cardinals $\lambda \geq \mu$ and numbers $n > k$, say that $(\lambda, k, \mu) \rightarrow n$ if whenever $F: [\lambda]^k \to [\lambda]^{<\mu}$, there is some $t \in [\lambda]^n$ which is independent with respect to $F$.
\end{definition}
%
%
%
%
%
%

The following theorem is due to Kuratowski and Sierpi\'{n}ski \cite{KSComb}; or see Theorems 45.7 and 46.1 of \cite{Combinatorics}.
\begin{theorem}\label{NonSatLemmaWeak}
	Suppose $\lambda, \mu$ are cardinals. Then $(\lambda, \ell, \mu) \rightarrow \ell+1$ if and only if $\lambda \geq \mu^{+\ell}$.
\end{theorem}

However, it turns out this is not quite the right notion for what we want. The following refinement, made with a great deal of hindsight, isolates the combinatorial contingency on which the arguments of \cite{InfManyClass} depend:

\begin{definition}
	Suppose $\theta$ is a regular cardinal, and $\mu = \mu^{<\theta}$, and $2 \leq k < \aleph_0$, and $\lambda$ is infinite. Say that $F: [\lambda]^{<\aleph_0} \to [\lambda]^{<\theta}$ is inflationary if for all $w \in [\lambda]^{<\aleph_0}$, $w \subseteq F(w)$.  For every set $X$, let $P_{X \mu \theta}$ be the poset of all partial functions from $\lambda$ to $\mu$ of cardinality less than $\theta$, ordered by reverse inclusions.

	Say that $\lambda$ is big enough for $(\mu, \theta, k)$ if the following holds: suppose $F: [\lambda]^{<\aleph_0} \to [\lambda]^{<\theta}$ is inflationary, and $G: [\lambda]^{<\aleph_0} \to P_{\lambda \mu \theta}$. Then there is some $t \in [\lambda]^{k}$ and some sequence $(w_s: s \in [t]^{k-1})$ from $[\lambda]^{<\aleph_0}$, such that each $w_s \cap t = F(w_s) \cap t = s$, and such that $\{G(w_s): s \in [t]^{k-1}\}$ has a lower bound in $P_{\lambda \mu \theta}$ (i.e. $\bigcup \{G(w_s): s \in [t]^{k-1}\}$ is a function).
\end{definition}

\begin{remark}
	Suppose $\theta$ is a regular cardinal, $\mu = \mu^{<\theta}$, $2 \leq k'< k < \aleph_0$, and $\lambda < \lambda'$. If $\lambda$ is big enough for $(\mu, \theta, k)$ then $\lambda'$ is big enough for $(\mu, \theta, k')$ (i.e. we can increase $\lambda$ and decrease $k$).
\end{remark}

Thus the following definition makes sense:

\begin{definition}
	Suppose $\theta$ is a regular cardinal, $\mu = \mu^{<\theta}$, $2 \leq k < \aleph_0$. Then let $\Psi(\mu, \theta, k)$ be the least $\lambda$ such that $\lambda$ is big enough for $(\mu, \theta, k)$ (such a $\lambda$ will always exist). Rephrasing, $\lambda$ is big enough for $(\mu, \theta, k)$ if and only if $\lambda \geq \Psi(\mu, \theta, k)$.
\end{definition}

The arguments of Malliaris and Shelah \cite{InfManyClass} show, under different terminology, that $\mu^{+(k-1)} \leq \Psi(\mu, \theta, k) \leq \mu^{+k}$, and that $\Psi(\mu, \theta, 2)  = \mu^+$. The worrisome scenario that might allow $T_{k+1, k} \trianglelefteq T_{k+2, k+1}$ is $\Psi(\mu, \theta, \cdot)$ might not be injective: more exactly, there may be some $k$ such that $\Psi(\mu, \theta, k) = \Psi(\mu, \theta, k+1) = \mu^{+k}$. In this case, it is possible that none of the ultrafilters constructed by Malliaris and Shelah in \cite{InfManyClass} would separate $T_{k+1, k}$ from $T_{k+2, k+1}$. But we will show in fact that $\Psi(\mu, \theta, k+1) > \Psi(\mu, \theta, k)$ for all $k$, so this cannot happen.

We begin with the following simple observation; it is the combinatorial content of Lemma 4.2 of \cite{Optimals}. 

\begin{theorem}\label{CombThm0}
	Suppose $\theta$ is a regular cardinal and $\mu = \mu^{<\theta}$. Then $\Psi(\mu, \theta, 2) = \mu^+$.
\end{theorem}
\begin{proof}
	First, we show that $\mu^+$ is big enough for $(\mu, \theta, 2)$. Write $\lambda = \mu^+$. Suppose $F: [\lambda]^{<\aleph_0} \to [\lambda]^{<\theta}$ is inflationary, and $G: [\lambda]^{<\aleph_0} \to P_{\lambda \mu \theta}$.  For each $\alpha < \lambda$ let $g_\alpha = G(\{\alpha\})$. By the $\Delta$-system lemma we can find some $X \in [\lambda]^\lambda$ such that for all $\alpha, \beta \in X$, $g_\alpha$ and $g_\beta$ are compatible. Since $\lambda \geq \mu^+ \geq \theta^+$ and $(\mu^+, 1, \mu) \rightarrow 2$ we can find $\alpha < \beta$ both in $X$ such that $\alpha \not \in F(\beta)$ and $\beta \not \in F(\alpha)$. Then $t := \{\alpha, \beta\}$ and $w_{\{\alpha\}} := \{\alpha\}$, $w_{\{\beta\}} := \{\beta\}$ witness that $\lambda$ is big enough for $(\mu, \theta, 2)$.
	
	So it suffices to show that $\mu$ is not big enough for $(\mu, \theta, 2)$. But this is clear, since we can choose $G: [\mu]^{<\aleph_0} \to P_{\mu \mu \theta}$ so as to put $[\mu]^{<\aleph_0}$ in bijection with an antichain of $P_{\mu \mu \theta}$. 
\end{proof}

%
%

The following is essentially Lemma 1.6 of Malliaris and Shelah \cite{InfManyClass}.

\begin{theorem}\label{CombThm1}
	Suppose $\theta$ is a regular cardinal, $\mu = \mu^{<\theta}$, and $k \geq 2$. Then $\Psi(\mu, \theta, k) \leq \mu^{+k}$.
\end{theorem}
\begin{proof}
	Suppose $\lambda \geq \mu^{+k}$; we show $\lambda$ is big enough for $(\mu, \theta, k)$. Suppose $F: [\lambda]^{<\aleph_0} \to [\lambda]^{<\theta}$ is inflationary, and $G: [\lambda]^{<\aleph_0} \to P_{\lambda \mu \theta}$. Let $\mathcal{B}_{\lambda \mu \theta}$ be the Boolean algebra completion of $P_{\lambda \mu \theta}$. For each $t \in [\lambda]^{<\aleph_0}$ let $\mathbf{A}(t) \in \mathcal{B}_{\lambda \mu \theta}$ be defined as $\bigvee \{G(t'): t \subseteq t' \in [\lambda]^{<\aleph_0}\}$. Note that $t \subseteq t'$ implies $0 < \mathbf{A}(t') \leq \mathbf{A}(t)$.
	
	Since $P_{\lambda \mu \theta}$ (and hence $\mathcal{B}_{\lambda \mu \theta}$) has the $\mu^+$-c.c., for each $s \in [\lambda]^{k-1}$ we can find $H(s) \in [\lambda]^{\leq \mu}$ such that $\mathbf{A}(s) = \bigvee \{G(t): s \subseteq t \in [H(s)]^{<\aleph_0}\}$. By fattening $H$ we can suppose that $s \subseteq H(s)$ for all $s \in [\lambda]^{k-1}$, and that $H(s)$ is closed under $F$. Note $\lambda \geq (\mu^+)^{+(k-1)}$. Hence $(\lambda, k-1, \mu^+) \rightarrow k$; thus we can find some $t \in [\lambda]^k$ such that for all $s \in [w]^{k-1}$, $H(s) \cap w = s$.
	
	Enumerate $[t]^{k-1} = (s_0, \ldots, s_{k-1})$. Inductively choose $w_{s_i}$ with $s \subseteq w_{s_i} \in [H(s)]^{<\aleph_0}$ such that $G(w_{s_i})$ is compatible with $G(t) \cup \bigcup_{j < i} G(w_{s_j})$; this is possible because $G(t) \leq \mathbf{A}(s_i)$. Note then that for each $s \in [t]^{k-1}$,  $w_{s} \cap t = F(w_{s}) \cap t = s$ since $F(w_s) \subseteq H(s)$, and note also $\bigcup_{s \in [t]^{k-1}}G(w_s)$ is a function. This witnesses $\lambda$ is big enough for $(\mu, \theta, k)$.
	
\end{proof}

%
%

Finally, we prove the following. Lemma 2.5 of \cite{InfManyClass} implies that $\Psi(\mu, \theta, k) \geq \mu^{+(k-1)}$; what we prove is stronger, since we know $\Psi(\mu, \theta, 2) = \mu^+$.

\begin{theorem}\label{CombThm2}
	Suppose $\theta$ is a regular cardinal, and $\mu = \mu^{<\theta}$. Then for all $2 \leq k < \aleph_0$, $\Psi(\mu, \theta, k+1) > \Psi(\mu, \theta, k).$
\end{theorem}
\begin{proof}
	Suppose $\lambda$ is not big enough for $(\mu, \theta, k)$; it suffices to show $\lambda^+$ is not big enough for $(\mu, \theta, k+1)$. Choose $F: [\lambda]^{<\aleph_0} \to [\lambda]^{<\theta}$, $G:[\lambda]^{<\aleph_0} \to P_{\lambda \mu \theta}$ witnessing that $\lambda$ is not big enough for $(\mu, \theta, k)$.
	
	For each $\alpha < \lambda^+$, let $h_\alpha: \alpha \to |\alpha|$ be a bijection; note then that each $\mbox{range}(h_\alpha) \subseteq \lambda$.
	
	Let $F': [\lambda^+]^{<\aleph_0} \to [\lambda^+]^{<\theta}$ be defined as follows. Put $F'(\emptyset) = \emptyset$. If $w' \in [\lambda^+]^{<\aleph_0}$ is nonempty, then write $\max(w') = \alpha_*$ and write $w = h_{\alpha_*}[w' \backslash \{\alpha_*\}]$ and put $F'(w') = w' \cup h_{\alpha_*}^{-1}[F(w)]$. Define $G': [\lambda^+]^{<\aleph_0} \to P_{\lambda^+ \mu \theta}$ so that if $G'(w)$ and $G'(w')$ are compatible, then:
	
	\begin{itemize}
		\item[(I)] $w$ and $w'$ have the same cardinality; let $\phi: w \to w'$ be the unique order-preserving bijection;
		\item[(II)] $\phi$ is the identity on $w \cap w'$;
		\item[(III)] For all $\alpha \in w$ and for all $v \subseteq w \cap \alpha$, $G(h_\alpha[v])$ and $G(h_{\phi(\alpha)}[\phi[v]])$ are compatible.
	\end{itemize} 
	
	To see that this is possible: first of all, write $X= \lambda^+ + 1$ and define $G_0: [\lambda^+]^{<\aleph_0} \to P_{X \mu \theta}$ as follows (we can reindex to get a map into $P_{\lambda \mu \theta}$ afterwards). Suppose $w \in[\lambda^+]^{<\aleph_0}$; enumerate $w = \{\alpha_i: i < n\}$ in increasing order. Then let $\mbox{dom}(G_0(w)) = w \cup \{\lambda^+\}$, and put $G_0(w)(\alpha_i) = i$, and put $G_0(w)(\lambda^+) = n$. Then $G_0$ clearly works for conditions (I) and (II).
	
	Also, for each nonempty $s \in [\omega]^{<\aleph_0}$ (this is $\omega$ the ordinal, not $w$), define $G_s: [\lambda^+]^{<\aleph_0} \to P_{\lambda \mu \theta}$ as follows: suppose $w \in [\lambda^+]^{<\aleph_0}$; enumerate $w = \{\alpha_i: i < n\}$ in increasing order. If $s \not \subseteq n$ then let $G_s(w) = \emptyset$. Otherwise, let $i_* = \max(s)$, and let $G_s(w) = G(h_{i_*}[\alpha_i: i \in s, i < i_*])$. 
	
	Then choose $G': [\lambda^+]^{<\aleph_0} \to P_{\lambda^+ \mu \theta}$ so that if $G'(w)$ and $G'(w')$ are compatible, then $G_0(w)$ and $G_0(w')$ are compatible, as are $G_s(w)$ and $G_s(w')$ for each nonempty $s \in [\omega]^{<\aleph_0}$. Easily, this works.
	
	Suppose towards a contradiction $F', G'$ did not witness that $\lambda^+$ is not big enough for $(\mu,\theta, k+1)$. Then we can find $t' \in [\lambda^+]^{k+1}$ and $(w'_s: s \in [t']^{k})$ from $[\lambda^+]^{<\aleph_0}$, such that each $w'_s \cap t' = F'(w'_s) \cap t' = s$ and such that $\bigcup\{G'(w'_s): s \in [t']^{k}\}$ is a function. Write $\alpha_* = \max(t')$.
	
	For each $s, s' \in [t']^k$ with $s \not= s'$, let $\phi_{s', s}: w_{s'} \to w_{s}$ be the unique order-preserving isomorphism. Write $v_{s', s} := s \cup \phi_{s', s}[s']$. We claim that $v_{s', s}$ only depends on $s$. Indeed, write $u= s \cap s'$, and write $t' \backslash s = \{\alpha\}$, so $s' = u \cup \{\alpha\}$. Then $\phi_{s', s}[s'] = u \cup \phi_{s', s}(\alpha)$, since $\phi_{s', s}$ is the identity on $u$. Thus $v_{s', s} = s \cup \phi_{s', s}(\alpha)$. So suppose $s'' \in [t']^k$ is distinct from both $s$ and $s'$. Then $\alpha \in s' \cap s''$, hence $\alpha \in w_{s'} \cap w_{s''}$, hence $\phi_{s's''}(\alpha) = \alpha$. Since $\phi_{s's} = \phi_{s''s} \circ \phi_{s's''}$, plugging in $\alpha$ gives $\phi_{s's}(\alpha) = \phi_{s''s}(\alpha)$. Thus $v_{s', s} = v_{s'', s}$.
	
	For each $s \in [t']^k$ write $v_s = v_{s', s}$, for some or any $s' \in [w']^k$ with $s' \not= s$. Then $s \subseteq v_s \in [\lambda^+]^{k+1}$. Moreover, if $\max(s) = \alpha_*$ (which is the only case we will use) then $\max(v_s) = \alpha_*$, since each $\phi_{s', s}$ is order-preserving.  
	
	Let $t = h_{\alpha_*}[t' \backslash \{\alpha_*\}]$. For each $s \in [t]^{k-1}$, let $u_s = h_{\alpha_*}^{-1}[s] \cup \{\alpha_*\} \in [t']^k$, and let $w_s = h_{\alpha_*}[v_{u_s} \backslash \{\alpha_*\}]$.
	
	Note that each $w_s \cap t = F(w_s) \cap t = s$, since visibly each $s \subseteq w_s$, and if some $F(w_s) \supseteq t$, then $F'(w'_{u_s}) \supseteq t'$ by definition of $F'$.
	
	So it suffices to show that $\bigcup \{G(w_s): s \in [t]^{k-1}\}$ is a function, for which it suffices to check each pair. So suppose $s, s' \in [t]^{k-1}$. Write $u = u_s$, write $u' = u_{s'}$. Note that $\phi_{u, u'}[v_u] = v_{u'}$: indeed, $\phi_{u, u'}[v_u] = \phi_{u, u'}[u \cup \phi_{u', u}[u']] = \phi_{u, u'}[u] \cup u' = v_{u'}$. Thus, $G(w_s)$ and $G(w_{s'})$ are compatible, by condition (III) in the definition of $G$.
\end{proof}

\begin{corollary}\label{CombCor}
	Suppose $\theta$ is a regular cardinal and $\mu = \mu^{<\theta}$. Then there is a unique $k_*$ with $3 \leq k_* \leq \aleph_0$, such that for all $2 \leq k < k_*$, $\Psi(\mu, \theta, k) = \mu^{+(k-1)}$, and for all $k_* \leq k < \aleph_0$, $\Psi(\mu, \theta, k) = \mu^{+k}$.
\end{corollary}

\begin{question}
	What is this $k_*$?
\end{question}

It is not too hard to use Corollary~\ref{CombCor} to verify that $T_{k+1, k} \not \trianglelefteq T_{k'+1, k'}$ for all $k < k'$, following the proof of Malliaris and Shelah in \cite{Optimals} for the case $k+1 < k'$. We prefer to take a different approach, where we will always take $\lambda$ to be large (say $\geq \mu^{+\omega}$); this will give sharper model-theoretic bounds for our dividing lines.

Before moving on, we give a characterization of ``not big enough" that will be helpful later, in the case when $\theta$ is uncountable. Recall (for instance, see \cite{Jech}):

\begin{definition}
	Suppose $\theta$ is an uncountable regular cardinal and $\lambda \geq \theta$. Then $C \subseteq [\lambda]^{<\theta}$ is closed unbounded (club) if $C$ is closed under unions of ascending chains (or equivalently, upwards-directed systems), and for all $w \in [\lambda]^{<\theta}$ there is $v \in C$ with $w \subseteq v$.
\end{definition}

\begin{theorem}\label{CombThm3}
	
	Suppose $\theta$ is a regular uncountable cardinal, and $\mu = \mu^{<\theta}$, and $2 \leq k < \aleph_0$. Then $\lambda$ is big enough for $(\mu, \theta, k)$ if and only if for every club $C \subseteq [\lambda]^{<\theta}$ and for every map $G: C \to P_{\lambda \mu \theta}$, there is some $t \in [\lambda]^{k}$ and $(w_s: s \in [t]^{k-1})$ from $C$, such that each $w_s \cap t = s$,  and such that $\bigcup \{G(w_s): s \in [t]^{k-1}\}$ is a function.
\end{theorem}
\begin{proof}
	Suppose first that $\lambda$ is big enough for $(\mu, \theta, k)$. Let $C, G$ be given. Define $F': [\lambda]^{<\aleph_0} \to C$ so that each $w \subseteq F'(w) \in C$ (possible since $C$ is unbounded). Define $G': [\lambda]^{<\aleph_0} \to P_{\lambda \mu \theta}$ via $G'(w) = G(F'(w))$. By hypothesis there is some $t \in [\lambda]^k$ and some $(w'_s: s \in [t]^{k-1})$ such that each $w'_s \cap t = F'(w'_s) \cap t = s$, and such that $\bigcup \{G'(w'_s): s \in [t]^{k-1}\}$ is a function. Write $w_s = F'(w'_s)$; then $(w_s: s \in [t]^{k-1})$ is as desired.
	
	Next, suppose $\lambda$ is not big enough for $(\mu, \theta, k)$, as witnessed by $F', G'$. Let $C \subseteq [\lambda]^{<\theta}$ be the club set of all $w$ with $F'[[w]^{<\aleph_0}] \subseteq w$. As in the proof of Theorem~\ref{CombThm2}, we can choose some $G: C \to P_{\lambda \mu \theta}$ so that for all $w, w' \in C$, if $G(w)$ and $G(w')$ are compatible, then:
	
	\begin{itemize}
		\item[(I)]  $w$ and $w'$ have the same order-type; let $\phi: w \to w'$ be the unique order-preserving bijection;
		\item[(II)] $\phi$ is the identity on $w \cap w'$;
		\item[(III)] For all $v \in [w]^{<\aleph_0}$, $G'(v)$ and $G'(\phi[v])$ are compatible.
	\end{itemize}
	
	Suppose $t \in [\lambda]^k$, and suppose $(w_s: s \in [t]^{k-1})$ is a sequence from $C$ with each $w_s \cap t = s$.
	
	For each $s \not= s' \in [t]^{k-1}$, let $\phi_{s', s}: w_{s'} \to w_{s}$ be the unique order-preserving isomorphism. As in Theorem~\ref{CombThm2}, $s \cup \phi_{s', s}[s']$ only depends on $s$. For each $s \in [t]^{k-1}$ write $v_s = s \cup \phi_{s', s}[s']$, for some or any $s' \not= s$ with $s' \in [t]^{k-1}$. Then $s \subseteq v_s \in [w_s]^{k}$, and $\phi_{s, s'}(v_s) = v_{s'}$. Now $\bigcup\{G'(v_s): s \in [t]^{k-1}\}$ cannot be a function, by choice of $F', G'$; thus by condition (III) in the definition of $G$, $\bigcup\{G(w_s): s \in [t]^{k-1}\}$ cannot be a function.
\end{proof}

\section{Independent Systems of Models} \label{KeislerNewIndSys}

We define what we mean by independent systems of sets and models, quote some results from \cite{SP} (joint with Shelah), and define what it means for a simple theory to have $\mathcal{P}^-(k)$-amalgamation of models. Finally, we recall the notion of $<k$-type amalgamation from \cite{SP}, and show that $\mathcal{P}^-(k)$-amalgamation of models implies $<k$-type amalgamation.

The following definition is similar to the definition of stable system in Shelah \cite{ShelahIso} for stable theories, see Section XII.2. In fact we are modeling our definition after Fact 2.5 there (we cannot take the definition from \cite{ShelahIso} because we allow $P$ to contain infinite subsets of $I$). The terminology ``non-forking diagrams" is used in \cite{SP}, but we prefer ``independent systems" to align with \cite{GenAmalg}. Typically we will deal with the case where $\Delta$ is closed under subsets, i.e. is a pattern.

Alert: we emphatically do not adopt the convention of always working in the monster model $\mathfrak{C}$. This is important because systems of submodels of $\mathfrak{C}$ can always be amalgamated.

\begin{definition}
	
	Let $T$ be simple. 
	
	Suppose $\Delta \subseteq \mathcal{P}(I)$ is closed under finite intersections, and suppose $M \models T$. Say that $(A_s: s \in \Delta)$ is a system of subsets of $M$  if each $A_s \subseteq M$ and $s \subseteq t$ implies $A_s \subseteq A_t$, and each $A_s \cap A_t = A_{s \cap t}$. Say that $(A_s: s \in \Delta)$ is an independent system if for all $s_i: i < n, t \in \Delta$, $\bigcup_{i < n} A_{s_i}$ is free from $A_t$ over $\bigcup_{i < n} A_{s_i \cap t}$. If each $A_s$ is an elementary submodel of $M$, we say that $(A_s: s \in \Delta)$ is a system of submodels of $M$. 
	
	Say that $(M_s: s \in \Delta)$ is a system of models if each $M_s \models T$ and for each $s \in I$, $(M_t: t \in \Delta, t \subseteq s)$ is a system of submodels of $M_s$, and for all $s,t \in \Delta$, $M_s \cap M_t = M_{s \cap t}$. Say that $(M_s: s \in \Delta)$ is independent if for each $s \in P$, $(M_t: t \in \Delta, t \subseteq s)$ is independent. Finally, say that $M$ is a solution to $(M_s: s \in \Delta)$ if $M$ is a model of $T$ and $(M_s: s \in \Delta)$ is an independent system of submodels of $M$.
\end{definition}

The following gives some alternative equivalences of our definition of independence. It is Lemma 4.2 of \cite{SP}.

\begin{lemma}\label{EquivsForInd}
	Suppose $(A_s: s \in \Delta)$ is a system of subsets of $M$, where $\Delta \subseteq \mathcal{P}(I)$ is closed under finite intersections. Then the following are equivalent:
	
	\begin{itemize}
		\item[(A)] For all downward-closed subsets $\Delta_0, \Delta_1 \subseteq \Delta$, $\bigcup_{s \in \Delta_0} A_s$ is free from $\bigcup_{s \in \Delta_2} A_s$ over $\bigcup_{s \in \Delta_0 \cap \Delta_1} A_s$.
		\item[(B)] For all $s_i: i < n, t_j: j < m$ from $\Delta$, $\bigcup_{i < n} A_{s_i}$ is free from $\bigcup_{j < m} A_{t_j}$ over $\bigcup_{i < n, j < m} A_{s_i \cap t_j}$.
		\item[(C)] $(A_s: s \in \Delta)$ is independent.
	\end{itemize}
\end{lemma}

We will need the following Lemma.
\begin{lemma}\label{ExtendingInd}
Suppose $I$ is a finite index set and $\Delta \subseteq \mathcal{P}(I)$ is closed under subsets. Suppose $(A_s: s \in \mathcal{P}(I))$ is an independent system of subsets of $\mathfrak{C}$, such that $A_s \preceq \mathfrak{C}$ for all $s \in \Delta$. Then we can find an independent system of submodels $(M_s:s \in \mathcal{P}(I))$ of $\mathfrak{C}$ with each $A_s \subseteq M_s$, and moreover $A_s = M_s$ for each $s \in \Delta$.
    
\end{lemma}
\begin{proof}
We prove this by induction on $k := |\mathcal{P}(I) \backslash \Delta|$. When $k = 0$ there is nothing to do.

So suppose we have verified the lemma for all $\Delta'$ with $|\mathcal{P}(I) \backslash \Delta'| \leq k$, and $(\Delta, A_s: s \subseteq I)$ are given with $|\mathcal{P}(I) \backslash \Delta| = k+1$.

Choose $s_* \subseteq I$ such that $s_* \not \in \Delta$ and $|s_*|$ is minimal subject to this. 
Then $s \in \Delta$ for all $s \subsetneq s_*$. Choose $A'_{s_*} \preceq \mathfrak{C}$ with $A_{s_*} \subseteq A'_{s_*}$ and $A'_{s_*} \forkindep_{A_{s_*}} A_I$. Define a system $(A'_s: s \subseteq I)$ of subsets of $\mathfrak{C}$ by putting $A'_s = A_s$ unless $s \supseteq s_*$, in which case $A'_{s} = A'_{s_*} \cup A_s$.

Let $\Delta' = \Delta \cup \{s_*\}$. Then $|\mathcal{P}(I) \backslash \Delta'| = k-1$. Once we show $(A'_s: s \subseteq I)$ is independent, then $(\Delta', A'_s: s \subseteq I)$ is as in the statement of the lemma,  so by the inductive hypothesis we get an independent system of models $(M_s: s \subseteq I)$ as there, and this system clearly also works for $(\Delta, A_s: s \subseteq I)$, continuing the induction.

So it suffices to show $(A'_s: s \subseteq I)$ is independent. Suppose $s_i: i < n, t_j: j < m$ are subsets of $I$. We need to show $\bigcup_{i < n} A'_{s_i}$ is free from $\bigcup_{j < m} A'_{t_j}$ over $\bigcup_{i < n, j < m} A'_{s_i \cap t_j}$. 

We know $\bigcup_{i < n} A_{s_i}$ is free from $\bigcup_{j < m} A_{t_j}$ over $\bigcup_{i < n, j < m} A_{s_i \cap t_j}$. There are now some cases. For brevity write $A = \bigcup_{i < n} A_{s_i}$, write $B = \bigcup_{j < m} A_{t_j}$ and write $C = A \cap B$, so we know $A \forkindep_C B$.

If each $s_i, t_j \not \supseteq s_*$ then there is nothing to do.

Suppose there is some $i < n$ with $s_i \supseteq s_*$, but no such $j < m$. Then we are trying to show $A'_{s_*}A \forkindep_C B$. This follows from $A_{s_*} \subseteq A$ and $A'_{s_*} \forkindep_{A_{s_*}} AB$.

If there is some $j < m$ with $s_j \supseteq s_*$ but no such $i < n$, then the situation is symmetric to the preceding case.

Finally suppose there is some $i < n$ with $s_i \supseteq s_*$ and there is some $j < m$ with $t_j \supseteq s_*$. Then we are trying to show $A \forkindep_{A'_{s_*} C} B$. This follows from $A_{s_*} \subseteq C$ and $A'_{s_*} \forkindep_{A_{s_*}} AB$.
\end{proof}

Many of the arguments of Malliaris and Shelah concerning complete Boolean algebras $\mathcal{B}$ \cite{DividingLine},  \cite{Optimals} \cite{InfManyClass} implicitly take place in the generic extension by $\mathcal{B}$. Specifically, ``collision detection" is the filtrations of models in the generic extension $\mathbb{V}[G]$. The following thus captures these arguments, with significant notational savings; it is Theorem 4.4 of \cite{SP}.

\begin{theorem}\label{IndependentSystemsThm}
	Suppose $T$ is a simple theory in a countable language, and suppose ${M_*} \models T$ is sufficiently saturated, and suppose ${A_*} \subseteq {M_*}$ has size at most $\lambda$. Then we can find an independent system  of countable submodels $(M_s: s \in \mathcal{P}(\lambda))$ of ${M_*}$, with ${A_*} \subseteq \bigcup_s M_s$, which satisfies the following:
	
	\begin{itemize}
		\item[(I)] For each $s \subseteq \lambda$, $M_s = \bigcup\{M_t: t \in [s]^{\leq \aleph_0}\}$;
		\item[(II)] For each $s \subseteq \lambda$, $|M_s| = |s| + \aleph_0$;
		\item[(III)] For each $X \subseteq \mathcal{P}(\lambda)$, $M_{\bigcap X} = \bigcap\{M_s: s \in X\}$.
	\end{itemize}
\end{theorem}

We now describe some amalgamation properties that a theory can have.

\begin{definition}

%
%
		
	Suppose $\Delta \subseteq \mathcal{P}(I)$ is closed under intersections, and $T$ is a countable simple theory. Then $T$ has $\Delta$-amalgamation of models if every independent system of models $(M_s: s \in \Delta)$ of $T$ has a solution.
	
\end{definition}

We collect together some straightforward remarks. 	If $X$ is a set then let $\mathcal{P}^-(X)$ be the set of proper subsets of $X$. In particular, if $n  < \omega$ then $\mathcal{P}^-(n)$ is the set of all proper subsets of $n$; note $n = \{0, 1, \ldots, n-1\}$.

\begin{remark}
	\begin{itemize}
	\item If $k' < k$ and $T$ has $\mathcal{P}^-(k)$-amalgamation of models, then $T$ has $\mathcal{P}^-(k')$-amalgamation of models.
	\item If $\Delta \subseteq \mathcal{P}(I)$ is countable, then $T$ has $\Delta$-amalgamation of models if and only if every independent system of countable models $(M_s: s \in \Delta)$ has a solution.
		
	\item Every simple theory has $\mathcal{P}^-(3)$-amalgamation of models.
	\item It follows from Conclusion XII.2.12 of \cite{ShelahIso} that every stable theory has $\Delta$-amalgamation of models, whenever $\Delta \subseteq [I]^{<\aleph_0}$ is closed under subsets.

	\item $T_{rg}$, the theory of the random graph, has $\Delta$-amalgamation of models, whenever $\Delta \subseteq \mathcal{P}(I)$ is closed under finite intersections.
	\item For each $3 \leq k < n$, $T_{n, k}$ has $\mathcal{P}^-(k)$-amalgamation of models, but fails $\mathcal{P}^-(k+1)$-amalgamation of models.

	\end{itemize}
\end{remark}

Now we discuss type amalgamation; this was first introduced in joint work with Shelah \cite{SP}. It is hand-tailored to the $\leq_{SP}$-ordering, and in fact this amounts to being hand-tailored to the Keisler ordering as well. Given natural numbers $n, m$, we write $\,^n m$ for the set of functions from $n$ to $m$, to avoid ambiguity with exponentiation $m^n$.

\begin{definition}\label{DefOfTypeAP}
	Given $\Lambda \subseteq \,^n m$, let $\Delta_{\Lambda}$ be the set of all partial functions from $n$ to $m$ which can be extended to an element of $\Lambda$; so $\Delta_{\Lambda} \subseteq \mathcal{P}(n \times m)$ is closed under subsets, and $\Lambda$ is the set of maximal elements of $\Delta_{\lambda}$. 
	
	By a $\Lambda$-array, we mean an independent system $(N_s: s \in \Delta_{\Lambda})$ of submodels of $\mathfrak{C}$ together with maps $(\pi_{\eta, \eta'}: \eta, \eta' \in \Lambda)$ such that:
	\begin{itemize} 
		\item Each $\pi_{\eta \eta'}: N_{\eta} \to N_{\eta'}$ is an isomorphism,
		\item For all $\eta, \eta', \eta''$, $\pi_{\eta',\eta''} \circ \pi_{\eta, \eta'} = \pi_{\eta, \eta''}$;
		\item For all $\eta, \eta'$, if we put $u = \{i < n: \eta(i) = \eta(i')\}$, and if we put $s = \eta \restriction_u = \eta' \restriction_u$, then $\pi_{\eta, \eta'} \restriction_{N_s}$ is the identity.
	\end{itemize}
	
	Suppose $(\overline{N}, \overline{\pi})$ is a $\Lambda$-array and $\overline{x}$ is a finite tuple of variables. Then say that $\overline{p}(\overline{x})= (p_\eta(\overline{x}): \eta \in \Lambda)$ is a coherent system of types over $(\overline{N}, \overline{\pi})$ if each $p_\eta(\overline{x})$ is a type over $N_\eta$ which does not fork over $N_0$, and each $\pi_{\eta \eta'}[p_\eta(\overline{x})] = p_{\eta'}(\overline{x})$.
\end{definition}

\begin{definition}
	Suppose $\Lambda \subseteq \,^n m$. Then $T$ has $\Lambda$-type amalgamation if, whenever $(N_s: s \in \Delta_\Lambda), (\pi_{\eta, \eta'}: \eta, \eta' \in \Lambda)$ is a  $\Lambda$-array, and $(p_\eta(\overline{x}): \eta \in \Lambda)$ is a coherent system of types over $(\overline{N}, \overline{\pi})$, then $\bigcup_{\eta \in \Lambda}p_\eta(\overline{x})$ does not fork over $N_0$ (as computed in $\mathfrak{C}$).
	
	Let $\mathbf{\Lambda}$ be the set of all $\Lambda \subseteq \,^n m$, for varying $n, m < \omega$. Say that $T$ has $<k$-type amalgamation if for all $\Lambda \in \mathbf{\Lambda}$ with $|\Lambda| < k$, $T$ has $\Lambda$-type amalgamation.
\end{definition}

The role of the isomorphisms $\pi_{\eta, \eta'}$ is to exclude certain trivial failures of type amalgamation. We have the following example from \cite{GenAmalg}, for instance: work in $T_{rg}^{eq}$, the elimination of imaginaries of the random graph. Suppose $x$ is a variable corresponding to unordered pairs from the home sort; write $x = \{y, z\}$, so types in the variable $x$ correspond to symmetric types in the variables $y, z$ extending $y \not= z$. To show $T$ has $\Lambda$-type amalgamation, we need to show that there is no $\Lambda$-array $(\overline{N}, \overline{\pi})$ such that there are $a_0, a_1, a_2$ $\mathfrak{C}$ for which $\bigcup_{\eta \in \Lambda} p_\eta(x)$ asserts the following: 

\begin{itemize}
	\item For each $i < 2$, $y R a_i$ if and only if $\lnot z R a_i$;
	\item For all $i < j < 2$, $y R a_i$ if and only if $\lnot y R a_j$, and $z R a_i$ if and only if $\lnot z R a_j$.
\end{itemize}
But this cannot happen: otherwise, for each $i < j < 3$ we can find $\eta_{ij}\in \Lambda$ with $\{a_i, a_j\} \subseteq N_{\eta_{ij}}$. Write $b_0 = a_0$ and write $b_1 = a_1$. Also, write $b_2 = \pi_{\eta_{02} \,\eta_{01}}(a_2) = \pi_{\eta_{12} \,\eta_{01}}(a_2)$. Then since $\pi_{\eta \eta'}[p_\eta(x)] = p_{\eta'}(x)$ (applied to $\eta' = \eta_{01}$ and $\eta \in \{\eta_{02}, \eta_{12}\}$), we get that $p_{\eta_{01}}(x)$ implies the following: 

\begin{itemize}
	\item For each $i < 2$, $y R b_i$ if and only if $\lnot z R b_i$;
	\item For all $i < j < 2$, $y R b_i$ if and only if $\lnot y R b_j$, and $z R b_i$ if and only if $\lnot z R b_j$.
\end{itemize}

But then $p_{\eta_{01}}(x)$ is inconsistent, contradiction.

The following lemma is straightforward.
\begin{lemma}
	Suppose $\Lambda \subseteq \,^n m$. Then in the definition of $\Lambda$-type amalgamation, the following changes would not matter:
	\begin{itemize}
		\item[(A)] We could restrict to just countable models $N_s$.
		\item[(B)] We could allow $p_\eta(\overline{x})$ to be any partial type, or insist it is a single formula. Also, we could allow $\overline{x}$ to be of arbitrary length.
	\end{itemize}
\end{lemma}

\begin{example}
Every simple theory has $<3$-type amalgamation. $T_{rg}$ has $<\aleph_0$-type amalgamation, as does every stable theory.
\end{example}

More generally, we have the following:

\begin{theorem}\label{GenAmalgImpliesTypeAmalg}
Suppose $T$ is a countable simple theory and $3 \leq k \leq \aleph_0$. If $T$ has $\mathcal{P}^-(k)$-amalgamation of models, then $T$ has $<k$-type amalgamation.
\end{theorem}
\begin{proof}
It suffices to consider the case $k < \aleph_0$, since if $T$ has $\mathcal{P}^-(\aleph_0)$-amalgamation of models, then $T$ has $\mathcal{P}^-(k)$-amalgamation of models for all $k < \aleph_0$.

Suppose $\Lambda \subseteq \,^n m$ has $|\Lambda| \leq k-1$, and suppose $(N_s: s \in \Delta_\Lambda), (\pi_{\eta, \eta'}: \eta, \eta' \in \Lambda)$ is a $\Lambda$-array, and $(p_\eta(\overline{x}): \eta \in \Lambda)$ is a coherent system of types over $(\overline{N}, \overline{\pi})$. 

We claim we can extend $(N_s: s \in \Delta_\Lambda)$ to an independent system of submodels $(N_s: s \in\mathcal{P}(n \times m))$ of $\mathfrak{C}$. Indeed, define $A_s$ for $s \subseteq n \times m$ via: $A_s = \bigcup \{N_t: t \in \Delta_{\Lambda}, t \subseteq s\}$. Then $(A_s: s \subseteq n \times m)$ is an independent system of subsets of $\mathfrak{C}$ with $A_s = N_s$ for $s \in \Delta_{\Lambda}$; apply Lemma~\ref{ExtendingInd}.

Fix some $\eta_* \in \Lambda$ (if $\Lambda = \emptyset$ then there is nothing to do). Let $*$ be some new index element. Let $\overline{a}$ be a realization of $p_{\eta_*}(\overline{x})$. For each $s \subseteq \eta_* \cup \{*\}$ let $A_s$ be defined as follows: $A_s = N_s$ if $* \not \in s$, and otherwise $A_s = \overline{a} N_s$. Then $(A_s:s \subseteq \eta_* \cup \{*\})$ is an independent system of subsets of $\mathfrak{C}$ with $A_s = N_s$ for $s \subseteq \eta_*$. Apply Lemma~\ref{ExtendingInd} to get an independent system of models $(N_s: s \subseteq \eta_* \cup \{*\})$ extending $(N_s: s \subseteq \eta_*)$. After relabeling we can suppose $N_{\eta_* \cup \{*\}} \cap \mathfrak{C} = N_{\eta_*}$. This relabeling preserves that there is some $\overline{a} \in N_{\{*\}}$ realizing $p_{\eta_*}(\overline{x})$.

Let $\Delta \subseteq \mathcal{P}(n \times m \cup \{*\})$ be the closure of $\{\eta \cup \{*\}: \eta \in \Lambda\} \cup \{n \times m\}$ under $\subseteq$. Using the isomorphisms $\pi_{\eta_* \eta}: \eta \in \Lambda$, we can copy $(N_s: s \subseteq \eta_* \cup \{*\})$ over each $\eta \in \Lambda$ to get an independent system of models $(N_s: s \in \Delta)$, such that for each $\eta \in \Lambda$, $\overline{a}$ realizes $p_\eta(\overline{x})$ in $N_{\eta \cup \{*\}}$. Thus, to finish, it would suffice to show that $(N_s: s \in \Delta)$ has a solution $N$, since then $\overline{a}$ would realize $\bigcup_{\eta \in \Lambda} p_\eta(\overline{x})$ in $N$, and by independence, $\overline{a}$ would be free from $N_{n \times m}$ over $N_\emptyset$. 

Actually, we don't quite know how to show that $(N_s: s \in \Delta)$ has a solution; we will only be able to show that a coarser independent system has a solution. Namely, enumerate $\{\eta \cup \{*\}: \eta \in \Lambda\} \cup \{n \times m\} = \{s_u: u \in [k]^{k-1}\}$, with repetitions if necessary. We are using that $[k]^{k-1}$ has $k$ elements, and $|\Lambda| < k$. For each $u \in \mathcal{P}^-(k)$, define $f(u) = \bigcap\{s_v: v \in [k]^{k-1}, u \subseteq v\}$. Note that $f(u) \cap f(u') = f(u \cap u')$ for all $u, u' \in \mathcal{P}^-(k)$. Define $M_u = N_{f(u)}$, for each $u \in \mathcal{P}^-(k)$; it is straightforward to check that $(M_u: u \in \mathcal{P}^-(k))$ is an independent system of models. By hypothesis, it has a solution $M_k$.

Now $\overline{a}$ realizes $\bigcup_{\eta \in \Lambda} p_\eta(\overline{x})$ in $M_k$. We  need to show that $\overline{a}$ is free from $N_{n \times m}$ over $N_\emptyset$ in $M_k$. After reindexing, we can suppose $i_* < k$ is such that for all $i < k$, $s_{k \backslash \{i\}} = n \times m$ if and only if $i < i_*$. Note then that $\overline{a} \in M_{i_*}$ and $M_{k \backslash i_*} = N_{n \times m}$. By independence of $(M_u: u \in \mathcal{P}(k))$, $M_{i_*}$ is free from $M_{k \backslash i_*}$ over $M_\emptyset$, and so $\overline{a}$ is free from $N_{n \times m}$ over $M_{\emptyset}$. Let $\eta \in \Lambda$; note that $M_{\emptyset} \subseteq N_{\eta \cup \{*\}} \cap N_{n \times m} = N_{\eta}$. Since $p_\eta(\overline{x})$ does not fork over $N_\emptyset$, we get that $\overline{a}$ is free from $M_{\emptyset}$ over $N_{\emptyset}$. By transitivity of independence, we get that $\overline{a}$ is free from $N_{n \times m}$ over $N_{\emptyset}$, as desired.
\end{proof}

Conjecturally, the converse to the above theorem is true as well.

\section{Coloring Partial Types}\label{ColoringSec}

In this section, we shall given some set-theoretic consequences of amalgamation.

First, some definitions:

\begin{definition}
Suppose $k \leq \aleph_0$ and $\mathcal{H} \subseteq [X]^{<k}$ is a hypergraph. Then let $\chi(\mathcal{H})$, the chromatic number of $\mathcal{H}$, be the least cardinal $\mu$ such that there is some $\mu$-coloring of $\mathcal{H}$, i.e. some coloring $c: X \to \mu$ such that for all $\alpha < \mu$, $\mathcal{H} \cap [c^{-1}(\alpha)]^{<k} = \emptyset$.

Suppose $P$ is a partial order; for each $3 \leq k \leq \aleph_0$, let $\chi(P, k)$, the $k$-ary chromatic number of $P$, be the least cardinal $\mu$ such that there is some $(\mu, k)$-coloring of $P$, i.e. some function $c: P \to \mu$ such that for all $u \in [P]^{<k}$, if $c \restriction_u$ is constant then $u$ has a lower bound in $P$. In other words, $\chi(P, k)$ is the chromatic number of $\mathcal{H} := \{u \in [P]^{<k}: u \mbox{ has no lower bound in } P\}$. 

Suppose $\theta$ is a regular cardinal, $T$ is a countable simple theory, $M \models T$ and $M_0 \preceq M$ is countable. Then let $\Gamma^{\theta}_{M, M_0}$ be the poset of all partial types $p(x)$ over $M$ of cardinality less than $\theta$, which do not fork over $M_0$; we order $\Gamma^{\theta}_{M, M_0}$ by reverse inclusion. Really $x$ is allowed to be a finite tuple here, but we suppress this (alternatively, we should pass to $T^{eq}$).

Suppose $\theta$ is a regular cardinal, $\mu = \mu^{<\theta}$, $\lambda \leq 2^\mu$, and $3 \leq k \leq \aleph_0$. Then say that $T$ has the $(\lambda, \mu, \theta, k)$-coloring property if whenever $M \models T$ with $|M| \leq \lambda$ and whenever $M_0 \preceq M$ is countable, then $\chi(\Gamma^{\theta}_{M, M_0}, k) \leq \mu$. 

\end{definition}

\noindent \textbf{Question.} For which $\lambda, \mu, \theta, k$ does $T$ have the $(\lambda, \mu, \theta, k)$-coloring property?

\vspace{1 mm}

When $\lambda > 2^\mu$, we should tweak our definition of coloring property. Towards this:
\begin{definition}
Given sets $X, Y$ and a regular cardinal $\theta$, let $P_{X Y \theta}$ be the poset of all partial functions from $X$ to $\mu$ of cardinality less than $\theta$, ordered by reverse inclusions.
	
Suppose $P$ and $R$ are forcing notions, and $k \leq 3 \leq \aleph_0$ is a cardinal. Then say that $F: P \to R$ is an $(R, k)$-coloring, and that $P$ is $(R, k)$-colorable, if whenever $s \in [P]^{<k}$, if $F[s]$ has a lower bound in $R$, then $s$ has a lower bound in $P$.

\end{definition}

One can view colorability as a generalization of chromatic numbers. Specifically, for each cardinal $\mu$ let $(\mu, \emptyset)$ be the partial order on $\mu$ in which all elements of $\mu$ are incomparable. Then for every partial order $P$ and for every cardinal $3 \leq k \leq \aleph_0$, $\chi(P, k) \leq \mu$ if and only if $P$ is $((\mu, \emptyset), k)$-colorable.

Some simple facts:
\begin{lemma}\label{ObviousLemmaskColorings}
	Suppose $\theta$ is a regular cardinal, and $3 \leq k \leq \aleph_0$. Then for all forcing notions $P, R$, if $P$ is $(R, k)$-colorable, then so is every dense subset of $\mathcal{B}(P)_+$. Also, if $P$ is $(R, k)$-colorable and $R$ is $(R', k)$-colorable, then $P$ is $(R', k)$-colorable.
\end{lemma}

The following theorem is closely related to the classical Hewitt-Marczewski-Pondiczery theorem of topology; it is proved by Engleking and Karlowicz in \cite{Partitions}.

\begin{theorem}\label{ER}
	Suppose $\theta \leq \mu \leq \lambda$ are infinite cardinals such that $\theta$ is regular, $\mu = \mu^{<\theta}$, and $\lambda \leq 2^\mu$. Then there is a sequence $(\mathbf{f}_\gamma: \gamma < \mu)$ from $\,^\lambda \mu$ such that for all partial functions $f$ from $\lambda$ to $\mu$ of cardinality less than $\theta$, there is some $\gamma < \mu$ such that $\mathbf{f}_\gamma$ extends $f$. 
\end{theorem}

Reformulating:

\begin{corollary}\label{ERCor}
	Suppose $\theta \leq \mu$ are infinite cardinals such that $\theta$ is regular and $\mu = \mu^{<\theta}$. Then for every set $X$ with $|X| \leq 2^\mu$,  $\chi(P_{X\mu\theta}, \aleph_0) \leq \mu$. Hence, for every partial order $P$ with $|P| \leq 2^\mu$ and for every $3 \leq k \leq \aleph_0$, $\chi(P, k) \leq \mu$ if and only if $P$ is $(P_{X \mu\theta}, k)$-colorable for some set $X$ (since then we can arrange $|X| \leq 2^\mu$).
\end{corollary}

We remark that in practice, all colorings of $\Gamma^{\theta}_{M, M_0}$ naturally take values in some $P_{X \mu \theta}$, and then we apply Corollary~\ref{ERCor} to get a coloring in $\mu$. 

\begin{definition}
Suppose $\theta$ is a regular cardinal, $\mu = \mu^{<\theta}$, $3 \leq k \leq \aleph_0$, $\lambda$ is arbitrary, and $T$ is a countable simple theory. Then say that $T$ has the $(\lambda, \mu, \theta, k)$-coloring property if for all $M \models T$ with $|M| \leq \lambda$ and for all $M_0 \preceq M$ countable, $\Gamma^{\theta}_{M, M_0}$ is $(P_{X \mu \theta}, k)$-colorable for some set $X$.
\end{definition}

\begin{remark}
	
	We differ notation slightly from \cite{SP}. There, $SP^1_T(\lambda, \mu, \theta)$ means that for every $M \models T$ with $|M| \leq \lambda$ and for every $M_0 \subseteq M$ countable, $\chi(\Gamma^\theta_{M, M_0}, \aleph_0) \geq \mu$ (so when $\lambda \leq 2^\mu$, this is equivalent to $T$ having the $(\lambda, \mu, \theta, \aleph_0)$-coloring property). Also, what we are calling the $(\lambda, \mu, \theta, k)$-coloring property of theories is called the $(<k, \lambda, \mu, \theta)$-amalgamation property in \cite{SP}. Finally, in \cite{SP}, a forcing notion $P$ is defined to have the $(<k, \mu, \theta)$-amalgamation property if it is $(P_{X \mu \theta}, k)$-colorable for some $X$.
\end{remark}

\begin{remark}
Suppose $\theta$ is regular and $\mu = \mu^{<\theta}$ and $T$ is a countable simple theory.

If $\lambda \leq \mu$ then $T$ has the $(\lambda, \mu,\theta, k)$-coloring property.

If $T$ is stable then $T$ has the $(\lambda, \mu, \theta, \aleph_0)$-coloring property for all $\lambda$. 

If $T$ is unstable and $\lambda > 2^\mu$, then there is some $M \models T$ with $|M| \leq \lambda$ and some $M_0 \preceq M$ with $M_0$ countable, such that $\chi(\Gamma^\theta_{M, M_0},3) > \mu$. 
\end{remark}
\begin{proof}
	This is essentially worked out in \cite{SP} (joint with Shelah) but the notation is quite different. The first claim is clear.
	
	Suppose next that $T$ is stable. We break into cases depending on the value of $\theta$; first suppose $\theta > \aleph_0$. Note then that $\mu \geq 2^{\aleph_0}$. Let $M \models T$ and let $M_0 \preceq M$ be countable. We show that $\chi(\Gamma^\theta_{M, M_0}, \aleph_0) \leq 2^{\aleph_0}$, which suffices. Indeed, let $(p_\alpha(x): \alpha < 2^{\aleph_0})$ list all types over $M$ which do not fork over $M_0$, and given $p(x) \in \Gamma^\theta_{M, M_0}$ let $c(p(x))$ be the least $\alpha < 2^{\aleph_0}$ with $p(x) \subseteq p_\alpha(x)$. If on the other hand $\theta = \aleph_0$, then a similar argument applies, using that for every finite set $\Delta$ of partitioned formulas $\phi(x, y)$, there are only countably many complete $\Delta$-types over $M$ which do not fork over $M_0$.
	
	Finally, suppose $T$ is unstable and $\lambda > 2^\mu$. Let $\phi(x, \overline{y})$ have the independence property, and choose $M \models T$ with $|M| = \lambda$ such that there is some countable $M_0 \preceq M$ and some Morley sequence $(\overline{b}_\alpha: \alpha < \lambda)$ over $M_0$ in $M$ witnessing the independence of $\phi$. For every $f \in P_{\lambda 2 \theta}$, we get a corresponding partial type $p_f(x)$ over $M$ of size less than $\theta$, namely $p_f(x) = \{\phi(x, \overline{b}_\alpha): \alpha \in \mbox{dom}(f) \mbox{ and } f(\alpha) = 0\}  \cup \{\lnot \phi(x, \overline{b}_\alpha): \alpha \in \mbox{dom}(f) \mbox{ and } f(\alpha) =1\}$. Note that each $p_f(x)$ does not fork over $M_0$, since $(\overline{b}_\alpha: \alpha < \lambda)$ is a Morley sequence over $M_0$.  Hence $p_f(x) \in \Gamma^\theta_{M, M_0}$.
	
	Then $f \mapsto p_f(x)$ witnesses that $P_{\lambda 2 \theta}$ is $(\Gamma^\theta_{M, M_0}, 3)$-colorable. So it suffices to note that $\chi(P_{\lambda 2 \theta}, 3) > \mu$. Suppose towards a contradiction that $c: P_{\lambda 2 \theta} \to \mu$ is a $(\mu, 3)$-coloring. For each $\alpha < \mu$, extend $c^{-1}(\{\alpha\})$ to a function $f_\alpha: \lambda \to \theta$. For each $\gamma < \lambda$, let $I_\gamma = \{\alpha < \mu: f_\alpha(\gamma) = 1\}$. Then $\gamma \mapsto I_\gamma$ is an injection from $\lambda$ to $\mathcal{P}(\mu)$, contradicting $\lambda > 2^\mu$.
\end{proof}

\section{Consequences of $\mathcal{P}^-(k)$-amalgamation of models}\label{ConsequencesSec}

In this section, we deduce several consequences of $\mathcal{P}^-(k)$-amalgamation of models, all of which conjecturally reverse. Note that we have already proved that $\mathcal{P}^-(k)$ amalgamation of models implies $<k$-type amalgamation.

The following is Theorem 5.6 of \cite{SP} (joint with Shelah).  Note that the case $k = 3$ proves that for every simple theory $T$, for every regular uncountable cardinal $\theta$, for every $\mu = \mu^{<\theta}$ and for every cardinal $\lambda$, $T$ has the $(\lambda, \theta, 3)$-coloring property. (Actually, in \cite{SP}, only the case $\mu = \theta = \theta^{<\theta}$ is considered, but the general case is the same.)

\begin{theorem}\label{SatPortion}
	Suppose $\theta$ is a regular uncountable cardinal, $\mu = \mu^{<\theta}$,  $3 \leq k \leq \aleph_0$, and $T$ is a countable simple theory. Suppose $T$ has $<k$-type amalgamation. Then $T$ has the $(\lambda, \mu, \theta, k)$-coloring property for all $\lambda$.
\end{theorem}
\begin{proof}
	For the reader's convenience, we repeat the proof from \cite{SP}.

Suppose $M \models T$ has $|M| \leq \lambda$, and suppose $M_0 \preceq M$ is countable. We wish to show that $\Gamma^{\theta}_{M, M_0}$ is $(P_{X \mu \theta}, k)$-colorable for some set $X$.

After applying  Theorem~\ref{IndependentSystemsThm} and possibly increasing $M$ and $M_0$, we can suppose $M = \bigcup_{s \in \mathcal{P}(\lambda)} M_s$, where $(M_s: s \in \mathcal{P}(\lambda))$ is an independent system of models satisfying the three conditions of Theorem~\ref{IndependentSystemsThm}. (We will need to increase $M_0$ so that $M_0 = M_{\emptyset}$.) Let $<_*$ be a well-ordering of $M$.

Given $A \in [M]^{<\theta}$ let $s_A$ be the $\subseteq$-minimal $s \in [\lambda]^{\leq\aleph_0}$ with $A \subseteq M_{s_A}$, which exists by condition (III) on $(M_s:s \in \mathcal{P}(\lambda))$.

Let $P$ be the set of all $p(x) \in \Gamma^{\theta}_{M, M_0}$ such that for some $s \in [\lambda]^{<\theta}$, $p(x)$ is a complete type over $M_s$; we write $p(x, M_s)$ to indicate this. $P$ is dense in $\Gamma^{\theta}_{M, M_0}$, so it suffices to show that $P$ is $(P_{X \mu \theta}, k)$-colorable for some set $X$.

Choose $X$ large enough, and $F: P \to P_{X \mu \theta}$ so that if $F(p(x, M_s))$ is compatible with $F(q(x, M_t))$, then:

\begin{itemize}
	\item $s$ and $t$ have the same order-type, and if we let $\rho: s \to t$ be the unique order-preserving bijection, then $\rho$ is the identity on $s\cap t$;
	\item $M_s$ and $M_t$ have the same $<_*$-order-type, and the unique $<_*$-preserving bijection from $M_s$ to $M_t$ is in fact an isomorphism $\tau: M_s \cong M_t$;
	\item For each finite $\overline{a} \in M_{s}^{<\omega}$, if we write $s' = s_{\overline{a}}$ and if we write $t' = s_{\tau(\overline{a})}$, then: $\rho[s'] = t'$ and $\tau \restriction_{M_{s'}}: M_{s'} \cong M_{t'}$.
	\item $\tau[p(x)] = q(x)$.
\end{itemize}

This is not hard to do. Note that it follows that for every $s' \subseteq s$, $\tau \restriction_{M_{s'}}: M_{s'} \cong M_{\rho[s']}$, since $M_{s'} = \bigcup\{M_{s_{\overline{a}}}: \overline{a} \in (M_{s'})^{<\omega}\}$ and similarly for $M_{t'}$.

We claim that $F$ works.

So suppose $p_i(x, M_{s_i}): i < i_*$ is a sequence from $P$ for $i_* < k$, such that $(F(p_i(x)): i < i_*)$ is compatible in $P_{X \mu \theta}$.

Let $\gamma_*$ be the order-type of some or any $s_i$. Enumerate each $s_i = \{\alpha_{i, \gamma}: \gamma < \gamma_*\}$ in increasing order.  Let $E$ be the equivalence relation on $\gamma_*$ defined by: $\gamma E \gamma'$ iff for all $i, i' < k$, $\alpha_{i, \gamma} = \alpha_{i', \gamma}$ iff $\alpha_{i, \gamma'} = \alpha_{i', \gamma'}$. Let $(E_j: j < n)$ enumerate the equivalence classes of $E$. For each $i < i_*$, and for each $j < n$, let $X_{i, j} = \{\alpha_{i, \gamma}: \gamma \in E_j\}$. Thus $s_i$ is the disjoint union of $X_{i, j}$ for $j < n$. Moreover, $X_{i, j} \cap X_{i', j'} = \emptyset$ unless $j = j'$; and if $X_{i, j} \cap X_{i', j} \not= \emptyset$ then $X_{i, j} = X_{i', j}$. For each $j < n$, enumerate $\{X_{i, j}: i < i_*\} = (Y_{\ell, j}: \ell < m_i) $ without repetitions. Let $m = \max(m_j: j < n)$; and for each $i <i_*$, define $\eta_i \in \,^n m$ via: $\eta_i(j) =$ the unique $\ell < m_i$ with $X_{i, j} = Y_{\ell, j}$. 

Let $\Lambda = \{\eta_i: i < i_*\}$. For each $s \in P_{\Lambda}$, let $N_{s} = M_{t_s}$ where $t_s = \bigcup_{(j, \ell) \in s} Y_{\ell, j}$.Then the hypotheses on $F$ give commuting isomorphisms $\pi_{\eta, \eta'}: N_{\eta} \cong N_{\eta'}$ for each $\eta, \eta' \in \Lambda$, in such a way that $(\overline{N}, \overline{\pi})$ is a $(\lambda, \overline{N})$-array, and each $\pi_{\eta_i}(p_0(x)) = p_i(x)$. It follows by hypothesis on $T$ that $\bigcup_{i < i_*} p_i(x)$ does not fork over $N_0$. Since $p_0(x) \restriction_{N_0}$ does not fork over $M_0$, it follows by transitivity that $\bigcup_{i < i_*} p_i(x)$ does not fork over $M_0$.
\end{proof}

To state our next theorem, we need to recall some results from \cite{InterpOrders2Ulrich}.

\begin{definition}
	$\Delta$ is a pattern on $I$ if $\Delta \subseteq [I]^{<\aleph_0}$ is closed under subsets. If $T$ is a complete countable theory and $\phi(\overline{x}, \overline{y})$ is a formula of $T$, then say that $\phi(\overline{x}, \overline{y})$ admits $\Delta$ if we can choose $M \models T$ and $(\overline{a}_i: i \in I)$ from $M^{|\overline{y}|}$, such that for every $s \in [I]^{<\aleph_0}$, $M \models \exists \overline{x} \bigwedge_{i \in s} \phi(\overline{x}, \overline{a}_i)$ if and only if $s \in \Delta$.
\end{definition}
The following patterns will be key for analyzing the $T_{n, k}$'s.

\begin{definition}
	Suppose $S \subseteq [I]^k$ for some $k$, and suppose $n > k$. Then let $\Delta_{n, k}(S)$ be the pattern on $[I]^{k-1}$, consisting of all $s \in [[I]^{k-1}]^{<\aleph_0}$ such that there is no $v \in [I]^{n-1}$ with $[v]^{k-1} \subseteq s$ and $[v]^k \subseteq S$.
	
	For each $k \geq 2$ and for each $n > k$, let $S_{k}$ be a random $k$-ary graph on $\omega$, and let $\Delta_{n, k} = \Delta_{n, k}(S_{k})$. 
\end{definition}
We prove the following in \cite{InterpOrders2Ulrich}.

\begin{theorem}\label{TnkLemma3}
	Suppose $3 \leq k < n < \omega$. Then $T_{n, k}$ is a $\trianglelefteq$-minimal theory admitting $\Delta_{n, k}$.
\end{theorem}

We also show that if $T$ admits $\Delta_{n, k}$ then it does so in a particularly nice way:

\begin{theorem}\label{NesitrilRodlLemma}
	Suppose $T$ is a countable simple theory and $\phi(x, y)$ is a formula of $T$ which admits $\Delta_{n, k}$ (so $n > k \geq 3$, since $T$ is simple; possibly $x, y$ are tuples here). Let $\mathfrak{C}$ be the monster model of $T$. Then for every index set $I$ and for every $S \subseteq [I]^k$, we can find some countable $N \preceq \mathfrak{C}$ and some $(b_u: u \in [I]^{k-1})$ from $\mathfrak{C}$, such that for all $s \in [[I]^{k-1}]^{<\aleph_0}$:
	
	\begin{itemize}
		\item If $s \in \Delta_{n, k}(S_k)$, then $\{\phi(x, b_u): u \in [s]^{k-1}\}$ does not fork over $N$;
		\item Otherwise, $\{\phi(x, b_u): u \in [w]^{k-1}\}$ is inconsistent.
	\end{itemize}
\end{theorem}

We now prove the following theorem. This won't be directly used in our arguments for Keisler's order, but gives useful insight to our nonsaturation argument (Theorem~\ref{NonSatPortionGeneral}).

\begin{theorem}\label{NonSat2}
	Suppose $\theta$ is regular, $\mu = \mu^{<\theta}$, $3 \leq k_* \leq \aleph_0$, and $T$ is a countable simple theory which admits $\Delta_{k+1, k}$ for some $k < k_*$.
	
	Suppose $\lambda$ is big enough for $(\mu, \theta, k)$. Write $P = P_{[\lambda]^k \mu \theta}$. Then $P$ forces that $T$ fails the $(\lambda, \theta, \theta, k_*)$-coloring property.
\end{theorem}
\begin{proof}
	Say $\phi(x, y)$ admits $\Delta_{k+1, k}$ (where possibly $x, y$ are tuples).
	
	Let $G$ be $P$-generic over $\mathbb{V}$; we show that $T$ fails the  $(\lambda, \theta, \theta, k_*)$-coloring property in $\mathbb{V}[G]$, which suffices. Note that in $\mathbb{V}[G]$, $|\mu| = \theta = \theta^{<\theta}$, so this makes sense. Now, in $\mathbb{V}[G]$, let $f: [\lambda]^k \to \theta$ be the generic function added by $G$ (so $f =\bigcup G$). Let $R = \{v \in [\lambda]^k: f(v) \not= 0\}$. By Theorem~\ref{NesitrilRodlLemma}, we can find $M \models T$ and $M_0 \preceq M$ countable, and $(b_u: u \in [\lambda]^{k-1})$ from $M$, such that for every $s \in \Delta_{k+1, k}(R)$, $\{\phi(x, b_u): u \in s\}$ does not fork over $M_0$, and for every $s \in [[\lambda]^{k-1}]^{<\aleph_0} \backslash \Delta_{k+1, k}(R)$, $\{\phi(x, b_u): u \in s\}$ is inconsistent. We can suppose $|M| \leq \lambda$.
	
	Suppose towards a contradiction that $\Gamma^\theta_{M, M_0}$ is $(\theta, \theta, k)$-colorable in $\mathbb{V}[G]$; let $F: \Gamma^\theta_{M, M_0} \to P_{X \theta \theta}$ be a $(P_{X \theta \theta}, k)$-coloring. We can suppose $X \in \mathbb{V}$ (in fact we can suppose it is a cardinal in $\mathbb{V}$). Pull everything back to $\mathbb{V}$ to get names $\dot{f}, \dot{R}, \dot{M}, \dot{M}_0, (\dot{b}_u: u \in [\lambda]^{k-1}), \dot{F}$.
	
	For each $v \in [\lambda]^{<\aleph_0}$, let $\dot{\psi}_v(x)$ be the $P$-name for the formula $\exists x \bigwedge_{u \in [v]^{k-1}} \phi(x, \dot{b}_u)$. Choose $p_v \in P$ such that $p_v$ forces $\{\dot{\psi}_v(x)\} \in \dot{\Gamma}^\theta_{\dot{M}, \dot{M}_0}$ (i.e. $\dot{\psi}_v(x)$ does not fork over $\dot{M}_0$) and such that $p_v$ decides $\dot{F}(\dot{\psi}_v)$, say $p_v$ forces that $\dot{F}(\dot{\psi}_v) = \check{f}_v$ for some $f_v \in P_{X \theta \theta}$.
	
	Define $F: [\lambda]^{<\aleph_0} \to [\lambda]^{<\theta}$ via $F(v)= v \cup \bigcup \mbox{dom}(p_v)$. Write $Y = X \cup [\lambda]^{k-1}$, and define $G: [\lambda]^{<\aleph_0} \to P_{Y \mu \theta}$ via $G(v) = p_v \cup f_v$. Since $\lambda$ is big enough for $(\mu, \theta, k)$, we can find some $v \in [\lambda]^k$ and some $(w_u: u \in [v]^{k-1}$) from $[\lambda]^{<\aleph_0}$, such that each $w_u \cap v = F(w_u) \cap v = u$, and such that $\bigcup \{G(w_u): u \in [v]^{k-1}\}$ is a function.
	
	Write $p= \bigcup_{u \in [v]^{k-1}} p_{w_u} \in P$. Then $p$ forces each $\{\dot{\psi}_{w_u}(x)\} \in \dot{\Gamma}^\theta_{\dot{M}, \dot{M}_0}$, and each $\dot{F}(\dot{\psi}_{w_u}(x)) = \hat{f}_{w_u}$; since $\bigcup \{f_{w_u}: u \in [v]^{k-1}\}$ is a function, we get that $p$ forces $\{\dot{\psi}_{w_u}(x): u \in [v]^{k-1}\} \in \dot{\Gamma}^\theta_{\dot{M}, \dot{M}_0}$, in particular it is consistent. Note that $v \not \in \mbox{dom}(p)$, since if $v \in \mbox{dom}(p_{w_u})$ say, then $v \subseteq \bigcup \mbox{dom}(p_{w_u})$, contradicting that $F(w_u) \cap v = u$. Thus we can choose $p' \leq p$ in $P$ with $p'(v) = 1$.
	
	But then $p'$ forces that $v \in \dot{R}$ and $\{\exists x \bigwedge_{u \in [v]^{k-1}} \phi(x, \dot{b}_u)\}$ is consistent, contradiction.
\end{proof}

\begin{corollary}\label{SyntConsOfKeislerCor}
	Suppose $T$ is simple and $3 \leq k \leq \aleph_0$. Then (A) implies (B) implies (C) implies (D):
	
	\begin{itemize}
		\item[(A)] $T$ has $\mathcal{P}^-(k)$-amalgamation of models;
		\item[(B)] $T$ has $<k$-type amalgamation;
		\item[(C)] For every regular uncountable $\theta$, for every $\mu = \mu^{<\theta}$, and for every cardinal $\lambda$, $T$ has the $(\lambda, \mu, \theta, k)$-coloring property, and moreover this statement continues to hold in every forcing extension;
		\item[(D)] $T$ does not admit $\Delta_{k'+1, k'}$ for any $k' < k$.
	\end{itemize}
\end{corollary}
\begin{proof}
	(A) implies (B) is Theorem~\ref{GenAmalgImpliesTypeAmalg}. (B) implies (C) is Theorem~\ref{SatPortion}. (C) implies (D) is Theorem~\ref{NonSat2}. 
\end{proof} 

\begin{remark}
	In the version stated in the introduction, we required $\lambda \leq 2^\mu$. This follows, since in the forcing extension of Theorem~\ref{NonSat2}, $\lambda \leq 2^\mu$.
\end{remark}

\begin{conjecture}\label{Conj1Later}
	(A), (B), (C), (D) above are equivalent.
\end{conjecture}

\section{When $\lambda$ is Small}\label{Refinements}

In this section, we discuss some generalizations of Theorem~\ref{SatPortion}, covering the cases  where $\lambda$ is not big enough for $(\mu, \theta, k)$; we also discuss the cases where $\theta = \aleph_0$.

We will want the following notion of dimension.

\begin{definition}
	Suppose $\Delta \subseteq \mathcal{P}(I)$ is closed under finite intersections and $k \geq 2$. Then say that $\mbox{dim}(\Delta) \geq k$ if there is some $t \in [I]^k$ and some sequence $(w_s: s \in [t]^{k-1})$ from $\Delta$, such that for all $w'_s \in \Delta$ with $w_s \subseteq w'_s$, we have $w'_s \cap t = s$.
	
	Note that $2 \leq k' \leq k$ and $\mbox{dim}(\Delta) \geq k$ implies $\mbox{dim}(\Delta) \geq k'$; thus we can let $\mbox{dim}(\Delta)$ be the supremum of all $2 \leq k < \aleph_0$ such that $\mbox{dim}(\Delta) \geq k$. (If $\mbox{dim}(\Delta) \not \geq 2$ then define $\mbox{dim}(\Delta) = 1$.)
	
	Suppose $\Delta \subseteq \mathcal{P}(I)$ is closed under finite intersections. Then define $\Delta^+ \subseteq \mathcal{P}(I \cup \{*\})$ to be $\Delta \cup \{w \cup \{*\}: w \in \Delta\} \cup \{I\}$.
\end{definition}

We make several easy remarks:

\begin{lemma}
	\begin{itemize}
		\item[(A)] $\mbox{dim}(\mathcal{P}^-(k))= k$ for each $k \geq 1$, as witnessed by $t:= k$ and $w_s := s$. More generally, $\mbox{dim}([n]^{<k}) = k$ for all $n \geq k \geq 1$.
		\item[(B)] Suppose $n > k \geq 3$, and $\Delta \subseteq \mathcal{P}(I)$ is closed under finite intersections and $\mbox{dim}(\Delta) \leq k$. Then $T_{n, k}$ has $\Delta$-amalgamation of models.
		\item[(C)] Suppose $\Delta \subseteq \mathcal{P}^-(k)$ is closed under finite intersections. Then $\mbox{dim}(\Delta^+) = \min(\mbox{dim}(\Delta) +1, \aleph_0)$.
		\item[(D)] If $T$ has $\Delta$-amalgamation of models whenever $\Delta \subseteq \mathcal{P}(n)$ is closed under subsets and satisfies $\mbox{dim}(\Delta) \leq k$, then $T$ has $\Lambda$-type amalgamation whenever $\Lambda \in \mathbf{\Lambda}$ satisfies that $\mbox{dim}(\Delta_\Lambda) < k$. 
	\end{itemize}
\end{lemma}
\begin{proof}
The only claim that is not immediate is (D). This follows from the first half of the proof of Theorem~\ref{GenAmalgImpliesTypeAmalg}.
\end{proof}

We make the following extension to Conjecture~\ref{Conj1Later}:

\begin{conjecture}
Suppose $3 \leq k \leq \aleph_0$. If $T$ has $\mathcal{P}^-(k)$-amalgamation of models, then $T$ has $\Delta$-amalgamation of models whenever $\Delta \subseteq \mathcal{P}(I)$ is closed under finite intersections and satisfies that $\mbox{dim}(\Delta) \leq k$.
\end{conjecture}

Theorem~\ref{SatPortion} has the following adaptation when $\lambda$ is small:

\begin{theorem}\label{SatPortionv2}
Suppose $T$ is simple, and $3 \leq k \leq \aleph_0$, and $\theta$ is a regular uncountable cardinal, and $\mu = \mu^{<\theta}$, and $\lambda$ is not big enough for $(\mu, \theta, k)$. Finally, suppose $T$ has $\Lambda$-type amalgamation for all $\Lambda \in \mathbf{\Lambda}$ with $\mbox{dim}(\Delta_\Lambda) < k$. Then $T$ has the $(\lambda, \mu, \theta, \aleph_0)$-coloring property.
\end{theorem}
\begin{proof}
Suppose $M \models T$ has $|M| \leq \lambda$, and suppose $M_0 \preceq M$ is countable. We wish to show that $\Gamma^{\theta}_{M, M_0}$ is $(P_{X \mu \theta}, k)$-colorable for some set $X$. Let $<_*$ be any well-ordering of $M$.

As in Theorem~\ref{SatPortion}, $M = \bigcup_{s \in \mathcal{P}(\lambda)} M_s$, where $(M_s: s \in \mathcal{P}(\lambda))$ is an independent system of models satisfying the three conditions of Theorem~\ref{IndependentSystemsThm}.

Choose $C \subseteq [\lambda]^{<\theta}$ unbounded and $G: C \to P_{\lambda \mu \theta}$ as in Theorem~\ref{CombThm3}. Let $P$ be the set of all $p(x) \in \Gamma^{\theta}_{M, M_0}$ such that for some $s \in C$, $p(x)$ is a complete type over $M_s$; we write $p(x, M_s)$ to indicate this. $P$ is dense in $\Gamma^{\theta}_{M, M_0}$, so it suffices to show that $P$ is $(P_{X \mu \theta}, \aleph_0)$-colorable for some set $X$.

We can choose $X$ large enough, and $F: P \to P_{X \mu \theta}$ so that if $F(p(x, M_s))$ is compatible with $F(q(x, M_t))$, then:

\begin{itemize}
	\item $s$ and $t$ have the same order-type, and if we let $\rho: s \to t$ be the unique order-preserving bijection, then $\rho$ is the identity on $s\cap t$;
	\item $M_s$ and $M_t$ have the same $<_*$-order-type, and the unique $<_*$-preserving bijection from $M_s$ to $M_t$ is in fact an isomorphism $\tau: M_s \cong M_t$;
	\item For each finite $\overline{a} \in M_{s}^{<\omega}$, if we write $s' = s_{\overline{a}}$ and if we write $t' = s_{\tau(\overline{a})}$, then: $\rho[s'] = t'$ and $\tau \restriction_{M_{s'}}: M_{s'} \cong M_{t'}$.
	\item $\tau[p(x)] = q(x)$.
\end{itemize}

Note that it follows that for every $s' \subseteq s$, $\tau \restriction_{M_{s'}}: M_{s'} \cong M_{\rho[s']}$, since $M_{s'} = \bigcup\{M_{s_{\overline{a}}}: \overline{a} \in (M_{s'})^{<\omega}\}$ and similarly for $M_{t'}$.

We claim that $F$ works.

So suppose $p_i(x, M_{s_i}): i < i_*$ is a sequence from $P$ for $i_* < \aleph_0$, such that $(F(p_i(x)): i < i_*)$ is compatible in $P_{X \mu \theta}$.

Let $\gamma_*$ be the order-type of some or any $s_i$. Enumerate each $s_i = \{\alpha_{i, \gamma}: \gamma < \gamma_*\}$ in increasing order.  Let $E$ be the equivalence relation on $\gamma_*$ defined by: $\gamma E \gamma'$ iff for all $i, i' < k$, $\alpha_{i, \gamma} = \alpha_{i', \gamma}$ iff $\alpha_{i, \gamma'} = \alpha_{i', \gamma'}$. Let $(E_j: j < n)$ enumerate the equivalence classes of $E$. For each $i < i_*$, and for each $j < n$, let $X_{i, j} = \{\alpha_{i, \gamma}: \gamma \in E_j\}$. Thus $s_i$ is the disjoint union of $X_{i, j}$ for $j < n$. Moreover, $X_{i, j} \cap X_{i', j'} = \emptyset$ unless $j = j'$; and if $X_{i, j} \cap X_{i', j} \not= \emptyset$ then $X_{i, j} = X_{i', j}$. For each $j < n$, enumerate $\{X_{i, j}: i < i_*\} = (Y_{\ell, j}: \ell < m_i) $ without repetitions. Let $m = \max(m_j: j < n)$; and for each $i <i_*$, define $\eta_i \in \,^n m$ via: $\eta_i(j) =$ the unique $\ell < m_i$ with $X_{i, j} = Y_{\ell, j}$. 

Let $\Lambda = \{\eta_i: i < i_*\}$. For each $s \in P_{\Lambda}$, let $N_{s} = M_{t_s}$ where $t_s = \bigcup_{(j, \ell) \in s} Y_{\ell, j}$.Then the hypotheses on $F$ give commuting isomorphisms $\pi_{\eta, \eta'}: N_{\eta} \cong N_{\eta'}$ for each $\eta, \eta' \in \Lambda$, in such a way that $(\overline{N}, \overline{\pi})$ is a $(\lambda, \overline{N})$-array, and each $\pi_{\eta_i}(p_0(x)) = p_i(x)$. 

To finish, it suffices to show that $\mbox{dim}(\Delta_\Lambda) < k$. Suppose towards a contradiction that there was some $t \in [n \times m]^k$ and some sequence $(w_s: s \in [t]^{k-1})$ from $\Delta_\lambda$ such that whenever $w_s \subseteq w'_s \in \Delta_{\Lambda}$, then $w'_s \cap t = s$. We can suppose each $w_s \in \Lambda$.

Now, since $k \geq 3$, every pair from $t$ must be covered by some $w_s$; since each $w_s$ is a function from $n$ to $m$, $t$ must be a partial function from $n$ to $m$. For each $j \in \mbox{dom}(t)$, choose some $\alpha_j \in Y_{t(j), j}$. Write $t' = \{\alpha_j: j \in \mbox{dom}(t)\}$. For each $s' \in [t']^{k-1}$, write $s = \{(j, t(j)): \alpha_j \in s'\}$, write $w_s = \eta_i$, and define $w'_{s'} = s_i$. Then $t' \in [\lambda]^k$, and for each $s' \in [t']^{k-1}$, $w'_{s'} \in C$ satisfies that $w'_{s'} \cap t' = s'$. Further, $\bigcup G(w'_{s'}): s' \in [t']^{k-1}$ is a function. This contradicts the choice of $G, C$.
\end{proof}

Recall from Section~\ref{Combinatorics} that $\mu^{+(k-2)}$ is never big enough for $(\mu, \theta, k)$, but $\mu^{+(k-1)}$  may or may not be (as far as we know). It turns out that when $\lambda \leq \mu^{+(k-2)}$ we can weaken the hypothesis on $T$ slightly. The key fact is the following; the case $k =3$ is Lemma 4.8 of \cite{LowDividingLine}, and is also implicit in arguments in \cite{Optimals}.

\begin{theorem}\label{CombinatoricsSmall}
Suppose $\theta$ is a regular uncountable cardinal, $\mu = \mu^{<\theta}$, $3 \leq k < \aleph_0$, and $\lambda \leq \mu^{+(k-2)}$. Then we can find some closed unbounded $C \subseteq [\lambda]^{<\theta}$ and some function $G: C \to \mu$ such that for all $t \in [\lambda]^{k-1}$ and for all sequences $(w_s: s \in [t]^{k-2} \cup \{t\})$ from $C$, if each $w_s \cap t = s$, then $(G(w_s): s \in [t]^{k-2} \cup \{t\})$ is not constant.
\end{theorem}
\begin{proof}
Suppose $n < \omega$. It suffices to show that there is some club $C \subseteq [\mu^{+n}]$ and some map $G: C \to P_{\mu^{+n} \mu \theta}$, such that whenever $t \in [\lambda]^{n+1}$ and whenever  $(w_s: s \in [t]^{n} \cup \{t\})$ is a sequence from $C_{n}$, if each $w_s \cap t \supseteq s$ and if each $(G(w_s): s \in [t]^{n} \cup \{t\})$ is constant, then there is some $s \in [t]^{n}$ with $w_s \supseteq t$. The proof is by induction on $n$; $n = 0$ is trivial, and the step case is like Theorem~\ref{CombThm2}.
\end{proof}

We thus make the following definition.

\begin{definition}
Suppose $\Lambda \in \mathbf{\Lambda}$; say $\Lambda \subseteq \,^n m$. Suppose $2 \leq k < \aleph_0$. Then say that $\mbox{dim}_*(\Lambda) \geq k$ if there is some $t \in [n \times m]^{k-1}$ and some sequence $(w_s: s \in [t]^{k-2} \cup \{t\})$ from $\Lambda$, such that each $w_s \cap t = s$. (In particular, this implies $t$ is a partial function from $n$ to $m$.) Note that $2 \leq k' \leq k$ and $\mbox{dim}_*(\Lambda) \geq k$ implies $\mbox{dim}_*(\Lambda) \geq k'$; thus we can let $\mbox{dim}_*(\Lambda)$ be the supremum of all $2 \leq k < \aleph_0$ such that $\mbox{dim}(\Lambda) \geq k$. (If $\mbox{dim}_*(\Lambda) \not \geq 2$ then let $\mbox{dim}_*(\lambda) = 1$.)
\end{definition}

Note that always $|\Lambda| \geq \mbox{dim}_*(\Lambda) \geq \mbox{dim}(\Delta_\Lambda)$. Further, $\mbox{dim}_*(\Lambda) = 1$ if and only if $\Lambda$ is a singleton, and $\mbox{dim}_*(\Lambda) \leq 2$ if and only if the closure of $\Lambda$ under intersections forms a tree under $\subseteq$.

\begin{theorem}\label{SatPortionv3}
	Suppose $\theta$ is a regular uncountable cardinal, $\mu = \mu^{<\theta}$, $3 \leq k < \aleph_0$ and $\lambda \leq \mu^{+(k-2)}$. Suppose $T$ is a countable simple theory with $\Lambda$-type amalgamation for all $\Lambda \in \mathbf{\Lambda}$ with $\mbox{dim}_*(\Lambda) < k$. Then $T$ has the $(\lambda, \mu, \theta, \aleph_0)$-coloring property.
	
	In particular, if $T$ is any countable simple theory, then $T$ has the $(\mu^+, \mu, \theta, \aleph_0)$-coloring property.
\end{theorem}
\begin{proof}
The first claim follows by exactly the same proof as Theorem~\ref{SatPortionv2}.

For the second claim, it suffices to note that in every simple theory, we can amalgamate systems of types indexed by independent trees of models, provided the types do not fork over the root. This follows easily from the independence theorem, as spelled out in Lemma 4.9 of \cite{LowDividingLine}.
\end{proof}

We now consider the case when $\theta = \aleph_0$. Not much is known here, although the following theorems can be proved analogously to Theorems~\ref{SatPortion} and \ref{SatPortionv2}, respectively.

\begin{theorem}\label{SatPortionv4}
Suppose $3 \leq k < n < \aleph_0$. Then for all cardinals $\mu$ and $\lambda$, $T_{n, k}$ has the $(\lambda, \mu, \theta, k)$-coloring property.
\end{theorem}

\begin{theorem}\label{SatPortionv5}
	Suppose $3 \leq k < n < \aleph_0$. Then for all cardinals $\mu$ and for all cardinals $\lambda$ which are not big enough for $(\mu, \aleph_0, k)$, $T_{n, k}$ has the $(\lambda, \mu, \theta, \aleph_0)$-coloring property.
\end{theorem}

\section{Boolean Algebras and Forcing}\label{ForcingSec}

In this section, we fix notation for forcing and forcing iterations and observe some basic facts about them; essentially we follow \cite{KunenSets} and \cite{Jech}.

We assume the reader is familiar with the standard treatment of forcing in terms of partial orders; see for instance \cite{Kunen} as a reference. In the special case of forcing by a complete Boolean algebra $\mathcal{B}$ (or really, by the set $\mathcal{B}_+$ of positive elements of $\mathcal{B}$), certain notational simplications take place. Normally, passing to the Boolean algebra completion is not worth the trouble; however, since we are obliged to consider Boolean algebra completions in any case, we may as well avail ourselves of the advantages.

A forcing notion is a pre-order $(P, \leq)$ with a maximal element $1$; in other words, $\leq$ is a transitive relation. We will always identify $P$ with its separative quotient, defined by putting $p \sim q$ if for all $p' \in P$, $p'$ is compatible with $p$ if and only if $p'$ is compatible with $q$. When we say that $P$ is a forcing notion, we always mean $P$ is a set (rather than a proper class). If $P$ is a forcing notion and $p, q \in P$, then say that $p$ decides $q$ if either $p \leq q$ or else $p$ and $q$ are incompatible. By separativity, the set of all elements of $P$ which decide $q$ is dense in $P$.

A complete boolean algebra $\mathcal{B}$ is a structure $(\mathcal{B}, \leq, 0, 1, \bigwedge, \bigvee, \lnot)$ satisfying the axioms for a Boolean algebra, with the greatest lower bound property (equivalently, the least upper bound property).  When we view $\mathcal{B}$ as a forcing notion, we always mean $\mathcal{B}_+$, the set of positive elements of $\mathcal{B}$.

Suppose $\mathcal{B}_0, \mathcal{B}_1$ are complete Boolean algebras. Then say that $\mathcal{B}_0$ is a complete subalgebra of $\mathcal{B}_1$ if $\mathcal{B}_0$ is a subalgebra of $\mathcal{B}_1$, and for every $X \subseteq \mathcal{B}_0$, the join of $X$ as computed in $\mathcal{B}_0$ is the same as computed in $\mathcal{B}_1$. (This implies the corresponding statements for meets.)

Given a forcing notion $P$, let $\mathcal{B}(P)$ be its Boolean algebra completion; this is the Boolean algebra (unique up to isomorphism) such that $P$ densely embeds into $\mathcal{B}(P)$. (See Theorem 14.10 of \cite{Jech}.) We always view $P$ as a dense subset of $\mathcal{B}(P)$; this is possible because, as mentioned above, we always identify $P$ with its separative quotient. Every element of $\mathcal{B}(P)$ can be written as $\bigvee X$ for some $X \subseteq P$, which in fact can be chosen to be an antichain. Further, $\bigvee X \leq \bigvee Y$ if and only if for every $x \in X$, there is $x' \leq x$ and $y \in Y$ such that $x' \leq y$.

Say that the forcing notion $P$ is $\theta$-closed if every descending chain from $P$ of length less than $\theta$ has a lower bound in $P$. Say that $P$ is $\kappa$-c.c. if every antichain from $P$ has size less than $\kappa$.  Say that $P$ is $<\theta$-distributive if the intersection of every family of $<\theta$-many dense, downward closed subsets of $P$ is dense. Note that the latter two properties are preserved under passing to $\mathcal{B}(P)$, and if $P$ is $\theta$-closed then it is $<\theta$-distributive.

Suppose $\mathcal{B}$ is a complete Boolean algebra. Then define $\mathbb{N}_{\mathcal{B}}$, the class of nice $\mathcal{B}$-names, as follows: $\dot{\sigma} \in \mathbb{N}_{\mathcal{B}}$ if $\dot{\sigma}$ is a partial function from $\mathbb{N}_{\mathcal{B}}$ into $\mathcal{B}$.  For any forcing notion $P$, define $\mathbb{N}_P$, the set of nice $P$-names, to just be $\mathbb{N}_{\mathcal{B}(P)}$. Note that whenever $P$ is a forcing notion, then every $P$-name is equivalent to a nice $P$-name.

Suppose $\mathcal{B}$ is a complete Boolean algebra, $\phi(x_i: i < n)$ is a formula of set theory, and $\dot{\sigma}_i: i < n$ is a sequence of nice $\mathcal{B}$-names. Then define $\|\phi(\dot{\sigma}_i: i < n)\|_{\mathcal{B}}$ to be the supremum of all $\mathbf{a} \in \mathcal{B}_+$ such that $\mathbf{a} \Vdash_{\mathcal{B}_+} \phi(\dot{\sigma}_i: i < n)$. Note that $\|\phi(\dot{\sigma}_i: i < n)\|_{\mathcal{B}} \Vdash \phi(\dot{\sigma}_i: i < n)$, i.e. the supremum is attained. Thus, whenever $G$ is $\mathcal{B}_+$-generic over $\mathbb{V}$, then $\mathbb{V}[G] \models \phi(\dot{\sigma}_i^G: i < n)$ if and only if $\|\phi(\dot{\sigma}_i: i < n)\|_{\mathcal{B}} \in G$.

Suppose $\mathcal{B}$ is a complete Boolean algebra and $\dot{X} \in \mathbb{N}_{\mathcal{B}}$. Then a partition of $\mathcal{B}$ by $\dot{X}$ is a map $\dot{A}: \mbox{dom}(\dot{X}) \to \mathcal{B}$ (so in particular $\dot{A} \in \mathbb{N}_{\mathcal{B}}$) such that $\mathcal{B}$ forces that $\dot{A}$ has a single element $\bigcup \dot{A}$, which is in $\dot{X}$. Define $\mathbb{N}_\mathcal{B}(\dot{X})$, the set of nice names for elements of $\dot{X}$, to be the set of all $\bigcup \dot{A}$, for $\dot{A}$ a partition of $\mathcal{B}$ by $\dot{X}$ (formally, for each partition $\dot{A}$ of $\mathcal{B}$ by $\dot{X}$, pick some name $\bigcup \dot{A} \in \mathbb{N}_{\mathcal{B}}$ for the unique element of $\dot{A}$). The point is that when considering names for elements of $\dot{X}$, it is enough to consider just names in $\mathbb{N}_\mathcal{B}(\dot{X})$, which is a set.

If $P$ is a forcing notion and $\dot{X}$ is a nice $P$-name, then let $\mathbb{N}_P(\dot{X}) := \mathbb{N}_{\mathcal{B}(P)}(\dot{X})$.
\begin{definition}
	Suppose $\alpha_*> 0$ is an ordinal. By a $<\theta$-support forcing iteration of length $\alpha_*$, we mean sequences $(P_\alpha: \alpha \leq \alpha_*)$, $(\dot{Q}_\alpha: \alpha < \alpha_*)$, where:
	\begin{itemize}
		\item Each $P_\alpha$ is a forcing notion consisting of $\alpha$-sequences, so $P_0 = \{0\}$ is the trivial forcing notion;
		\item For each $\alpha < \alpha_*$, $\dot{Q}_\alpha$ is a nice $P_\alpha$-name for a forcing notion; we can always suppose $P_\alpha$ decides what $1^{\dot{Q}_\alpha}$ is;
		\item For each $\alpha < \alpha_*$, $P_{\alpha+1}$ is the set of all $\alpha+1$-sequences $p$ such that $p \restriction_\alpha \in P_\alpha$ and $p(\alpha) \in \mathbb{N}_{P_\alpha}(\dot{Q}_\alpha)$, and where $p \leq^{P_{\alpha+1}} q$ if: $p \restriction_\alpha \leq^{P_\alpha} q \restriction_\alpha$, and $p \restriction_\alpha$ forces that $p(\alpha) \leq^{\dot{Q}_\alpha} q(\alpha)$. 
		\item For all $\alpha \leq \alpha_*$ limit, $P_\alpha$ is the set of all $\alpha$-sequences $p$ such that for all $\beta < \alpha$, $p \restriction_\beta \in P_\beta$, and further, $\mbox{supp}(p)$ has cardinality less than $\theta$, where $\mbox{supp}(p)$ is $\{\beta < \alpha: p(\beta) = 1^{\dot{Q}_\beta}\}$; put $p \leq^{P_\alpha} q$ if for all $\beta < \alpha$, $p \restriction_\beta \leq^{P_\beta} q \restriction_\beta$.
	\end{itemize}
	
	Note that $\dot{Q}_0$ is really just a forcing notion in $\mathbb{V}$, so we write it as $Q_0$. In the case $\alpha_* = 2$, we write $P_2 = Q_0 * \dot{Q}_1$. 
	
\end{definition}

Note that under our definitions, if $P, Q$ are forcing notions, then $P*\dot{Q}$ is larger than $P \times Q$ (although they both have the same Boolean algebra completions). Also, it is not generally true that $<\theta$-support products are equivalent to the corresponding $<\theta$-support forcing iterations. However, we will only consider $<\theta$-forcing iterations where each $P_\alpha$ forces that $\dot{Q}_\alpha$ is $\theta$-closed; this avoids such pathologies. Recall the standard fact that if each $P_\alpha$ forces that $\dot{Q}_\alpha$ is $\theta$-closed, then $P_{\alpha_*}$ is $\theta$-closed; see \cite{KunenSets}.

 Note that forcing iterations $P*\dot{Q}$ are almost never separative, since whenever $P$ and $\dot{Q}$ are nontrivial, then we can find $p \in P$ and distinct $\dot{q}_0, \dot{q}_1 \in \mathbb{N}_P(\dot{Q})$ such that $p \Vdash \dot{q}_0 = \dot{q}_1$. Nonetheless, we remind the reader of our notational deceit of always identifying forcing notions with their separative quotients.

If $(P_\alpha: \alpha \leq \alpha_*), (\dot{Q}_\alpha: \alpha < \alpha_*)$ is a forcing iteration, then for all $\alpha < \beta \leq \alpha_*$, $\mathcal{B}(P_\alpha)$ is a complete subalgebra of $\mathcal{B}(P_{\beta})$. It turns out we get projection maps in this scenario. These maps will be very helpful later:

\begin{definition}\label{projectionsDef}
	Suppose $\mathcal{B}_0$ is a complete subalgebra of $\mathcal{B}_1$. Then define $\pi = \pi_{\mathcal{B}_1 \mathcal{B}_0}: \mathcal{B}_1 \to \mathcal{B}_0$ as follows. Suppose $\mathbf{a} \in \mathcal{B}_1$; then let $\pi(\mathbf{a})$ be the meet of all $\mathbf{b} \in \mathcal{B}_0$ with $\mathbf{b} \geq \mathbf{a}$.
\end{definition}

We now have a couple of lemmas exploring this notion.

\begin{lemma}\label{projectionsLemma1}
	\begin{itemize}
		\item[(A)] Suppose $\mathcal{B}_0$ is a complete subalgebra of $\mathcal{B}_1$ and $\mathbf{a} \in \mathcal{B}_1$. Then $\pi_{\mathcal{B}_1 \mathcal{B}_0}(\mathbf{a}) \geq \mathbf{a}$, and is the least element of $\mathcal{B}_0$ satisfying this.
		\item[(B)] Each $\pi_{\mathcal{B}, \mathcal{B}}$ is the identity of $\mathcal{B}$. If $\mathcal{B}_0 \subseteq \mathcal{B}_1 \subseteq \mathcal{B}_2$ are complete subalgebras, then $\pi_{\mathcal{B}_1 \mathcal{B}_0} \circ \pi_{\mathcal{B}_2 \mathcal{B}_1} = \pi_{\mathcal{B}_2 \mathcal{B}_0}$.
		
		\item[(C)] Suppose $\mathcal{B}_0$ is a complete subalgebra of $\mathcal{B}_1$, and $\mathbf{a} \in \mathcal{B}_1$. Write $\pi = \pi_{\mathcal{B}_1 \mathcal{B}_0}$. Then for every $\mathbf{b} \in \mathcal{B}_0$, $\mathbf{b} \wedge \pi(\mathbf{a})$ is nonzero if and only if $\mathbf{b} \wedge \mathbf{a}$ is nonzero. This characterizes $\pi(\mathbf{a})$.
	\end{itemize}
\end{lemma}
\begin{proof}
	(A): Let $X$ be the set of all $\mathbf{b} \in \mathcal{B}_0$ with $\mathbf{b} \geq \mathbf{a}$. Since $\mathbf{a}$ is a lower bound to $X$, we get that $\mathbf{a} \leq \bigwedge X = \pi_{\mathcal{B}_1\mathcal{B}_0}(\mathbf{a})$. The second statement is clear.
	
	(B): Clearly, $\pi_{\mathcal{B} \mathcal{B}}$ is the identity. For the second part, suppose $\mathbf{a}_2 \in \mathcal{B}_2$ is given. Let $X_{21}$ be the set of all $\mathbf{a} \in \mathcal{B}_1$ with $\mathbf{a} \geq \mathbf{a}_2$, and write $\mathbf{a}_{21} = \bigwedge X_{21}$. Similarly, let $X_{20}$ be the set of all $\mathbf{a} \in \mathcal{B}_0$ with $\mathbf{a} \geq \mathbf{a}_0$, and write $\mathbf{a}_{20} = \bigwedge X_{20}$;  and let $X_{210}$ be the set of all $\mathbf{a} \in \mathcal{B}_0$ with $\mathbf{a} \geq \mathbf{a}_{21}$, and write $\mathbf{a}_{210} = \bigwedge X_{210}$. We wish to show that $\mathbf{a}_{210} = \mathbf{a}_{20}$; for this it suffices to show that $X_{210} = X_{20}$. That is, given $\mathbf{a} \in \mathcal{B}_0$, we show that $\mathbf{a} \geq \mathbf{a}_{21}$ if and only if $\mathbf{a} \geq \mathbf{a}_2$. By part (A) we have that $\mathbf{a}_{21} \geq \mathbf{a}_2$, so suppose $\mathbf{a} \geq \mathbf{a}_2$; we show $\mathbf{a} \geq \mathbf{a}_{21}$. But this is clear, since $\mathbf{a}$ must be in $X_{21}$. 
	
	(C): Since $\pi(\mathbf{a}) \geq \mathbf{a}$, we have that if $\mathbf{b} \wedge \mathbf{a}$ is nonzero, then so is $\mathbf{b} \wedge \pi(\mathbf{a})$. On the other hand, if $\mathbf{b} \wedge \mathbf{a} = 0$, then $\lnot \mathbf{b} \geq \mathbf{a}$, so $\pi(\mathbf{a}) \leq \lnot \mathbf{b}$, so $\mathbf{b} \wedge \pi(\mathbf{a}) = 0$. Uniqueness is clear. 
\end{proof}

We relate this to forcing:
\begin{lemma}\label{projectionsLemma2}
	Suppose $(P, \dot{Q})$ is a two-step forcing iteration. Write $\mathcal{B}_1 = \mathcal{B}(P*\dot{Q})$, and write $\mathcal{B}_0 = \mathcal{B}(P)$, and write $\pi = \pi_{\mathcal{B}_1 \mathcal{B}_0}$.
	\begin{itemize}
		\item[(A)]  Suppose $(q, \dot{q}) \in P*\dot{Q} \subseteq \mathcal{B}_1$; then $\pi(q, \dot{q}) = q$. 
		\item[(B)] More generally, if $\mathbf{a} = \bigvee_{\delta < \delta_*} (q_\delta, \dot{q}_\delta) \in \mathcal{B}_1$, then $\pi(\mathbf{a}) = \bigvee_{\delta < \delta_*} q_\delta$.
		\item[(C)] If $\mathbf{a} \in \mathcal{B}_1$, then $\pi(\mathbf{a})$ is the join of all $q \in P$ such that for some $\dot{q} \in \mathbb{N}_P(\dot{Q})$, we have $(q, \dot{q}) \leq \mathbf{a}$.
		\item[(D)] If $\mathbf{a} \in \mathcal{B}_1$, and $q \in P$, then $q \leq \pi(\mathbf{a})$ if and only if there is some $\dot{q} \in \mathbb{N}_{P}(\dot{Q})$ such that $(q, \dot{q}) \leq \mathbf{a}$. This characterizes $\pi(\mathbf{a})$.
	\end{itemize}
\end{lemma}
\begin{proof}
	(A) follows from (B).
	
	(B): Write $\mathbf{a}_0 = \bigvee_{\delta < \delta_*} q_\delta$. We show that for all $\mathbf{b} \in \mathcal{B}_0$, $\mathbf{b} \wedge \mathbf{a}_0$ is nonzero if and only if $\mathbf{b} \wedge \mathbf{a}$ is  nonzero; this suffices, by Lemma~\ref{projectionsLemma1}. Suppose $\mathbf{b} \wedge \mathbf{a}_0$ is nonzero; then we can find $\delta < \delta_*$ such that $\mathbf{b} \wedge q_\delta$ is nonzero. Choose $q \in P$ with $q \leq \mathbf{b} \wedge q_\delta$; then $(q, \dot{q}_\delta) \leq \mathbf{b} \wedge \mathbf{a}$ is nonzero, as desired. Conversely, if $\mathbf{b} \wedge \mathbf{a}_0$ is nonzero, then we can find $\delta < \delta_*$ such that $\mathbf{b} \wedge (q_\delta, \dot{q}_\delta)$ is nonzero; thus $\mathbf{b} \wedge q_\delta$ is nonzero. (We are identifying $q \in P$ with $(q, 1)$ in $P*\dot{Q}$.)
	
	(C) follows from (B) (let $(q_\delta, \dot{q}_\delta): \delta < \delta_*$ list all elements of $P*\dot{Q}$ below $\mathbf{a}$).
	
	(D): let $X$ be the set of all $q \in P$ such that there is some $\dot{q} \in \mathbb{N}_P(\dot{Q})$ with $(q, \dot{q}) \leq \mathbf{a}$. By (C) (or (B)), $\pi(\mathbf{a}) = \bigvee X$. Thus, whenever $q \in X$ then $q \leq \pi(\mathbf{a})$. (D) asks for the converse. So suppose $q \leq \pi(\mathbf{a})$. Let $C$ be a maximal antichain of $P$ below $q$, such that $C \subseteq X$ (this is possible since $\pi(\mathbf{a}) = \bigvee X$). For each $p \in C$ choose $\dot{p}(p) \in \dot{Q}$ such that $(p, \dot{p}(p)) \leq \mathbf{a}$. Let $\dot{q} \in \mathbb{N}_{P}(\dot{Q})$ be a $P$-name for an element of $\dot{Q}$, such that for each $p \in C$, $p \Vdash \dot{q} = \dot{p}(p)$.  I claim that $(q, \dot{q}) \leq \mathbf{a}$. It suffices to show that $(q, \dot{q}) \leq \bigvee_{p \in C} (p, \dot{p}(p))$. This amounts to showing that if $(r, \dot{r}) \leq (q, \dot{q})$, then $(r, \dot{r})$ is compatible with $(p, \dot{p}(p))$ for some $p \in C$; but this is clear, since we can find some $p \in C$ such that $r$ is compatible with $p$ (since $q = \bigvee C$), so choose $r' \leq r \wedge p$. Then $r' \Vdash \dot{p}(p) = \dot{q}$, so $(p, \dot{p}(p))$ and $(\dot{r}, \dot{q})$ are compatible, as witnessed by the lower bound $(r', \dot{p}(p))  = (r', \dot{q})$ (recalling our convention that we really always work in the separative quotient). Uniqueness is clear.
\end{proof}

Finally, we will need the following lemma. The special case when $\mathcal{B} = \mathcal{P}(\lambda)$ is Proposition 4.1 of \cite{Kanamori}. A cardinal $\tau$ is strongly compact if, whenever $\Gamma$ is a set of $\mathcal{L}_{\tau \tau}$-formulas, if every subset of $\Gamma$ of size less than $\tau$ is satisfiable, then $\Gamma$ is satisfiable. Every supercompact cardinal is strongly compact.

\begin{lemma}\label{CompactExt}
	Suppose $\tau$ is strongly compact, and $\mathcal{B}$ is a complete Boolean algebra, and $\mathcal{D}$ is a $\tau$-complete filter on $\mathcal{B}$. Suppose $\mathcal{B}$ is $<\tau$-distributive. Then $\mathcal{D}$ extends to a $\tau$-complete ultrafilter on $\mathcal{B}$.
\end{lemma}
\begin{proof}
	
	Let $\mathcal{L}$ be the language with a constant symbol for each element $\mathbf{a} \in \mathcal{B}$ (also denoted $\mathbf{a}$), and with a unary relation symbol $U$. Let $\Gamma$ assert the following:
	
	\begin{itemize}
		\item $\{\mathbf{a} \in \mathcal{B}: U(\mathbf{a})\}$ is an ultrafilter (this is first-order);
		\item For every descending chain $(\mathbf{a}_\gamma: \gamma < \gamma_*)$ from $\mathcal{B}$ of length less then $\tau$, if $U(\mathbf{a}_\gamma)$ holds for each $\gamma$ then $U(\bigcap_\gamma \mathbf{a}_\gamma)$ holds;
		\item $U(\mathbf{a})$ holds for all $\mathbf{a} \in \mathcal{D}$.
	\end{itemize}
	
	To see that this is $\tau$-satisfiable: let $\Gamma_0 \subseteq \Gamma$ have size less than $\tau$. Choose $X \in [\mathcal{B}]^{<\tau}$ containing all the constants appearing in $\Gamma_0$. Let $\mathbf{a}_0 = \bigcap\{\mathbf{a}: \mathbf{a} \in X \cap \mathcal{D}\}$, so $\mathbf{a} \in \mathcal{D}$ is nonzero (as $\mathcal{D}$ is $\tau$-complete). Choose $\mathbf{a}_1 \leq \mathbf{a}_0$ nonzero such that $\mathbf{a}_1$ decides every element of $X$ (this is possible since $\mathcal{B}$ is $<\tau$-distributive). For all $\mathbf{b} \in \mathcal{B}$, define $U({{\mathbf{b}}})$ to hold if and only if $\mathbf{a} \leq \mathbf{b}$. Then this clearly defines a model of $\Gamma_0$.
\end{proof}

\section{Full Boolean-Valued Models}\label{ReviewSec}
We recall the setup of \cite{BVModelsUlrich}.

As a convention, if $X$ is a set and $\mathcal{L}$ is a language, then $\mathcal{L}(X)$ is the set of formulas of $\mathcal{L}$ with parameters taken from $X$.

Suppose $\mathcal{B}$ is a complete Boolean algebra. A $\mathcal{B}$-valued structure is a pair $(\mathbf{M}, \| \cdot \|_{\mathbf{M}})$ where $\mathbf{M}$ is a set and $\|\cdot \|_{\mathbf{M}}$ is a map from $\mathcal{L}(\mathbf{M})$ to $\mathcal{B}$, which satisfies a natural list of axioms, for instance: for every formula $\phi(x, \overline{a}) \in \mathcal{L}(\mathbf{M})$, $\|\exists x \phi (x, \overline{a}) \|_{\mathbf{M}} = \sup_{a \in \mathbf{M}} \|\phi(a, \overline{a})\|_{\mathbf{M}}$. We are only interested in the case when $\mathbf{M}$ is full, i.e. when in fact  $\|\exists x \phi (x, \overline{a}) \|_{\mathbf{M}} = \mbox{max}_{a \in \mathbf{M}} \|\phi(a, \overline{a})\|_{\mathbf{M}}$. If $T$ is a theory, then we write $\mathbf{M} \models^{\mathcal{B}} T$, and say that $\mathbf{M}$ is a full $\mathcal{B}$-valued model of $T$, if $\|\phi\|_{\mathbf{M}} = 1$ for all $\phi\in T$.

For example, (ordinary) $\mathcal{L}$-structures are the same as full $\{0, 1\}$-valued $\mathcal{L}$-structures, which can thus be viewed as full $\mathcal{B}$-valued structures for any $\mathcal{B}$. Also, if $M$ is an $\mathcal{L}$-structure and $\lambda$ is a cardinal, then $M^\lambda$ is a $\mathcal{P}(\lambda)$-valued $\mathcal{L}$-structure; moreover, we have the canonical elementary embedding $\mathbf{i}: M \preceq M^\lambda$, given by the diagonal map. We call this the pre-{\L}o{\'s} map. 

If $\mathbf{M}$ is a full $\mathcal{B}$-valued model of $T$ and $\mathcal{U}$ is an ultrafilter on $\mathcal{B}$, then we can form the specialization $\mathbf{M}/\mathcal{U} \models T$, which comes equipped with a canonical surjection $[\cdot]_{\mathcal{U}}: \mathbf{M}\to \mathbf{M}/\mathcal{U}$, satisfying that for all $\phi(\overline{a}) \in \mathcal{L}(\mathbf{M})$, $\mathbf{M}/\mathcal{U}\models \phi([\overline{a}]_{\mathcal{U}})$ if and only if $\|\phi(\overline{a})\|_{\mathbf{M}} \in \mathcal{U}$. This generalizes the ultrapower construction $M^\lambda/\mathcal{U}$; note that the {\L}o{\'s} embedding of $M$ into $M^\lambda/\mathcal{U}$ is the composition of the pre-{\L}o{\'s} embedding with $[\cdot]_{\mathcal{U}}$.

In \cite{BVModelsUlrich}, we prove the following compactness theorem for  full Boolean-valued models:

\begin{theorem}\label{Compactness}
	Suppose $\mathcal{B}$ is a complete Boolean algebra, $X$ is a set, $\Gamma \subseteq \mathcal{L}( X)$, and $F_0, F_1: \Gamma \to  \mathcal{B}$ with $F_0(\phi(\overline{a})) \leq F_1(\phi(\overline{a}))$ for all $\phi(\overline{a}) \in \Gamma$. Then the following are equivalent:
	
	\begin{itemize}
		\item[(A)] There is some full $\mathcal{B}$-valued structure $\mathbf{M}$ and some map $\tau: X \to \mathbf{M}$, such that for all $\phi(\overline{a}) \in \Gamma$, $F_0(\phi(\overline{a})) \leq \|\phi(\tau(\overline{a}))\|_{\mathbf{M}} \leq F_1(\phi(\overline{a}))$;
		
		\item[(B)] For every finite $\Gamma_0 \subseteq \Gamma$ and for every $\mathbf{c} \in \mathcal{B}_+$, there is some $\{0, 1\}$-valued $\mathcal{L}$-structure $M$ and some map $\tau: X \to M$, such that for every $\phi(\overline{a}) \in \Gamma$, if $\mathbf{c} \leq F_0(\phi(\overline{a}))$ then $M \models \phi(\tau(\overline{a}))$, and if $\mathbf{c} \leq \lnot F_1(\phi(\overline{a}))$ then $M \models \lnot \phi(\tau(\overline{a}))$. 
	\end{itemize}
\end{theorem}

Here is a first application: given $\mathcal{B}$-valued models $\mathbf{M}\subseteq \mathbf{N}$, say that $\mathbf{M} \preceq \mathbf{N}$ if $\|\cdot\|_{\mathbf{M}} \subseteq \| \cdot \|_{\mathbf{N}}$. Say that $\mathbf{N}$ is $\lambda^+$-saturated if for every $\mathbf{M}_0 \preceq \mathbf{N}$ with $|\mathbf{M}_0| \leq \lambda$ and for every $\mathbf{M}_1 \succeq \mathbf{M}_0$ with $|\mathbf{M}_1| \leq \lambda$, there is some elementary embedding $f: \mathbf{M}_1 \preceq \mathbf{N}$ extending the inclusion from $\mathbf{M}_0$ into $\mathbf{N}$. Then in \cite{BVModelsUlrich}, we show that for every $\mathcal{B}$-valued structure $\mathbf{M}$ and for every $\lambda$, there is an elementary extension $\mathbf{N} \succeq \mathbf{M}$ such that $\mathbf{N}$ is full and moreover $\lambda^+$-saturated.

Suppose $T$ is a complete countable theory, and $\mathcal{U}$ is an ultrafilter on the complete Boolean algebra $\mathcal{B}$. We observe in \cite{BVModelsUlrich} that if there is some $\lambda^+$-saturated $\mathbf{M} \models^{\mathcal{B}} T$ with $\mathbf{M}/\mathcal{U}$ $\lambda^+$-saturated, then for every $\lambda^+$-saturated $\mathbf{M} \models^{\mathcal{B}} T$, $\mathbf{M}/\mathcal{U}$ is $\lambda^+$-saturated. We define that $\mathcal{U}$ $\lambda^+$-saturates $T$ in this case. This coincides with Malliaris and Shelah's notion of $(\lambda, \mathcal{B}, T)$-morality.

Finally, in \cite{BVModelsUlrich} we give the following generalization of Malliaris and Shelah's Existence Theorem and Separation of Variables:

\begin{theorem}\label{ExistenceThm}
	Suppose $\mathcal{B}_0, \mathcal{B}_1$ are complete Boolean algebras such that $\mbox{c.c.}(\mathcal{B}_0) > \lambda$ (i.e. $\mathcal{B}_0$ has an antichain of size $\lambda$) and $2^{<\mbox{c.c.}(\mathcal{B}_1)} \leq 2^\lambda$. Suppose $\mathcal{U}_1$ is an ultrafilter on $\mathcal{B}_1$. Then there is a strongly $\lambda$-regular ultrafilter $\mathcal{U}_0$ on $\mathcal{B}_0$ such that for all complete countable theories $T$, $\mathcal{U}_0$ $\lambda^+$-saturates $T$ if and only if $\mathcal{U}_1$ does.
\end{theorem}

We immediately obtain the following corollary:
\begin{corollary}\label{KeislerOrderCharacterization}
	Suppose $T_0, T_1$ are theories. Then $T_0 \trianglelefteq T_1$ if and only if for every $\lambda$, for every complete Boolean algebra $\mathcal{B}$ with the $\lambda^+$-c.c., and for every ultrafilter $\mathcal{U}$ on $\mathcal{B}$, if $\mathcal{U}$ $\lambda^+$-saturates $T_1$, then $\mathcal{U}$ $\lambda^+$-saturates $T_0$.
	
	In fact, if $\mathcal{B}$ has the $\lambda^+$-c.c. and $\mathcal{U}$ is an ultrafilter on $\mathcal{B}$, then the set of all complete countable theories which are $\lambda^+$-saturated by $\mathcal{U}$ is a principal dividing line in Keisler's order.
\end{corollary}

It will be convenient to have a more combinatorial formulation of $\mathcal{U}$ $\lambda^+$-saturating $T$. 

\begin{definition}
	Given an index set $I$, an $I$-distribution in $\mathcal{B}$ is a function $\mathbf{A}: [I]^{<\aleph_0} \to \mathcal{B}_+$, such that $\mathbf{A}(\emptyset) = 1$, and $s \subseteq t$ implies $\mathbf{A}(s) \geq \mathbf{A}(t)$. If $\mathcal{D}$ is a filter on $\mathcal{B}$, we say that $\mathbf{A}$ is in $\mathcal{D}$ if $\mbox{im}(\mathbf{A}) \subseteq \mathcal{D}$. $I$ will often be $\lambda$, but at other times it is convenient to let $I$ be a partial type $p(\overline{x})$.

	Say that $\mathbf{A}$ is multiplicative if for all $s \in [I]^{<\aleph_0}$, $\mathbf{A}(s) = \bigwedge_{i \in s} \mathbf{A}(\{i\})$. Clearly, multiplicative distributions are in correspondence with maps $\mathbf{A}: I \to \mathcal{B}_+$ such that the image of $\mathbf{A}$ has the finite intersection property.
	
	If $\mathbf{A}, \mathbf{B}$ are $I$-distributions in $\mathcal{B}$, then say that $\mathbf{B}$ refines $\mathbf{A}$ if $\mathbf{B}(s) \leq \mathbf{A}(s)$ for all $s \in [I]^{<\aleph_0}$.  
	
	Suppose $T$ is a theory, and $I$ is an index set. Say that $\overline{\phi}$ is an $I$-sequence of formulas if $\overline{\phi} = (\phi_i(\overline{x}, \overline{y}_i): i \in I)$ for some sequence of formulas $\phi_i(\overline{x}, \overline{y}_i)$, where all of the $\overline{y}_i$'s are disjoint with each other and with $\overline{x}$. If $\mathcal{B}$ is a complete Boolean algebra and $\mathbf{A}$ is an $I$-distribution in $\mathcal{B}$, then say that $\mathbf{A}$ is an $(I, T, \overline{\phi})$-{\L}o{\'s} map if there is some $\mathbf{M} \models^{\mathcal{B}} T$ and some choice of $\overline{a}_i \in \mathbf{M}^{|\overline{y}_i|}$ such that for every $s \in [I]^{<\aleph_0}$, $\mathbf{A}(s) = \| \exists \overline{x} \bigwedge_{i \in s} \phi_i(\overline{x}, \overline{a}_i) \|_{\mathbf{M}}$. Say that $\mathbf{A}$ is an $(I, T)$-{\L}o{\'s} map if it is an $(I, T, \overline{\phi})$-{\L}o{\'s} map for some $\overline{\phi}$.
	
\end{definition}

In \cite{BVModelsUlrich} we prove the following:

\begin{theorem}
	Suppose $\mathcal{U}$ is an ultrafilter on the complete Boolean algebra $\mathcal{B}$, and $T$ is a complete countable theory. Then $\mathcal{U}$ $\lambda^+$-saturates $T$ if and only if every $(\lambda, T)$-{\L}o{\'s} map in $\mathcal{U}$ has a multiplicative refinement in $\mathcal{U}$.
\end{theorem}

Finally, we quote a pair of nonsaturation theorems from \cite{InterpOrders2Ulrich}. $T_{rf}$ is the theory of the random binary function; it is proven to be a Keisler-minimal unsimple theory in \cite{InterpOrders2Ulrich}. $T_{nlow}$ is a supersimple nonlow theory defined by Casanovas and Kim \cite{SupersimpleNonlow}; it is proven to be a Keisler-minimal nonlow theory in \cite{InterpOrders2Ulrich}.  

\begin{theorem}~\label{SimpleNonSat}
	Suppose $\mathcal{B}$ is a complete Boolean algebra; write $\lambda = \mbox{c.c.}(\mathcal{B})$. Suppose $\mathcal{U}$ is a nonprincipal ultrafilter on $\mathcal{B}$. Then $\mathcal{U}$ does not $\lambda^+$-saturate any nonsimple theory.  In fact, we can find a $(\lambda, T_{rf})$-{\L}o{\'s} map $\mathbf{A}$ in $\mathcal{U}$, such that whenever $\mathcal{B}$ is a complete subalgebra of $\mathcal{B}_*$, if $\mathcal{B}_*$ has the $\lambda$-c.c., then $\mathbf{A}$ has no multiplicative refinement in $\mathcal{B}_*$.
\end{theorem}

\begin{theorem}\label{LowNonSat}
	Suppose $\mathcal{B}$ is a complete Boolean algebra; write $\lambda = \mbox{c.c.}(\mathcal{B})$. Suppose $\mathcal{U}$ is an $\aleph_1$-incomplete ultrafilter on $\mathcal{B}$. Then $\mathcal{U}$ does not $\lambda^+$-saturate any nonlow theory. In fact, there is a $(\lambda, T_{nlow})$-{\L}o{\'s} map $\mathbf{A}$ in $\mathcal{U}$ such that whenever $\mathcal{B}$ is a complete subalgebra of $\mathcal{B}_*$, if $\mathcal{B}_*$ has the $\lambda$-c.c., then $\mathbf{A}$ has no multiplicative refinement in $\mathcal{B}_*$.
\end{theorem}

\section{Colorability is Preserved Under Forcing Iterations}\label{KeislerNewForcingAmalg}

This section is a mild generalization of results from \cite{SP}; there, only the case $\theta > \aleph_0$ is dealt with.

The following will follow from Theorem~\ref{IterationPreserves1} (and has a somewhat easier proof):

\vspace{1 mm}

\noindent \textbf{Theorem.} Suppose $\theta$ is a regular cardinal, $(P_\alpha: \alpha \leq \alpha_*, \dot{Q}_\alpha: \alpha < \alpha_*)$ is a $<\theta$-support forcing iteration, and suppose $R$ is a forcing notion. Suppose $3 \leq k \leq \theta$, and each $P_\alpha$ forces that $\dot{Q}_\alpha$ is $\theta$-closed, has the greatest lower bounds property, and is $(\check{R}, k)$-colorable. Then $P_{\alpha_*}$ is $(\prod_{\alpha_*} R, k)$-colorable, where $\prod_{\alpha_*} R$ is the $<\theta$-support product of $\alpha_*$-many copies of $R$. 

\vspace{1 mm}

In fact, this would be enough for our applications, but it is inconvenient that the hypotheses are not fully preserved. Namely, the greatest lower bound property is not necessarily preserved under $<\theta$-forcing iterations. 

First of all, this is in fact only a problem when $\theta> \aleph_0$. The following is a mild generalization of Lemma 1.4 of \cite{nLinked}, which addresses the case $\alpha_* \leq 2^{\aleph_0}$ and $R = (\omega, \emptyset)$. 

\begin{theorem}\label{IterationPreserves0}
	Suppose  $(P_\alpha: \alpha \leq \alpha_*, \dot{Q}_\alpha: \alpha < \alpha_*)$ is a finite support forcing iteration, and suppose $R$ is a poset. Suppose $3 \leq k \leq \aleph_0$, and each $P_\alpha$ forces that $\dot{Q}_\alpha$ is $(\check{R}, k)$-colorable. Then $P_{\alpha_*}$  is $(\prod_{\alpha_*} R, k)$-colorable, where $\prod_{\alpha_*} R$ is the finite support product of $\alpha_*$-many copies of $R$. 
\end{theorem}

\begin{proof}
Suppose $(P_\alpha: \alpha \leq \alpha_*), (\dot{Q}_\alpha: \alpha < \alpha_*)$ and $R$ are given. For each $\alpha < \alpha_*$, $P_\alpha$ forces there is some $(\check{R}, k)$-coloring $\dot{F}_\alpha: \dot{Q}_\alpha \to \check{R}$ of $ \dot{Q}_\alpha$. We can suppose $P_\alpha$ forces that $\dot{F}_\alpha(1) = 1$. Let $R' = \prod_{\alpha_*} R$ be the finite support product of $\alpha_*$-many copies of $R$; we show that $P_{\alpha_*}$ is $(R', k)$-colorable. 

Let $P^0 \subseteq P_{\alpha_*}$ be the set of all $p$ such that for each $\alpha < \alpha_*$, $p \restriction_\alpha$ decides $\dot{F}_\alpha(p(\alpha))$. I claim that $P^0$ is dense in $P_{\alpha_*}$. Suppose towards a contradiction $p \in P_{\alpha_*}$ had no extension in $P^0$. Write $p_0 = p$. Having defined $p_n \leq p$, let $\alpha_n < \alpha_*$ be largest so that $p_n \restriction_{\alpha_n}$ does not decide $\dot{F}_{\alpha_n}(p_n(\alpha_n))$ (this is possible since $\mbox{supp}(p_n)$ is finite). Choose $q_n \leq p_n \restriction_{\alpha_n}$ in $P_{\alpha_n}$ which decides $\dot{F}_{\alpha_n}(p_n(\alpha_n))$. Let $p_{n+1} \in P_{\alpha_*}$ be defined by: $p_{n+1}(\alpha) = p_n(\alpha)$ for all $\alpha \geq \alpha_n$, and $p_{n+1}(\alpha) = q_n(\alpha)$ for all $\alpha < \alpha_n$. Then $p_{n+1} < p_n$ and we can continue. But this will give an infinite decreasing sequence of ordinals $(\alpha_n: n < \omega)$.

Thus $P^0$ is dense in $P_{\alpha_*}$. We now find an $(R', k)$-coloring $F: P^0 \to R'$, which suffices. Given $p \in P^0$ and $\alpha < \alpha_*$, let $r_\alpha(p) \in R$ be such that $p \restriction_\alpha$ forces that $\dot{F}_\alpha(p(\alpha)) = \check{r}_\alpha(\check{p})$. Let $r = (r_\alpha: \alpha < \alpha_*)$; since $r_\alpha = 1$ whenever $\alpha \not \in \mbox{supp}(p)$, we have that $r \in R'$. Define $F(p) = r$. 

Now suppose $(p_i: i < i_*)$ is a sequence from $P^0$ with $i_* < k$, such that $(F(p_i): i < i_*)$ are compatible in $R'$.  Write $\Gamma = \bigcup_{i < i_*} \mbox{supp}(p_i)$. 

By induction $\alpha \leq \alpha_*$, we construct a lower bound $s_\alpha$ to $(p_i \restriction_\alpha: i < i_*)$ in $P_\alpha$, such that $\mbox{supp}(s_\alpha) \subseteq \Gamma \cap \alpha$, and for $\alpha < \alpha'$, $s_{\alpha'} \restriction_\alpha = s_\alpha$. 

Limit stages of the induction are clear. So suppose we have constructed $s_\alpha$. If $\alpha \not \in \Gamma$ clearly we can let $s_{\alpha+1} = s_\alpha \,^\frown (1^{\dot{Q}_\alpha})$; so suppose instead $\alpha \in \Gamma$. $s_\alpha$ forces that each $\dot{F}_\alpha(p_i(\alpha)) = \check{r}_\alpha(\check{p}_i)$, and $(r_\alpha(p_i): i < i_*)$ are compatible in $R_\alpha$, thus we can choose $\dot{q} \in \mathbb{N}_{P_\alpha}(\dot{Q}_\alpha)$, such that $s_\alpha$ forces $\dot{q}$ is a lower bound to $(p_i(\alpha): i < i_*)$ in $\dot{Q}_\alpha$. Let $s_{\alpha+1} = s_\alpha\,^\frown(\dot{q})$. 

Thus the induction goes through, and $s_{\alpha_*}$ is a lower bound to $(p_i: i < i_*)$. 
\end{proof}

For $\theta > \aleph_0$, we find the following sweet spot intermediate between being $(R, k)$ colorable, and being $(R, k)$-colorable and $\theta$-closed and having the greatest lower bounds property.

\begin{definition}
	Suppose $P, R$ are forcing notions and $3 \leq k \leq \aleph_0$. Then say that $P$ is strongly $(R, k)$-colorable if for some dense subset $P_0$ of $P$, there is some $F: P_0 \to R$, such that for every sequence $s \in [P_0]^{<k}$, if $F[s]$ is compatible in $R$, then $s$ has a greatest lower bound in $P$. We also say that $F: (P, P_0) \to R$ is a strong $(R, k)$-coloring of $P$.
	
	Say that $P$ has greatest lower bounds for $<\theta$-chains if whenever $(p_\alpha: \alpha < \alpha_*)$ is a descending chain from $P$ of length $\alpha_* < \theta$, then $(p_\alpha: \alpha < \alpha_*)$ has a greatest lower bound in $P$. (In particular, this implies $P$ is $\theta$-closed.)
\end{definition}

The following lemma sums up some immediate facts.

\begin{lemma}\label{ObviousLemmasForcingAP}
	Suppose $\theta$ is a regular cardinal, $3 \leq k \leq \aleph_0$, and $P, R$ are forcing notions.
	\begin{enumerate}
		\item If $P$ is strongly $(R, k)$-colorable then $P$ is $(R, k)$-colorable.
		\item If $P$ is strongly $(R, k)$-colorable and $R$ is $(R', k)$-colorable, then $P$ is strongly $(R', k)$-colorable.
		\item If $P$ has the greatest lower bound property, then $P$ is $(R, k)$-colorable if and only if $P$ is strongly $(R, k)$-colorable. In particular, this holds whenever $P$ is of the form $\Gamma^\theta_{M, M_0}$.
	\end{enumerate}
\end{lemma}

\begin{theorem}\label{IterationPreserves1}
	Suppose $\theta$ is a regular uncountable cardinal, $(P_\alpha: \alpha \leq \alpha_*, \dot{Q}_\alpha: \alpha < \alpha_*)$ is a $<\theta$-support forcing iteration, and suppose $R$ is a forcing notion. Suppose $3 \leq k \leq \aleph_0$, and each $P_\alpha$ forces that $\dot{Q}_\alpha$ has greatest lower bounds for $<\theta$-chains, and is strongly $(\check{R}, k)$-colorable. Then $P_{\alpha_*}$  has greatest lower bounds for $<\theta$-chains, and is strongly $(\prod_{\alpha_*} R, k)$-colorable, where $\prod_{\alpha_*} R$ is the $<\theta$-support product of $\alpha_*$-many copies of $R$. 
\end{theorem}
\begin{proof}
	This is theorem 3.5 of \cite{SP}, but we repeat the proof for the reader's convenience.
	
	Suppose $(P_\alpha: \alpha \leq \alpha_*), (\dot{Q}_\alpha: \alpha < \alpha_*)$ and $R$ are given. For each $\alpha < \alpha_*$, $P_\alpha$ forces there is some strong $(\check{R}, k)$-coloring $\dot{F}_\alpha: (\dot{Q}_\alpha, \dot{Q}^0_\alpha) \to \check{R}$ of $ \dot{Q}_\alpha$ (so $\dot{Q}^0_\alpha$ is forced to be a dense subset of $\dot{Q}_\alpha$, and $\dot{F}_\alpha: \dot{Q}^0_\alpha \to \check{R}$). We can suppose $P_\alpha$ forces that $\dot{F}_\alpha(1^{\dot{Q}_\alpha}) = 1^{\check{R}}$. Since each $\dot{Q}^0_\alpha$ is forced by $P_\alpha$ to be dense in $\dot{Q}_\alpha$, the sequence $(\dot{Q}^0_\alpha: \alpha < \alpha_*)$ induces a forcing iteration $(P^0_\alpha: \alpha \leq \alpha_*), (\dot{Q}^0_{\alpha}: \alpha < \alpha_*)$, with each $P^0_\alpha$ dense in $P_\alpha$ (so these are equivalent forcing iterations). 
	
	We begin by proving the following easy claim.
	
	\vspace{1 mm}
	\noindent \textbf{Claim.} $P_{\alpha_*}$ has greatest lower bounds for $<\theta$-chains. In fact, suppose $\gamma_*< \theta$, and $(p_\gamma: \gamma < \gamma_*)$ is a descending chain from $P_{\alpha_*}$; then it has a greatest lower bound $p$ in $P_{\alpha_*}$, such that $\mbox{supp}(p) \subseteq \bigcup_{\gamma < \gamma_*} \mbox{supp}(p_\gamma)$.

	\begin{proof}By induction on $\alpha \leq \alpha_*$, we construct $(q_\alpha: \alpha \leq \alpha_*)$ such that each $q_\alpha \in P_\alpha$ with $\mbox{supp}(q_\alpha) \subseteq \bigcup_{\gamma < \gamma_*} \mbox{supp}(p_\gamma) \cap \alpha$, and for $\alpha < \beta \leq \alpha_*$, $q_{\beta} \restriction_\alpha = q_\alpha$, and for each $\alpha \leq \alpha_*$, $q_\alpha$ is a greatest lower bound to $(p_\gamma \restriction_\alpha: \gamma < \gamma_*)$ in $P_{\alpha}$. At limit stages there is nothing to do; so suppose we have defined $q_\alpha$. If $\alpha \not \in \bigcup_{\gamma < \gamma_*} \mbox{supp}(p_\gamma)$ then let $q_{\alpha+1} = q_\alpha \,^\frown (1^{\dot{Q}_\alpha})$. Otherwise, since $q_\alpha$ forces that $(p_\gamma(\alpha): \gamma < \gamma_*)$ is a descending chain from $\dot{Q}_\alpha$, we can find $\dot{q}$, a $P_\alpha$-name for an element of $\dot{Q}_\alpha$, such that $q_\alpha$ forces $\dot{q}$ is the greatest lower bound. Let $q_{\alpha+1} = q_\alpha \,^\frown(\dot{q})$.
	\end{proof}
	
	Now, if $\alpha_* < \aleph_0$, then we can finish as in Case 1 (since finite iterations are also finite support iterations). Thus we can suppose $\alpha_* \geq \aleph_0$.  Let $R' = \prod_{\omega \times \alpha_*} R$ be the finite support product of $\omega \times \alpha_*$-many copies of $R$; we show that $P_{\alpha_*}$ is $(R', k)$-colorable. (The only reason we need $\alpha_* \geq \aleph_0$ is to get $R' \cong \prod_{\alpha_*} R$.)
	
	Fix some $p \in P_{\alpha_*}^0$ for a while. Note that $\mbox{supp}(p) \in [\alpha_*]^{<\theta}$. 
	
	It is easy to find, for each $n < \omega$, elements $\mathbf{q}_n(p) \in P^0_{\alpha_*}$ with $\mathbf{q}_0(p) = p$, so that for all $n< \omega$:
	
	\begin{itemize}
		\item $\mathbf{q}_{n+1}(p) \leq \mathbf{q}_n(p)$;
		\item For all $\alpha <\alpha_*$, $\mathbf{q}_{n+1}(p) \restriction_\alpha$ decides $\dot{F}_\alpha(\mathbf{q}_{n}(\alpha))$.
	\end{itemize}
	
	For each $n > 0$ and for each $\alpha < \alpha_*$, we can find $r_{n-1, \alpha}(p) \in R$ such that $q_n \restriction_\alpha$ forces that $\dot{F}_\alpha(q_{n-1}(\alpha)) = \check{r}_{n-1, \alpha}(p)$. (Whenever $\alpha \not \in \mbox{supp}(\mathbf{a}_{n-1})$, we have $r_{n-1, \alpha}(p)= 0$.)
	
	Let $\mathbf{q}_\omega(p) \in P_{\alpha_*}$ be the greatest lower bound of $(\mathbf{q}_n(p): n < \omega)$; this exists by the claim. 
	
	Define $P^0 = \{\mathbf{q}_\omega(p): p \in P^0_{\alpha_*}\}$. For each $q \in P^0$, choose $\mathbf{p}(q) \in P^0_{\alpha_*}$ such that $q = \mathbf{q}_\omega(\mathbf{p}(q))$. For each $n < \omega$, let $\mathbf{p}_n(q) = \mathbf{q}_{n}(\mathbf{p}(q))$, and for each $\alpha < \alpha_*$, let $r_{n, \alpha}(q) = r_{n, \alpha}(\mathbf{p}(q))$.
	
	We have arranged that for all $q \in P^0$, $q$ is the greatest lower bound of $(\mathbf{p}_n(q): n < \omega)$, and for all $n < \omega$ and $\alpha < \alpha_*$, $\mathbf{p}_{n+1}(q) \restriction_\alpha$ forces that $\dot{F}_\alpha(\mathbf{p}_n(q)(\alpha)) = \check{r}_{n, \alpha}(\check{q})$.
	
	Define $F: P^0 \to R'$ via $F(q) = (r_{n, \alpha}(q): \alpha < \alpha_*, n < \omega)$. We claim that $F: (P, P_0) \to R'$ is a strong $(R', k)$-coloring.
	
	So suppose $(q_i: i < i_*)$ is a sequence from $P^0$ with $i_* < k$, such that $(F(q_i): i < i_*)$ are compatible.  Write $\Gamma = \bigcup_{i < i_*, n < \omega} \mbox{supp}(\mathbf{p}_{n}(q_i))$. 
	
	By induction $\alpha \leq \alpha_*$, we construct a greatest lower bound $s_\alpha$ to $(\mathbf{p}_{n}(q_i) \restriction_\alpha: i < i_*, n < \omega)$ in $P_\alpha$, such that $\mbox{supp}(s_\alpha) \subseteq \Gamma \cap \alpha$, and for $\alpha < \alpha'$, $s_{\alpha'} \restriction_\alpha = s_\alpha$. 
	
	Limit stages of the induction are clear. So suppose we have constructed $s_\alpha$. If $\alpha \not \in \Gamma$ clearly we can let $s_{\alpha+1} = s_\alpha \,^\frown (1^{\dot{Q}_\alpha})$; so suppose instead $\alpha \in \Gamma$. Let $n < \omega$ be given. Then $(r_{n, \alpha}(q_i): i< i_*)$ are compatible, and $s_\alpha$ forces that $\dot{F}_\alpha(\mathbf{p}_{n}(q_i)(\alpha)) = \check{r}_{n, \alpha}(\check{r}_i)$ for each $i < i_*$, since $\mathbf{p}_{n+1}(q_i) \restriction_\alpha$ does. Thus $s_\alpha$ forces that $(\mathbf{p}_n(q_i)(\alpha): i < i_*)$ has the greatest lower bound $\dot{s}_n$ in $\dot{Q}_\alpha$. Now $s_\alpha$ forces that $(\dot{s}_{n}: n < \omega)$ is a descending chain in $\dot{Q}_\alpha$, and hence has the greatest lower bound $\dot{s}$. Let $s_{\alpha+1} = s_\alpha \,^\frown(\dot{s})$.
	
	Thus the induction goes through, and $s_{\alpha_*}$ is a greatest lower bound to $(q_i: i < i_*)$ in $P_{\alpha_*}$.

\end{proof}

\begin{definition}
	Suppose $\theta$ is a regular cardinal and $Y$ is a set. Let $\mathbb{P}_{Y, \theta, \infty}$ be the class of all forcing notions of the form $P_{X Y \theta}$, for some set $X$. Suppose also $3 \leq k \leq \aleph_0$. If $\theta > \aleph_0$, then let $\mathbb{P}_{Y, \theta, k}$ be the class of all forcing notions which are strongly $(P_{XY\theta}, k)$-colorable for some set $X$. Also, let $\mathbb{P}_{Y, \aleph_0,k}$ be the class of all forcing notions which are $(P_{XY\aleph_0}, k)$-colorable for some set $X$.
\end{definition}

We will only be interested in the case where $Y = \mu$ is a cardinal with $\mu = \mu^{<\theta}$, but we formulate the definition as above because forcing notions in $\mathbb{P}_{\mu, \theta, k}$ will typically collapse $\mu$, and we still wish to speak of $\mathbb{P}_{\mu, \theta, k}$ in the forcing extension.

In \cite{DividingLine} and the sequel \cite{Optimals}, Malliaris and Shelah obtain dividing lines in Keisler's order by constructing sufficiently generic ultrafilters on the Boolean-algebra completion of $P_{2^\lambda \mu \theta}$ where $\mu = \theta^{<\theta}$. In other words they are working with $\mathbb{P}_{\mu, \theta, \infty}$. In order to detect various amalgamation properties of theories, they vary the target level of saturation $\mu^+ \leq \lambda \leq \mu^{+\omega}$. Working with  $\mathbb{P}_{\mu, \theta, k}$ allows us to obtain sharper model-theoretic results.

The following are some key properties of $\mathbb{P}_{\mu, \theta, k}$.
\begin{theorem}\label{IterationPreserves}
	Suppose $\theta$ is a regular cardinal and $\mu = \mu^{<\theta}$.
	\begin{itemize}
		\item[(A)] For every $P \in \mathbb{P}_{\mu, \theta, k}$, $P$ is $\theta$-closed and has the $\mu^+$-c.c.
		\item[(B)] Suppose $P, Q \in \mathbb{P}_{\mu, \theta, k}$. Then $P$ forces that $\check{Q} \in \mathbb{P}_{\mu, \theta, k}^{\mathbb{V}[\dot{G}]}$.
		\item[(C)] Suppose $(P_\alpha: \alpha \leq \alpha_*), (\dot{Q}_\alpha: \alpha < \alpha_*)$ is a $<\theta$-support forcing iteration, such that each $P_\alpha$ forces $\dot{Q}_\alpha \in \mathbb{P}_{\mu, \theta, k}^{\mathbb{V}[\dot{G}_\alpha]}$, where $\dot{G}_\alpha$ is the $P_\alpha$-generic filter. Then $P_{\alpha_*} \in \mathbb{P}_{\mu, \theta, k}$. 
		\item[(D)] $\mathbb{P}_{\mu, \theta, k}$ is closed under $<\theta$-support products.
	\end{itemize}
\end{theorem}
\begin{proof}
	
	(A): $P$ is $\theta$-closed by definition of $\mathbb{P}_{\mu, \theta, k}$. For the $\mu^+$-c.c.: we can find some $R := P_{X \mu \theta}$, and some $(R, 3)$-coloring $F: P \to R$. Now $R$ has the $\mu^+$-c.c. by the $\Delta$-system lemma, so it immediately follows that $P$ does as well.
	
	(B): If $F: (Q, Q_0) \to R$ is a strong $(R, k)$-coloring, then this will continue to work in any forcing extension by $P$, because $k \leq \theta$ and $P$ is $\theta$-closed.
	
	(C): We want to apply Theorem~\ref{IterationPreserves0} or Theorem~\ref{IterationPreserves1} (depending on whether $\theta = \aleph_0$). To do so, we need to find some fixed $R \in \mathbb{P}_{\theta, \mu, \infty}$ such that each $P_\alpha$ forces $\dot{Q}_\alpha$ is (strongly) $(R, k)$-colorable. Take $R = P_{\lambda \mu \theta}$ for $\lambda$ large enough.
	
	(D) follows immediately from (B) and (C) (and also using that each $P \in \mathbb{P}_{\mu \theta k}$ is $\theta$-closed, so that $<\theta$-support products are equivalent to the corresponding $<\theta$-support iterations).
\end{proof}

\section{The Ultrafilter Constructions}\label{KeislerNewUltCons}

In this section, we give a streamlined construction of the perfect and optimal ultrafilters of Malliaris and Shelah from \cite{DividingLine}, \cite{Optimals}.

Recall that $\mathbf{s}= (\lambda,\mu, \theta, \tau, k)$ is a suitable sequence of cardinals if:

\begin{itemize}
	\item $\tau \leq \theta \leq \mu = \mu^{<\theta} < \lambda$;
	\item $\theta$ is regular, and $\tau$ is either $\aleph_0$ or else supercompact;
	\item $3 \leq k \leq \aleph_0$.
\end{itemize}

Suppose $\mathbf{s}$ is a suitable sequence. Our goal is to build a long forcing iteration $(P_{\alpha}: \alpha \leq \alpha_*, \dot{Q}_\alpha: \alpha < \alpha_*)$ from $\mathbb{P}_{\mu, \theta, k}$, then build a sufficiently generic $\tau$-complete ultrafilter on $\mathcal{B}(P_{\alpha_*})$, and then check which theories it $\lambda^+$-saturates. While we follow the general outline of Malliaris and Shelah, our treatment differs in two respects. First, Malliaris and Shelah use $\mathbb{P}_{\mu, \theta, \infty}$ rather than $\mathbb{P}_{\mu, \theta, k}$; our approach allows us to circumvent some ingenious but ad-hoc coding methods (e.g. ``collision detection") and obtain sharper bounds on our final dividing lines. Second, 
in our construction of $(P_\alpha: \alpha \leq \alpha_*)$, $(\dot{Q}_\alpha: \alpha < \alpha_*)$, we will be anticipating not only $(\lambda, T)$-{\L}{o}{\'s} maps in $\mathcal{B}(P_\alpha)$, but also entire ultrafilters on $\mathcal{B}(P_\alpha)$. The upshot is that we get a better handle on which theories our eventual ultrafilter will $\lambda^+$-saturate. On the other hand, we lose control over the length $\alpha_*$ of the forcing iteration. In view of Corollary~\ref{KeislerOrderCharacterization}, this is not actually important, since $\mathcal{B}(P_{\alpha_*})$ will have the $\lambda^+$-c.c (in fact, the $\lambda$-c.c.)

If $\mathbf{A}$ is an $I$-distribution in $\mathcal{B}$, then for $S \subseteq I$, it is convenient to define $\mathbf{A}(S) = \bigwedge_{s \in [S]^{<\aleph_0}} \mathbf{A}(s)$. We will only use this notation when $\mathbf{A}$ is in a $\tau$-complete ultrafilter $\mathcal{U}$ and $|S| < \tau$, in which case $\mathbf{A}(S)$ is also in $\mathcal{U}$. 

The following lemma describes the situation we will be interested in while building our generic ultrafilter $\mathcal{U}$ on $\mathcal{B}(P_{\alpha_*})$. 

\begin{lemma}\label{SolutionsLemma}
	Suppose $\mathbf{s} = (\lambda, \mu, \theta, \tau, k)$ is a suitable sequence, and $T$ is a complete countable first order theory. Suppose $P \in \mathbb{P}_{\mu,\theta, k}$, and $\mathcal{U}$ is a $\tau$-complete ultrafilter on $\mathcal{B}(P)$, and $\mathbf{A}$ is a $\lambda$-distribution in $\mathcal{U}$. Then the following are equivalent:
	
	\begin{itemize}
		
		\item[(A)] There is some $\dot{Q} \in \mathbb{P}_{\mu,\theta, k}^{\mathbb{V}[\dot{G}]}$ and some multiplicative refinement $\mathbf{B}$ of $\mathbf{A}$ in $\mathcal{B}(P*\dot{Q})$, such that for every $S \in [\lambda]^{<\tau}$, $\pi(\mathbf{B}(S)) \in \mathcal{U}$. (Here $\dot{G}$ is the name for the $P$-generic filter, and $\pi = \pi_{\mathcal{B}(P*\dot{Q}), \mathcal{B}(P)}: \mathcal{B}(P*\dot{Q}) \to \mathcal{B}(P)$ is the projection map. $\mathbb{P}^{\dot{V}[G]}_{\mu, \theta, k}$ is the class of $P$-names for elements of $\mathbb{P}_{\mu, \theta, k}$, as computed in the forcing extension.)
		\item[(B)] There is some $\dot{Q} \in \mathbb{P}_{\mu,\theta, k}^{\mathbb{V}[\dot{G}]}$ and some $\tau$-complete ultrafilter $\mathcal{V}$ on $\mathcal{B}(P*\dot{Q})$ extending $\mathcal{U}$, such that $\mathbf{A}$ has a multiplicative refinement $\mathbf{B}$ in $\mathcal{V}$.
		
	\end{itemize}
\end{lemma}
\begin{proof}
	(A) implies (B): for each $S \in [\lambda]^{<\tau}$, we have that $\pi(\mathbf{B}(S)) \in \mathcal{U}$. By Lemma~\ref{projectionsLemma1}(C), it follows that $\mathcal{U} \cup \{\mathbf{B}(S): S \in [\lambda]^{<\tau}\}$ generates a $\tau$-complete filter on $\mathcal{B}(P*\dot{Q})$. Since either $\tau =\aleph_0$ or else is supercompact (and in particular strongly compact), by Lemma~\ref{CompactExt} we can find a $\tau$-complete ultrafilter $\mathcal{V}$ on $\mathcal{B}(P*\dot{Q})$ extending $\mathcal{U}$ such that $\mathbf{B}$ is in $\mathcal{V}$.
	
	(B) implies (A): trivial.
\end{proof}

We turn the lemma into a definition.

\begin{definition}\label{DefSetAP}
	Suppose $\mathbf{s} = (\lambda, \mu, \theta, \tau, k)$ is a suitable sequence, and $T$ is a complete countable first order theory. Then say that $(P, \mathcal{U}, \mathbf{A})$ is a $(T, \mathbf{s})$-problem if $P \in \mathbb{P}_{\mu, \theta, k}$, and $\mathcal{U}$ is a $\tau$-complete ultrafilter on $\mathcal{B}(P)$, and $\mathbf{A}$ is a $(\lambda, T)$-{\L}o{\'s} map in $\mathcal{U}$. $(P*\dot{Q}, \mathbf{B})$ is an $\mathbf{s}$-solution to $(P, \mathcal{U}, \mathbf{A})$ if $\dot{Q} \in \mathbb{P}_{\mathbf{s}, k}^{\mathbb{V}[\dot{G}]}$, and $\mathbf{B}$ is a multiplicative refinement of $\mathbf{A}$ in $\mathcal{B}(P * \dot{Q})$, and for every $S \in [\lambda]^{<\tau}$, $\pi(\mathbf{B}(S))\in \mathcal{U}$. Say that $T$ has the $\mathbf{s}$-extension property if every $(T, \mathbf{s})$-problem has an $\mathbf{s}$-solution.
\end{definition}

The $\mathbf{s}$-extension property will be an upper bound for our eventual principal dividing line in Keisler's order. In fact, if $\tau = \aleph_0$ then it will exactly be a principal dividing line in Keisler's order---given a sufficiently generic iteration sequence $(P_\alpha: \alpha \leq \alpha_*), (\dot{Q}_\alpha: \alpha < \alpha_*)$, we will build our desired ultrafilter $\mathcal{U}$ on $\mathcal{B}(P_{\alpha_*})$ as a union of a chain of ultrafilters $\mathcal{U}_\alpha$ on $\mathcal{B}(P_\alpha)$, and the $\mathbf{s}$-extension property for $T$ will be exactly what we need to arrange a multiplicative refinement for a given $(\lambda, T)$-{\L}o{\'s} map at stage $\alpha$. 

For $\tau > \aleph_0$, we cannot construct $\tau$-complete ultrafilters so na\"{i}vely, and we will need the following technical strengthening for our lower bound. This is essentially the difference between perfect and optimal ultrafilters in \cite{Optimals}. 

\begin{definition}\label{DefSetAP2}
	Suppose $\mathbf{s}, T$ are as above. Suppose $(P, \mathcal{U}, \mathbf{A})$ is a $(T, \mathbf{s})$-problem. Then say that $(P*\dot{Q}, \mathbf{B})$ is a smooth $\mathbf{s}$-solution to $\mathbf{A}$ if it is an $\mathbf{s}$-solution to $\mathbf{A}$, and for each $S \in [\lambda]^{<\tau}$, $\pi(\mathbf{B}(S)) = \bigwedge_{s \in [S]^{<\aleph_0}} \pi(\mathbf{B}(s))$. Say that $T$ has the smooth $\mathbf{s}$-extension property if every $(T, \mathbf{s})$-problem has a smooth solution.
\end{definition}
Note that if $\tau = \aleph_0$, then every solution is smooth.

The following example explains why we require $\lambda > \mu$ in the definition of suitable sequence.

\begin{example}
	Suppose, in the definition of suitable sequences $\mathbf{s} =(\lambda, \mu, \theta, \tau, k)$, we had allowed $\lambda \leq \mu$. Then for any such suitable sequence $\mathbf{s}$, we would have that every theory $T$ has the smooth $\mathbf{s}$-extension property.
\end{example}
\begin{proof}
	Choose $P_0 \in \mathbf{s}$ with an antichain of size $\lambda$, and let $P_1$ be the $<\theta$-support product of $\tau$-many copies of $P_0$. Easily, $P_1$ has an antichain of size $\lambda^{<\tau}$. Let $(\mathbf{c}_S: S \in [\lambda]^{<\tau})$ be an antichain of $P_1$. 
	
	Suppose $(P, \mathcal{U}, \mathbf{A})$ is a $(T, \mathbf{s})$-problem for some $T$. Let $\dot{Q} = \check{P}_1$. For each $s \in [\lambda]^{<\aleph_0}$, let $\mathbf{B}(s) = \bigvee_{s \subseteq S \in [\lambda]^{<\tau}} (\mathbf{A}(S), \mathbf{c}_S) \in \mathcal{B}(P) \times P_1 \subseteq \mathcal{B}(P * \dot{Q})$. Then $\mathbf{B}$ is a multiplicative refinement of $\mathbf{A}$, and in fact $\mathbf{B}(S) = \bigvee_{S \subseteq S' \in [\lambda]^{<\tau}} (\mathbf{A}(S'), \mathbf{c}_{S'})$ for all $S \in [\lambda]^{<\tau}$. 
	
	We check that each $\pi(\mathbf{B}(S)) = \mathbf{A}(S)$, which shows that $(P*\dot{Q}, \mathbf{B})$ is a smooth solution to $(P, \mathcal{U}, \mathbf{A})$. We apply Lemma~\ref{projectionsLemma1}(C). It suffices to show that for all $\mathbf{a} \in \mathcal{B}(P)$, $\mathbf{a} \wedge \mathbf{A}(S)$ is nonzero if and only if $\mathbf{a} \wedge \mathbf{B}(S)$ is nonzero. Since $\mathbf{B}(s) \leq \mathbf{A}(s)$, the reverse direction is trivial, so suppose $\mathbf{a} \wedge \mathbf{A}(S)$ is nonzero. Then $(\mathbf{a} \wedge \mathbf{A}(S), \mathbf{c}_S) \leq \mathbf{B}(S)$ is nonzero, as desired.
	
\end{proof}

On the other hand, this never happens when $\lambda > \mu$. Indeed, we can establish at once the following baselines for the $\mathbf{s}$-extension properties:

\begin{theorem}\label{baseline}
	Suppose $\mathbf{s} = (\lambda, \mu, \theta, \tau, k)$ is a suitable sequence.
	\begin{itemize}
		\item[(A)] $T_{rf}$ fails the $\mathbf{s}$-extension property.
		\item[(B)] If $\tau = \aleph_0$, then $T_{nlow}$ fails the $\mathbf{s}$-extension property.
	\end{itemize}

Hence, if $\mathbf{T}$ is a dividing line in Keisler's order (i.e. a $\trianglelefteq$-downward closed set of complete countable theories) such that every $T \in \mathbf{T}$ has the $\mathbf{s}$-extension property, then $\mathbf{T}$ does not contain any unsimple theory; if $\tau = \aleph_0$, then $\mathbf{T}$ does not contain any nonlow theory.
\end{theorem}

\begin{proof}
	The final claim follows, since $T_{rf}$ is a $\trianglelefteq$-minimal unsimple theory and $T_{nlow}$ is a $\trianglelefteq$-minimal nonlow theory \cite{InterpOrders2Ulrich}.
	
	(A): Let $P \in \mathbb{P}_{\mu, \theta, k}$ have an antichain of size $\mu$ (or at least $\tau$). Then Theorem~\ref{SimpleNonSat} together with Lemma~\ref{CompactExt} give a $(T_{rf}, \mathbf{s})$-problem $(P, \mathcal{U}, \mathbf{A})$ with no $\mathbf{s}$-solution (namely, let $\mathcal{U}$ be any nonprincipal, $\tau$-complete ultrafilter on $\mathcal{B}(P)$, and let $\mathbf{A}$ be as given in Theorem~\ref{SimpleNonSat}. $\mathbf{A}$ will be a $(\lambda', T)$-{\L}o{\'s} map for $\lambda' = \mbox{c.c.}(\mathcal{B}(P)) \leq \lambda$; by Theorem~\ref{Compactness}, we can extend $\mathbf{A}$ to a $(\lambda, T)$-{\L}o{\'s} map.)
	
	(B): Similarly, by Theorem~\ref{LowNonSat}.
\end{proof}

And the following shows that we are doing is relevant. The case $\tau = \aleph_0$ is similar to the construction of perfect ultrafilters in \cite{Optimals}, and the case $\tau > \aleph_0$ is similar to the construction of optimal ultrafilters there.

\begin{theorem}\label{SetTheoryDividingLines}
	
	Suppose $\mathbf{s} = (\lambda, \mu, \theta, \tau, k)$ is a suitable sequence. Then for some $P \in \mathbb{P}_{\mu, \theta, k}$, there is an ultrafilter $\mathcal{U}$ on $\mathcal{B}(P)$ which $\lambda^+$-saturates every theory with the smooth $\mathbf{s}$-extension property, and does not $\lambda^+$-saturate any theory which fails the $\mathbf{s}$-extension property. 
\end{theorem}
\begin{proof}
	
	As a convenient abbreviation, say that $(P, \mathcal{U}, \mathbf{A})$ is a solvable $\mathbf{s}$-problem if it is a $(T, \mathbf{s})$-problem for some countable theory $T$, and it has a smooth $\mathbf{s}$-solution.
	
	Let $(T_\delta: \delta < 2^{\aleph_0})$ enumerate all complete countable first order theories which fail the $\mathbf{s}$-extension property. For each $\delta < \aleph_0$, let $(Q_{0, \delta},  \mathcal{U}_{0, \delta}, \mathbf{A}_{0, \delta})$ be a $(T_\delta, \mathbf{s})$-problem with no $\mathbf{s}$-solution. Let $Q_0$ be the $<\theta$-support product of $(Q_{0, \delta}: \delta < 2^{\aleph_0})$; so $Q_0 \in \mathbb{P}_{\theta, k}$. Let $\mathcal{V}_1$ be a $\tau$-complete ultrafilter on $\mathcal{B}(Q_0)$ extending each $\mathcal{U}_{0, \delta}$.  Easily, for each $\delta < 2^{\aleph_0}$, $(Q_0, \mathcal{V}_0, \mathbf{A}_{0, \delta})$ is a $(T_\delta, \mathbf{s})$-problem with no $\mathbf{s}$-solution (this is what we need, going forth). 
	
	The following setup is straightforward to arrange:
	
	\begin{enumerate}
		\item $\alpha_*$ is an ordinal, and $(\chi_\delta: \delta < \lambda^+)$ is a cofinal sequence of cardinals in $\alpha_*$ (so $\alpha_*$ is a limit cardinal of cofinality $\lambda^+$);
		\item $(P_\alpha: \alpha \leq \alpha_*)$, $(\dot{Q}_\alpha: \alpha < \alpha_*)$ is a $<\theta$-support forcing iteration, and each $P_\alpha$ forces that $\dot{Q}_\alpha \in \mathbb{P}_{\theta, k}$;
		\item For each $1 \leq \alpha < \alpha_*$: $\beta_\alpha \leq \alpha$ and $(P_{\beta_\alpha}, \mathcal{U}_\alpha, \mathbf{A}_\alpha)$ is a solvable $\mathbf{s}$-problem, with the smooth $\mathbf{s}$-solution $(P_{\beta_\alpha} * \dot{Q}_\alpha, \mathbf{B}_\alpha)$ (so in particular $\dot{Q}_\alpha$ is a $P_{\beta_\alpha}$-name);
		
		\item For each $1 \leq \delta < \lambda^+$, for each $\beta \leq \chi_\delta$, and for each solvable $\mathbf{s}$-problem $(P_\beta, \mathcal{U}, \mathbf{A})$, there is some $\alpha < \chi_{\delta+1}$ such that $\beta_{\alpha} = \beta$, $\mathcal{U}_{\alpha} = \mathcal{U}$ and $\mathbf{A}_{\alpha} = \mathbf{A}$.
	\end{enumerate}

	%
	%
	%
	%
	
	By Corollary~\ref{IterationPreserves}, each  $P_{\alpha} \in \mathbb{P}_{\mu, \theta, k}$, and hence is $\theta$-closed and has the $\mu^+$-c.c. It suffices to find a $\tau$-complete ultrafilter $\mathcal{V}$ on $\mathcal{B}(P_{\alpha_*})$ which $\lambda^+$-saturates every theory with the smooth $\mathbf{s}$-extension property, and no theory which fails the $\mathbf{s}$-extension property.
	
	We claim that it suffices to find a $\tau$-complete ultrafilter $\mathcal{V}$ on $\mathcal{B}(P_{\alpha_*})$ extending $\mathcal{V}_1$, such that for all $1 \leq \alpha < \alpha_*$, if $\mathcal{V}$ extends $\mathcal{U}_\alpha$ then $\mathbf{B}_\alpha$ is in $\mathcal{V}$ (i.e. $\mathbf{B}_\alpha(s) \in \mathcal{U}_\alpha$ for all $s \in [\lambda]^{<\aleph_0}$, or equivalently for all $S \in [\lambda]^{<\tau}$). Indeed, suppose $\mathcal{V}$ is given as such.
	
	First suppose $\delta < 2^{\aleph_0}$; then $\mathcal{V}$ does not $\lambda^+$-saturate $T_\delta$, since if $\mathbf{B}$ were a multiplicative refinement of $\mathbf{A}_{0, \delta}$ in $\mathcal{V}$, then $(\dot{P}, \mathbf{B})$ would be an $\mathbf{s}$-solution to $(Q_0, \mathcal{V}_0, \mathbf{A}_{0, \delta})$, where $\dot{P}$ is the natural $Q_0$-name such that $\mathcal{B}(P_{\alpha_*}) \cong \mathcal{B}(Q_0 * \dot{P})$.

	Next, we show that $\mathcal{V}$ $\lambda^+$-saturates every $T$ with the smooth $\mathbf{s}$-extension property. Indeed, suppose $\mathbf{A}$ is a $(T, \mathcal{B}(P_{\alpha_*}), \lambda)$-{\L}o{\'s} map in $\mathcal{V}$. Since $\mbox{cof}(\alpha_*) = \lambda^+$ and since $\mathcal{B}(P_{\alpha_*})$ has the $\lambda^+$-c.c., we have that $\mathcal{B}(P_{\alpha_*}) = \bigcup_{\alpha < \alpha_*} \mathcal{B}(P_\alpha)$, and so $\mathbf{A}$ is in $\mathcal{B}(P_\beta)$ for some $\beta < \alpha_*$. Let $\mathcal{U} = \mathcal{V} \cap \mathcal{B}(P_\beta)$. Since $T$ has the smooth $\mathbf{s}$-extension property, $(P_\beta, \mathcal{U}, \mathbf{A})$ is a solvable $\mathbf{s}$-problem. Choose $\delta < \lambda^+$ with $\chi_\delta \geq \beta$. Choose $\alpha < \chi_{\delta+1}$ such that $\beta_\alpha = \beta$, $\mathcal{U}_\alpha = \mathcal{U}$ and $\mathbf{A}_\alpha = \mathbf{A}$. Since $\mathcal{V}$ extends $\mathcal{U}_\alpha$ we must have that $\mathcal{V}$ extends $\mathbf{B}_\alpha$, but $\mathbf{B}_\alpha$ is a multiplicative refinement to $\mathbf{A}$ so we are done.
	
	So it remains to find $\mathcal{V}$. If $\tau = \aleph_0$ then this is fairly trivial; having constructed $\mathcal{V} \cap \mathcal{B}(P_\alpha)$, if $\mathcal{V} \cap \mathcal{B}(P_\alpha)$ extends $\mathcal{U}_\alpha$, then note that for all $s \in [\lambda]^{<\aleph_0}$, $\pi_{\mathcal{B}(P_{\alpha+1}), \mathcal{B}(P_{\beta_\alpha})}(\mathbf{B}_\alpha(s)) \in \mathcal{U}_\alpha$, and so $(\mathcal{V} \cap \mathcal{B}(P_\alpha)) \cup \{\mathbf{B}(s): s \in [\lambda]^{<\aleph_0}\}$ has the finite intersection property. So we can find $\mathcal{V}_{\alpha+1}$. 
	
	Finally, suppose $\tau > \aleph_0$. We cannot adopt a straightforward construction as above, since we cannot preserve $\tau$-completeness through limit stages. Thus we take a different approach. The remainder of the argument mirrors Theorem 5.9 of \cite{Optimals}.
	
	Let $\mathcal{E}$ be a normal, $\tau$-complete ultrafilter on $[H(\chi)]^{<\tau}$ where $\chi$ is large enough, and where $H(\chi)$ is the set of sets of hereditary cardinality less than $\chi$. Let $\Omega$ be the set of all $N \in [H(\chi)]^{<\tau}$ such that $N \preceq (H(\chi), \in, \ldots)$ where $\ldots$ is a list of the finitely many relevant parameters we have mentioned in the proof so far. Note $\Omega \in \mathcal{E}$.
	
	\vspace{1 mm}
	
	\noindent \textbf{Claim.} Suppose $N \in \Omega$. Then there is some $p_N \in P_{\alpha_*}$ such that $p_N \leq \bigwedge (\mathcal{V}_1 \cap N)$, and $p_N$ decides every element of $P_{\alpha_* \cap N}$, and for all $\alpha \in \alpha_* \cap N$, either $p_N \leq \mathbf{B}_\alpha(\lambda \cap N)$, or else $p_N$ contradicts some element of $\mathcal{U}_\alpha \cap N$. 
	
	\vspace{1 mm}
	
	\noindent \emph{Proof.} Let $(\alpha_\gamma: \gamma < \gamma_*)$ enumerate $N \cap \alpha_*$ in the increasing order, so $\gamma_* < \tau$. By induction on $\gamma < \gamma_*$ we construct $(p_{\gamma}: \gamma < \gamma_*)$ with each $p_\gamma \in P_{\alpha_\gamma}$, such that:
	
	\begin{itemize}
		\item For $\gamma < \gamma'$, $p_\gamma \geq p_{\gamma}'$;
		\item For each $\gamma < \gamma_*$, and for each $\alpha \in N \cap \alpha_*$, $p_\gamma$ decides every element of $\mathcal{B}(P_{\alpha_\gamma}) \cap N$;
		\item $p_1 \leq \bigwedge (\mathcal{V}_1 \cap N)$;
		
		\item If $p_\gamma \leq \bigwedge (\mathcal{U}_{\alpha_\gamma} \cap N)$ then $p_{\gamma+1} \leq \mathbf{B}_{\alpha_\gamma}(\lambda \cap N)$.
	\end{itemize}
	
	The base case is easy. If $\delta < \gamma_*$ is a limit ordinal, then when constructing $p_\delta$ we just need to handle the first and second conditions. We can do this because $P_{\alpha_\delta}$ is $\theta$-closed, and hence $\tau$-closed.
	
	The key point is the following. Suppose $p_{\gamma}$ is defined; write $\alpha = \alpha_\gamma$. Suppose $p_{\gamma} \leq \bigwedge ( \mathcal{U}_{\alpha} \cap N)$. We need to show that $p_{\gamma} \wedge \mathbf{B}_\alpha(\lambda \cap N)$ is nonzero. Write $S = \lambda \cap N$, and let $\pi = \pi_{\mathcal{B}(P_{{\alpha+1}}), \mathcal{B}(P_{\beta_\alpha})}$ be the projection map. It suffices to show that $p_{\gamma} \leq \pi(\mathbf{B}_\alpha(S))$. Now, $\mathbf{B}_\alpha(S) = \bigwedge_{s \in [S]^{<\aleph_0}} \mathbf{B}_\alpha(s)$, and each $\pi(\mathbf{B}_\alpha(s)) \in \mathcal{U}_\alpha$, by definition of smooth solution. But then each $\pi(\mathbf{B}_\alpha(s)) \in \mathcal{U}_\alpha \cap N$, since $[S]^{<\aleph_0} \subseteq N$. Thus $p_\gamma \leq \bigwedge_{s \in [S]^{<\aleph_0}} \pi(\mathbf{B}_\alpha(s)) = \pi(\mathbf{B}_\alpha(S))$, as desired. The other conditions are easy to arrange, since $P_{\alpha_{\gamma+1}}$ is $\tau$-closed.
	
	Let $p_N \in P_{\alpha_*}$ be a lower bound to $(p_\gamma: \gamma < \gamma_*)$. Then clearly this works, and so we have proven the claim.
	
	\vspace{1 mm}
	
	Fix such a $p_N$ for every $N \in \Omega$. Define $\mathcal{V}$ to be the set of all $\mathbf{a} \in \mathcal{B}(P_{\alpha_*})$ such that $\{N \in \Omega: p_N \leq \mathbf{a}\} \in \mathcal{E}$. I claim that $\mathcal{V}$ is as desired. $\mathcal{V}$ is obviously a filter. Given $\mathbf{a} \in \mathcal{B}(P_{\alpha_*})$, we have that $\{N \in \Omega: \mathbf{a} \in N\} \in \mathcal{E}$ since $\mathcal{E}$ is fine, thus $\mathcal{V}$ is an ultrafilter. Since $\mathcal{E}$ is $\tau$-complete, so is $\mathcal{V}$.
	
	Finally, suppose $\alpha < \alpha_*$; we need to show that either $\mathbf{B}_\alpha$ is in $\mathcal{V}$ or else $\mathcal{V}$ does not extend $\mathcal{U}_\alpha$. Let $C_1:= \{N \in \Omega: p_N \leq \bigwedge \mathcal{U}_\alpha \cap N \}$, and let $C_2 = \{N \in \Omega: p_N \not \leq \bigwedge \mathcal{U}_\alpha \cap N\}$. Either $C_1 \in \mathcal{E}$ or else $C_2 \in \mathcal{E}$. Suppose first that $C_1 \in \mathcal{E}$. Then for each $N \in C_1$ and for each $s \in [\lambda]^{<\aleph_0} \cap N$,  $p_N \leq \mathbf{B}_\alpha(s)$; since $\mathcal{E}$ is fine, it follows that each $\mathbf{B}_\alpha(s) \in \mathcal{V}$, so $\mathbf{B}$ is in $\mathcal{V}$. Next, suppose $C_2 \in \mathcal{E}$; for each $N \in C_2$, we can choose $f(N) \in \mathcal{U}_\alpha \cap N$ such that $p_N \not \leq f(N)$. Since $\mathcal{E}$ is normal, we can find $C \subseteq C_2$ with $C \in \mathcal{E}$, such that $f$ is constant on $C$, say with value $\mathbf{a} \in \mathcal{U}_\alpha$. Then $\mathbf{a} \not \in \mathcal{V}$, so $\mathcal{V}$ does not extend $\mathcal{U}_\alpha$. Thus, in either case, either $\mathbf{B}_\alpha$ is in $\mathcal{V}$ or else $\mathcal{V}$ does not extend $\mathcal{U}_\alpha$.
\end{proof}

\begin{corollary}\label{SetTheoryDividingLinesCor}
	Suppose $\mathbf{s} = (\lambda, \theta, \tau, k)$ is a suitable sequence.
	Then there is a principal dividing line in Keisler's order between the $\mathbf{s}$-extension property and the smooth $\mathbf{s}$-extension property. If $\tau = \aleph_0$ then the $\mathbf{s}$-extension property is itself a principal dividing line.
\end{corollary}
\begin{proof}
	This follows from Theorem~\ref{SetTheoryDividingLines} since for every $P \in \mathbb{P}_{\mu \theta k}$, $\mathcal{B}(P)$ has the $\mu^+$-c.c.
\end{proof}

\section{The $\mathbf{s}$-extension property}\label{KeislerNewSat}
In this section, we give sufficient and necessary conditions (separately) for $T$ to have the (smooth) $\mathbf{s}$-extension property. 

The following is the necessary condition; note the similarity of proof to Theorem~\ref{NonSat2}.

\begin{theorem}\label{NonSatPortionGeneral}
	Suppose $\mathbf{s} = (\lambda, \mu, \theta, \tau, k_*)$ is a suitable sequence. Suppose $2 \leq k < k_*$, and $\lambda$ is big enough for $(\mu, \theta, k)$. If $T$ admits $\Delta_{k+1, k}$, then $T$ fails the $\mathbf{s}$-extension property.
\end{theorem}

\begin{proof}
	Let $P = P_{[\lambda]^k \mu \theta} \in \mathbb{P}_{\mu, \theta, k_*}$. Write $\mathcal{B} = \mathcal{B}(P)$; for each $v \in [\lambda]^k$, write $\mathbf{c}_v = \{(v, 0)\} \in P \subseteq \mathcal{B}$. Let $\mathbf{A}$ be the $[\lambda]^{k-1}$-distribution in $\mathcal{B}$, defined by putting $\mathbf{A}(s) = \bigwedge\{\mathbf{c}_v: v \in [\lambda]^k, [v]^{k-1} \subseteq s\}$. Easily, $\mathbf{A}$ is a $([\lambda]^{k-1}, T)$-{\L}o{\'s} map.

	Suppose $\dot{Q} \in \mathbb{P}_{\mu, \theta, k_*}^{\mathbb{V}[\dot{G}]}$ is given; it suffices to show that $\mathbf{A}$ has no multiplicative refinement in $\mathcal{B}(P*\dot{Q})$. Suppose towards a contradiction that there were, say $\mathbf{B}$. For a large enough set $X$, we can find a $\mathcal{B}$-name $\dot{F}$ such that $\mathcal{B}$ forces $\dot{F}: \dot{Q} \to \check{P}_{\check{X} \mu \theta}$ is a $(\check{P}_{\check{X} \mu \theta}, k_*)$-coloring. Write $R = P_{X \mu \theta}$. 
	
	For each $v \in [\lambda]^{<\aleph_0}$ choose $(p_v, \dot{q}_v) \in P * \dot{Q}$ such that $(p_v, \dot{q}_v) \leq \mathbf{B}([v]^{k-1})$ and $p_v$ decides $\dot{F}(\dot{q}_v)$, say $p_v$ forces that $\dot{F}(\dot{q}_v) = \check{f}_v$ for some $f_v \in R$. 
	
	Define $F: [\lambda]^{<\aleph_0} \to [\lambda]^{<\theta}$ via $F(v)= v \cup \bigcup \mbox{dom}(p_v)$. Write $Y = X \cup [\lambda]^k$, and define $G: [\lambda]^{<\aleph_0} \to P_{Y \mu \theta}$ via $G(v) = p_v \cup f_v$. Since $\lambda$ is big enough for $(\mu, \theta, k)$, we can find some $v \in [\lambda]^k$ and some $(w_u: u \in [v]^{k-1}$) from $[\lambda]^{<\aleph_0}$, such that each $w_u \cap v = F(w_u) \cap v = u$, and such that $\bigcup \{G(w_u): u \in [v]^{k-1}\}$ is a function.
	
	Write $p= \bigcup_{u \in [v]^{k-1}} p_{w_u} \in P$. Then $p$ forces that each $\dot{F}(\dot{q}_{w_u}) = \hat{f}_{w_u}$; since $\bigcup \{f_{w_u}: u \in [v]^{k-1}\}$ is a function, we get that $p$ forces that $\{\dot{q}_{w_u}: u \in [v]^{k-1}\}$ are compatible in $\dot{Q}$, say with lower bound $\dot{q}$. So $(p, \dot{q}) \in P*\dot{q}$ is a lower bound to $((p_{w_u}, \dot{q}_{w_u}): u \in [v]^{k-1})$. Note that $v \not \in \mbox{dom}(p)$, since if $v \in \mbox{dom}(p_{w_u})$ say, then $v \subseteq \bigcup \mbox{dom}(p_{w_u})$, contradicting that $F(w_u) \cap v = u$. Thus we can choose $p' \leq p$ in $P$ with $p'(v) = 1$; note than $(p', \dot{q}) \in P * \dot{Q}$.
	
	Now for each $u \in [v]^{k-1}$, $(p', \dot{q}) \leq (p_{w_u}, \dot{q}_{w_u}) \leq \mathbf{B}([w_u]^{k-1}) \leq \mathbf{B}(\{u\})$. Thus, by multiplicativity, $(p', \dot{q}) \leq \mathbf{B}([v]^{k-1}) \leq \mathbf{A}([v]^{k-1}) = \{(v, 0)\}$, contradicting the choice of $p'$.
\end{proof}

For the sufficient conditions, we relate the $\mathbf{s}$-extension property to the $(\lambda, \mu, \theta, k)$-coloring property. This will take some preparation.

Recall that by a theorem of Kim \cite{KimForking}, in any simple theory $T$, forking is the same as dividing; that is, $\phi(\overline{x}, \overline{a})$ forks over $A$ if and only if it divides over $A$. We thus use the terms \emph{forking} and \emph{dividing} interchangeably. Also, we define low:

\begin{definition} The complete countable theory $T$ has the {finite dividing property} if there is some formula $\phi(\overline{x}, \overline{y})$ such that for every $k$ there is some indiscernible sequence $(\overline{b}_n: n < \omega)$ over the emptyset such that $\{\phi(\overline{x}, \overline{b}_n): n < \omega\}$ is $k$-consistent but not consistent. $T$ is low if it is simple and does not have the finite dividing property.
\end{definition}
This is equivalent to the original definition of Buechler \cite{Buechler}; however, some authors define low to mean ``not the finite dividing property," for instance this is the definition in \cite{DividingLine}.

In \cite{KeislerNotLinear}, we observed the following equivalent of $T$ being low. 

\begin{theorem}\label{LowEquiv}
	Suppose $T$ is simple. Then the following are equivalent:
	
	\begin{itemize}
		\item[(A)] $T$ is low.
		\item[(B)] Suppose $\phi(\overline{x}, \overline{b})$ does not fork over $A$. Then there is some $\overline{c} \in A$ and some $\psi(\overline{y}, \overline{z}) \in \mbox{tp}(\overline{b}, \overline{c})$ such that whenever $(\overline{b}', \overline{c}') \models \psi(\overline{y}, \overline{z})$, then $\phi(\overline{x}, \overline{b}')$ does not fork over $\overline{c}'$.
	\end{itemize}
\end{theorem}

We now make a couple of observations about forking in Boolean-valued models.

\begin{definition}
	
	Suppose $T$ is a countable complete theory, and $\mathbf{M} \models^{\mathcal{B}} T$ (recall this means $\mathbf{M}$ is a full $\mathcal{B}$-valued model of $T$). Then $p(\overline{x})$ is a partial type over $\mathbf{M}$ if $p(\overline{x}) \subseteq \mathcal{L}(\mbox{range}(\overline{x}) \cup \mathbf{M})$, and for every finite $\Gamma(\overline{x}) \subseteq p(\overline{x})$, $\|\exists \overline{x} \bigwedge \Gamma(\overline{x})\|_{\mathbf{M}} > 0$. If $\mathcal{U}$ is an ultrafilter on $\mathcal{B}$, then let $[p(\overline{x})]_{\mathcal{U}}$ be the image of $p(\overline{x})$ under $[\cdot]_{\mathcal{U}}$; so this is a set of formulas over $\mathbf{M}/\mathcal{U}$. Say that $p(\overline{x})$ is a $\mathcal{U}$-type over $\mathbf{M}$ if $[p(\overline{x})]_{\mathcal{U}}$ is consistent. If $T$ is simple and $A \subseteq \mathbf{M}$, then say that $p(\overline{x})$ does not $\mathcal{U}$-fork over $A$ if $[p(\overline{x})]_{\mathcal{U}}$ does not fork over $[A]_{\mathcal{U}}$.
	
	Suppose $\mathbf{M} \models^{\mathcal{B}} T$. Let $\mathbb{V}[G]$ be a forcing extension by $\mathcal{B}_+$. Then $G$ is an ultrafilter on $\mathcal{B}$ in $\mathbb{V}[G]$; now $\mathcal{B}$ is typically not complete in $\mathbb{V}[G]$, but the definition of specializations did not require completeness, and so we can still form the specialization $(\mathbf{M}/G, [\cdot]_G)$. Thus, in $\mathbb{V}$, $\check{\mathbf{M}}/\dot{G}$ is a $\mathcal{B}$-name for a model of $T$, and  $[\cdot]_{\dot{G}}$ is a name for a surjection $\check{\mathbf{M}} \to \check{\mathbf{M}}/\dot{G}$. We have that for every $\phi(\overline{a}) \in \mathcal{L}(\mathbf{M})$, $\|\phi(\overline{a})\|_{\mathbf{M}} = \|\check{\mathbf{M}}/\dot{G} \models \phi([\overline{a}]_{\dot{G}}) \|_{\mathcal{B}}$. We call $(\check{\mathbf{M}}/\dot{G}, [\cdot]_{\dot{G}})$ the generic specialization of $\mathbf{M}$.
\end{definition}

Then the following lemmas are straightforward:

\begin{lemma}\label{CruxForLow}
	Suppose $\mathcal{B}$ is a complete Boolean algebra, and $T$ is low. Suppose $\mathbf{M} \models^{\mathcal{B}} T$ (i.e. $\mathbf{M}$ is a full $\mathcal{B}$-valued model of $T$), $\mathbf{M}_0 \preceq \mathbf{M}$ is countable, and $\phi(\overline{x})$ is a formula over $\mathbf{M}$. Suppose $\mathcal{U}$ is an ultrafilter on $\mathcal{B}$. If $\phi(\overline{x})$ does not $\mathcal{U}$-fork over $\mathbf{M}_0$, then $\|\phi(\overline{x}) \mbox{ does not fork over } \check{\mathbf{M}}_0/\dot{G} \mbox{ in } \check{\mathbf{M}}/\dot{G}\|_{\mathcal{B}} \in \mathcal{U}$.
	
\end{lemma}
\begin{proof}
	
	Let $\mathcal{U}$ be an ultrafilter on $\mathcal{B}$ such that $\phi(x)$ does not $\mathcal{U}$-fork over $\mathbf{M}_0$. Suppose $\phi(\overline{x})$ is over $\overline{a} \in \mathbf{M}^{<\omega}$. Choose $\psi_\phi(\overline{a}, \overline{a}_0)$ such that $\overline{a}_0 \in \mathbf{M}_0$ and $\mathbf{M}/\mathcal{U} \models \psi_\phi([\overline{a}]_{\mathcal{U}}, [\overline{a}_0]_{\mathcal{U}})$, such that whenever $M \models T \land \psi_\phi(\overline{b}, \overline{b}_0)$, then $\phi(\overline{x}, \overline{b})$ does not fork over $\overline{b}_0$. Put $\mathbf{c} = \|\psi_\phi(\overline{a}, \overline{a}_0)\|_{\mathbf{M}}$. Then $\mathbf{c} \in \mathcal{U}$, and clearly $\mathbf{c} \leq \|\phi(\overline{x}) \mbox{ does not fork over } \check{\mathbf{M}}_0/\dot{G}\mbox{ in } \check{\mathbf{M}}/\dot{G}\|_{\mathcal{B}}$.
\end{proof}

This is false for nonlow theories; in general, we need to restrict to $\aleph_1$-complete ultrafilters. 

\begin{lemma}\label{CruxForSimple}
	Suppose $\mathcal{B}$ is a complete Boolean algebra, and $T$ is simple. Suppose $\mathbf{M} \models^{\mathcal{B}} T$, $\mathbf{M}_0 \preceq \mathbf{M}$ is countable, and $\phi(x)$ is a formula over $\mathbf{M}$. Suppose $\mathcal{U}$ is an $\aleph_1$-complete ultrafilter on $\mathcal{B}$. If $\phi(\overline{x})$ does not $\mathcal{U}$-fork over $\mathbf{M}_0$, then $\|\phi(\overline{x}) $ does not fork over $\check{\mathbf{M}}_0/\dot{G}\mbox{ in } \check{\mathbf{M}}/\dot{G}\|_{\mathcal{B}} \in \mathcal{U}$.
	
\end{lemma}
\begin{proof}
	
	Let $\mathcal{U}$ be an $\aleph_1$-complete ultrafilter on $\mathcal{B}$ such that $\phi(\overline{x})$ does not $\mathcal{U}$-fork over $\mathbf{M}_0$. Suppose $\phi(\overline{x})$ is over $\overline{a} \in \mathbf{M}^{<\omega}$. Let  $\mathbf{c} =  \bigwedge \{ \|\psi(\overline{a}, \overline{a}_0)\|_{\mathbf{M}}: \overline{a}_0 \in \mathbf{M}_0 \mbox{ and } \mathbf{M}/\mathcal{U} \models \psi([\overline{a}]_{\mathcal{U}}, [\overline{a}_0]_{\mathcal{U}})\}$. Then $\mathbf{c} \in \mathcal{U}$, and clearly $\mathbf{c} \leq \|\phi(\overline{x}) \mbox{ does not fork over } \check{\mathbf{M}}_0/\dot{G}\mbox{ in } \check{\mathbf{M}}/\dot{G}\|_{\mathcal{B}}$.
\end{proof}

We use these lemmas to obtain the following:

\begin{theorem}\label{satPortionGeneral1}
	Suppose $\mathbf{s} = (\lambda, \mu, \theta, \tau, k)$ is suitable. Suppose $T$ is simple, and either $\tau > \aleph_0$ or else $T$ is low. Suppose that every  $P \in \mathbb{P}_{\mu, \theta, k}$ forces that $T$ has the $(\lambda, |\mu|, \theta, k)$-coloring property (we write $|\mu|$ because $P$ might collapse $\mu$, in fact typically $P$ collapses $|\mu|$ to $\theta$). Then $T$ has the smooth $\mathbf{s}$-extension property.
\end{theorem}
\begin{proof}
	Let $(P, \mathcal{U}, \mathbf{A})$ be a $(T,\mathbf{s})$-problem.
	
	By definition of a {\L}o{\'s}-map, we can choose $\mathbf{M} \models^{\mathcal{B}} T$, and a partial type $p(x) = \{\phi_\alpha(x, \overline{a}_\alpha): \alpha < \lambda\}$ over $\mathbf{M}$ (where $x$ might be a tuple), such that for all $s \in [\lambda]^{<\aleph_0}$, $\|\exists x \bigwedge_{\alpha \in s} \phi_\alpha(x, \overline{a}_\alpha)\|_{\mathbf{M}} = \mathbf{A}(s)$; we can arrange $|\mathbf{M}| \leq \lambda$. Choose $\mathbf{M}_0 \preceq \mathbf{M}$ countable such that $p(x)$ does not $\mathcal{U}$-fork over $\mathbf{M}_0$.

	Let $\dot{Q}= \Gamma^\theta_{\check{\mathbf{M}}/\dot{G}, \check{\mathbf{M}}_0/\dot{G}}$. Since $P$ forces that $\dot{Q}$ is $(|\mu|, \theta, k)$-colorable and has the greatest lower bounds property, we get that $\dot{Q} \in \mathbb{P}_{\mu, \theta, k}^{\mathbb{V}[\dot{G}]}$. For each $s \in [\lambda]^{<\aleph_0}$, define
	
	$$\mathbf{C}(s) =\bigwedge_{\alpha \in s} \| \phi_\alpha(x, [\overline{a}_\alpha]_{\dot{G}}) \mbox{ does not fork over } \check{\mathbf{M}}_0/\dot{G} \mbox{ in } \check{\mathbf{M}}/\dot{G}\|_{\mathcal{B}}$$
	
	so $\mathbf{C}(s) \in \mathcal{U}$, by Lemma~\ref{CruxForLow} or else Lemma~\ref{CruxForSimple} depending on whether $T$ is low or else $\tau > \aleph_0$.
	
	For each $s \in [\lambda]^{<\aleph_0}$, let $\mathbf{B}(s) = (\mathbf{C}(s), \{\phi_\alpha(x, [\overline{a}_\alpha]_{\dot{G}}): \alpha \in s\}) \in \mathcal{B}(P) * \dot{Q} \subseteq \mathcal{B}(P * \dot{Q})$. 
	
	Clearly $\mathbf{B}$ is a multiplicative refinement of $\mathbf{A}$. Moreover, letting $\pi: \mathcal{B}(P * \dot{Q}) \to \mathcal{B}(P)$ be the projection, note that whenever $(\mathbf{c}, \dot{p}(x)) \in \mathcal{B}(P) * \dot{Q}$, we have that $\pi(\mathbf{c}, \dot{p}(x)) = \mathbf{c}$. Hence $\pi(\mathbf{B}(s)) = \mathbf{C}(s)$ for each $s \in [\lambda]^{<\aleph_0}$. But moreover, given $S \in [\lambda]^{<\tau}$, $\mathbf{B}(S) = \bigwedge_{s \in [S]^{<\aleph_0}} \mathbf{B}(s)$, so $\pi(\mathbf{B}(S)) = \mathbf{C}(S) = \bigwedge_{s \in [S]^{<\aleph_0}} \mathbf{C}(s) = \bigwedge_{s \in [S]^{<\aleph_0}} \pi(\mathbf{B}(s))$ as desired. 
	
	Thus $(\dot{Q}, \mathbf{B})$ is a smooth $\mathbf{s}$-solution to $(P, \mathcal{U}, \mathbf{A})$.
\end{proof}

\begin{corollary}\label{SatFinal}
	Suppose $\mathbf{s} = (\lambda, \mu, \theta, \tau, k)$ is suitable, with $\theta  > \aleph_0$. Suppose $T$ is simple, and either $\tau > \aleph_0$ or else $T$ is low. If $T$ has $<k$-type amalgamation, then $T$ has the smooth $\mathbf{s}$-extension property.
\end{corollary}
\begin{proof}
	By Theorems~\ref{SatPortion} and \ref{satPortionGeneral1}.
\end{proof}

\section{Conclusion}\label{KeislerNewConc}

We summarize our results, and state some further conjectures. We begin with several convenient definitions.

\begin{definition}
Suppose $\mathbf{s} = (\lambda, \mu, \theta, \tau, k)$ is a suitable sequence. By Theorem~\ref{SetTheoryDividingLines}, we can find a complete Boolean algebra $\mathcal{B}_{\mathbf{s}}$ with the $\lambda$-c.c. and an ultrafilter $\mathcal{U}_{\mathbf{s}}$ on $\mathcal{B}_{\mathbf{s}}$ such that $\mathcal{U}_{\mathbf{s}}$ $\lambda^+$-saturates every theory with the smooth $\mathbf{s}$-extension property, and does not $\lambda^+$-saturate any theory without the $\mathbf{s}$-extension property. 

Say that $\mathbf{s}$ is of type I if $\tau > \aleph_0$. Say that $\mathbf{s}$ is of type II if $\tau = \aleph_0$ and $\theta > \aleph_0$. Say that $\mathbf{s}$ is of type III if $\tau  = \theta = \aleph_0$ and $k < \aleph_0$. Finally, say that $\mathbf{s}$ is of type IV if $\tau = \theta = k = \aleph_0$. 

Let $k_*(\mathbf{s}) = \mbox{sup}\{k'+1: 2 \leq k' < k \mbox{ and $\lambda$ is big enough for } (\mu, \theta, k')\}$, so $3 \leq k_*(\mathbf{s}) \leq k$. 
\end{definition}

Recall that if $\tau = \aleph_0$, then the smooth $\mathbf{s}$-extension property is the same as the $\mathbf{s}$-extension property, and hence $\mathcal{U}_{\mathbf{s}}$ $\lambda^+$-saturates $T$ if and only if $T$ has the $\mathbf{s}$-extension property. Also, by Theorem~\ref{ExistenceThm}, the set of complete countable theories which are $\lambda^+$-saturated by $\mathcal{U}_{\mathbf{s}}$ form a principal dividing line in Keisler's order.

\begin{theorem}\label{baseline2}
Suppose $\mathbf{s} = (\lambda, \mu, \theta, \tau, k)$ is a suitable sequence, and suppose $T$ is a complete countable theory. If $\mathcal{U}_{\mathbf{s}}$ $\lambda^+$-saturates $T$, then $T$ is simple and does not admit $\Delta_{k'+1, k'}$ for any $k' < k_*(\mathbf{s})$. If also $\tau = \aleph_0$, then $T$ is low.
\end{theorem}
\begin{proof}
$T$ is simple by Theorem~\ref{baseline}(A), and if $\tau = \aleph_0$ then $T$ is low by Theorem~\ref{baseline}(B). $T$ does not admit $\Delta_{k'+1, k'}$ for any $k' < k_*(\mathbf{S})$ by Theorem~\ref{NonSatPortionGeneral}.
\end{proof}

\begin{theorem}\label{MasterTheoremI}
Suppose $\mathbf{s}= (\lambda, \mu, \theta, \tau,k)$ is a suitable sequence of type I. If any of the following conditions hold, then $\mathcal{U}_{\mathbf{s}}$ $\lambda^+$-saturates $T$.
\begin{itemize}
	\item[(A)] $T$ is simple and has $<k$-type amalgamation;
	\item[(B)] $T$ is simple and has $\Lambda$-type amalgamation for all $\Lambda \in \mathbf{\Lambda}$ with $\mbox{dim}(\Delta_\Lambda) < k_*(\mathbf{s})$;
	\item[(C)] $T$ is simple and $k_*(\mathbf{s}) = 3$.
\end{itemize}
\end{theorem}
\begin{proof}
If (A) holds, use Theorems~\ref{SatPortion} and \ref{satPortionGeneral1}. If (B) holds, use Theorems~\ref{SatPortionv2} and \ref{satPortionGeneral1}. If (C) holds and if $k = 3$, then (A) holds; otherwise, use Theorem~\ref{SatPortionv3}.
\end{proof}

\begin{theorem}\label{MasterTheoremII}
	Suppose $\mathbf{s}= (\lambda, \mu, \theta, \tau,k)$ is a suitable sequence of type II. If any of the following conditions hold, then $\mathcal{U}_{\mathbf{s}}$ $\lambda^+$-saturates $T$.
	\begin{itemize}
		\item[(A)] $T$ is low and has $<k$-type amalgamation;
		\item[(B)] $T$ is low and has $\Lambda$-type amalgamation for all $\Lambda \in \mathbf{\Lambda}$ with $\mbox{dim}(\Delta_\Lambda) < k_*(\mathbf{s})$;
		\item[(C)] $T$ is low and $k_*(\mathbf{s}) = 3$.
	\end{itemize}
\end{theorem}
\begin{proof}
	If (A) holds, use Theorems~\ref{SatPortion} and \ref{satPortionGeneral1}. If (B) holds, use Theorems~\ref{SatPortionv2} and \ref{satPortionGeneral1}. If (C) holds and if $k = 3$, then (A) holds; otherwise, use Theorem~\ref{SatPortionv3}.
\end{proof}

\begin{theorem}\label{MasterTheoremIII}
	Suppose $\mathbf{s}= (\lambda, \mu, \theta, \tau,k)$ is a suitable sequence of type III or IV (i.e. such that $\theta = \aleph_0$). Then whenever $k_*(\mathbf{s}) \leq k' < n' < \aleph_0$, we have that $\mathcal{U}_{\mathbf{s}}$ $\lambda^+$-saturates $T_{n', k'}$.
\end{theorem}
\begin{proof}
If $k = k_*(\mathbf{s})$, use Theorems~\ref{SatPortionv4} and \ref{satPortionGeneral1}. Otherwise, use Theorems~\ref{SatPortionv5} and \ref{satPortionGeneral1}.
\end{proof}

We now use these theorems to get several dividing lines in Keisler's order. The following theorem is due to Malliaris and Shelah \cite{Optimals}.

\begin{theorem}\label{SimpleDividingLineThm}
	Suppose there is a supercompact cardinal. Then simplicity is a principal dividing line in Keisler's order.
\end{theorem}
\begin{proof}
	Let $\tau$ be supercompact. Write $\theta = \mu = \tau$; let $\lambda = \mu^{+\omega}$. Then $\mathbf{s} = (\lambda, \mu, \theta, \tau, 3)$ is a suitable sequence of type I, and $k_*(\mathbf{s}) = 3$. (It would be closer to Malliaris and Shelah's proof to take $\lambda = \mu^+$ and to take $k = \aleph_0$.) By Theorem~\ref{baseline2}, if $\mathcal{U}_{\mathbf{s}}$ $\lambda^+$-saturates $T$ then $T$ is simple. By Theorem~\ref{MasterTheoremI}, if $T$ is simple then $\mathcal{U}_{\mathbf{s}}$ $\lambda^+$-saturates $T$.
\end{proof}

I proved the following in \cite{LowDividingLine}.
\begin{theorem}\label{LowDividingLineThm}
	Lowness is a principal dividing line in Keisler's order.
\end{theorem}
\begin{proof}
	Let $\tau = \aleph_0$, let $\theta = \aleph_1$, let $\mu = \theta^{<\theta} = 2^{\aleph_0}$, and let $\lambda = \mu^{+\omega}$. Then $\mathbf{s} = (\lambda, \mu, \theta, \tau, 3)$ is a suitable sequence of type II, and $k_*(\mathbf{s}) = 3$. By Theorem~\ref{baseline2}, if $\mathcal{U}_{\mathbf{s}}$ $\lambda^+$-saturates $T$ then $T$ is low. By Theorem~\ref{MasterTheoremII}, if $T$ is low then $\mathcal{U}_{\mathbf{s}}$ $\lambda^+$-saturates $T$.
\end{proof}

The following theorem sharpens Malliaris and Shelah's result \cite{InfManyClass} that for all $3 \leq k < k'-1$, $T_{k'+1, k'} \not \trianglelefteq T_{k+1, k}$.

\begin{theorem}\label{AmalgLowDividingLineThm}
	Suppose $3 \leq k \leq \aleph_0$. Then there is a principal dividing line in Keisler's order, which includes every countable low theory with $<k$-type amalgamation, but does not include any theory which admits $\Delta_{k'+1, k'}$ for some $k' < k$, nor any nonlow theory. In particular, $T_{k'+1, k'} \not \trianglelefteq T_{k+1, k}$ for all $k' < k$. 
\end{theorem}
\begin{proof}
	Let $\tau = \aleph_0$, let $\theta = \aleph_1$, let $\mu = 2^{\aleph_0}$, and let $\lambda = \mu^{+\omega}$. Then $\mathbf{s} = (\lambda, \mu, \theta, \tau, k)$ is a suitable sequence of type II, and $k_*(\mathbf{s}) = k$. By Theorem~\ref{baseline2}, if $\mathcal{U}_{\mathbf{s}}$ $\lambda^+$-saturates $T$ then $T$ is low and does not admit $\Delta_{k'+1, k'}$ for any $k' < k$. By Theorem~\ref{MasterTheoremII}, if $T$ is low  and has $<k$-type amalgamation, then $\mathcal{U}_{\mathbf{s}}$ $\lambda^+$-saturates $T$.
\end{proof}

We also get the following, under the presence of a supercompact cardinal. 

\begin{theorem}\label{AmalgSimpleDividingLineThm}
	Suppose $3 \leq k_* \leq \aleph_0$, and suppose there is a supercompact cardinal. Then there is a principal dividing line in Keisler's order, which includes every countable simple theory $T$ with $<k$-type amalgamation, but does not include any theory which admits  $\Delta_{k'+1, k'}$ for some $k' < k$, nor any unsimple theory. 
\end{theorem}
\begin{proof}
Let $\tau$ be supercompact, let $\mu = \theta = \tau$, and let $\lambda \geq \mu^{+\omega}$. Then $\mathbf{s} = (\lambda, \mu, \theta, \tau, k)$ is a suitable sequence of type I, and $k_*(\mathbf{s}) = k$. By Theorem~\ref{baseline2}, if $\mathcal{U}_{\mathbf{s}}$ $\lambda^+$-saturates $T$ then $T$ is simple and does not admit $\Delta_{k'+1, k'}$ for any $k' < k$. By Theorem~\ref{MasterTheoremI}, if $T$ is simple and has $<k$-type amalgamation, then $\mathcal{U}_{\mathbf{s}}$ $\lambda^+$-saturates $T$.
\end{proof}

I observed the following in \cite{KeislerNotLinear}; Malliaris and Shelah observed it independently in \cite{InterpOrders}. 
\begin{corollary}\label{NotLinearCor}
	If there is a supercompact cardinal, then Keisler's order is not linear.
\end{corollary}
\begin{proof}
	Compare $T_{nlow}$ with $T_{4, 3}$, say. $T_{nlow}$ has $<\aleph_0$-type amalgamation, but is not low; $T_{4, 3}$ is low, but admits $\Delta_{4, 3}$. Thus we conclude by Theorems~\ref{LowDividingLineThm} and ~\ref{AmalgSimpleDividingLineThm}.
\end{proof}

We now consider suitable sequences of types III and IV. 

Note that if supercompact cardinals exist, then $T$ is simple if and only if there is some suitable sequence $\mathbf{s}$ of type I such that $\mathcal{U}_{\mathbf{s}}$ $\lambda^+$-saturates $T$, and in $ZFC$, $T$ is low if and only if there is some suitable sequence $\mathbf{s}$ of type II such that $\mathcal{U}_{\mathbf{s}}$ $\lambda^+$-saturates $T$. (We need the supercompact cardinal for type I because otherwise type I sequences don't exist).

This inspires the following provisional definitions:

\begin{definition}
Suppose $T$ is a countable complete theory. Then $T$ is strongly low if there is some suitable sequence $\mathbf{s}$ of type III such that $\mathcal{U}_{\mathbf{s}}$ $\lambda^+$-saturates $T$. $T$ is superlow if there is some suitable sequence $\mathbf{s}$ of type IV such that $\mathcal{U}_{\mathbf{s}}$ $\lambda^+$-saturates $T$.
\end{definition}

It follows from Theorem~\ref{baseline2} that strongly low and superlow both imply low; and it follows from Theorem~\ref{MasterTheoremIII} that for all $3 \leq k < n < \aleph_0$, $T_{n, k}$ is strongly low and superlow (and hence every stable theory is strongly low and superlow).

We suspect that strongly low = low. On the other hand, Malliaris and Shelah introduce some new simple theories in \cite{NewSimpleTheory} and show that they are $\trianglelefteq$-incomparable with each $T_{n, k}$; in our terminology, these new theories are strongly low and have $<\aleph_0$-type amalgamation, but are not superlow.

We finish with the following conjecture and questions.

\begin{conjecture}\label{ConjMaster}
Suppose $\mathbf{s}$ is a suitable sequence and $T$ is a complete countable theory.

\begin{itemize}
	\item[(I)] If $\mathbf{s}$ is of type I, then $\mathcal{U}_{\mathbf{s}}$ $\lambda^+$-saturates $T$ if and only if $T$ is simple and has $\mathcal{P}^-(k_*(\mathbf{s}))$-amalgamation of models.
	\item[(II)] If $\mathbf{s}$ is of type II, then $\mathcal{U}_{\mathbf{s}}$ $\lambda^+$-saturates $T$ if and only if $T$ is low and has $\mathcal{P}^-(k_*(\mathbf{s}))$-amalgamation of models.
	\item[(III)] If $\mathbf{s}$ is of type III, then $\mathcal{U}_{\mathbf{s}}$ $\lambda^+$-saturates $T$ if and only if $T$ is strongly low and has $\mathcal{P}^-(k_*(\mathbf{s}))$-amalgamation of models.
	\item[(IV)] If $\mathbf{s}$ is of type IV, then $\mathcal{U}_{\mathbf{s}}$ $\lambda^+$-saturates $T$ if and only if $T$ is superlow and has $\mathcal{P}^-(k_*(\mathbf{s}))$-amalgamation of models.
\end{itemize}
\end{conjecture}

\noindent \textbf{Question.} Are there model-theoretic characterizations of strongly low and superlow? Formally, are these notions absolute?

\bibliography{main}

\begin{thebibliography}{10}

\bibitem{nLinked}
J.~Barnett.
\newblock Weak {V}ariants of {M}artin's {A}xiom.
\newblock {\em Fundamenta Mathematicae}, 141:61--73, 1992.

\bibitem{Buechler}
S.~Buechler.
\newblock Lascar strong types in some simple theories.
\newblock {\em Journal of Symbolic Logic}, 64(2):817--824, 1999.

\bibitem{SupersimpleNonlow}
E.~Casanovas and B.~Kim.
\newblock A {S}upersimple {N}onlow {T}heory.
\newblock {\em Notre Dame J. Formal Logic}, 39(4):507--518, 1998.

\bibitem{Partitions}
R.~Engleking and M.~Karlowicz.
\newblock Some theorems of set theory and their topological consequences.
\newblock {\em Fundamenta Mathematicae}, 57:275--285, 1965.

\bibitem{Combinatorics}
Erd\"{o}s, Hajnal, M\'{a}t\'{e}, and Rado.
\newblock {\em Combinatorial {S}et {T}heory: {P}artition {R}elations for {C}ardinals}.
\newblock Studies in Logic and the Foundations of Mathematics. 1984.

\bibitem{Hrush}
E.~Hrushovski.
\newblock {\em Pseudo-finite fields and related structures}, volume~11 of {\em Quaderni di Matematica}, pages 151--212.
\newblock 2002.

\bibitem{Jech}
T.~Jech.
\newblock {\em Set theory. The third millennium edition, revised and expanded}.
\newblock Springer Monographs in Mathematics. Springer-Verlag, Berlin, 2003.

\bibitem{Kanamori}
A.~Kanamori.
\newblock {\em The {H}igher {I}nfinite}.
\newblock de Gruyter Series in Logic and its Applications. Springer-Verlag, Berlin, second edition, 2009.

\bibitem{Keisler}
J.~Keisler.
\newblock Ultraproducts which are not saturated.
\newblock {\em Journal of Symbolic Logic}, 32:23--46, 1967.

\bibitem{KimForking}
B.~Kim.
\newblock Forking in {S}imple {U}nstable {T}heories.
\newblock {\em Journal of the London Mathematical Society}, 57:257--267, 1998.

\bibitem{GenAmalg}
B.~Kim, A.~Kolesnikov, and A.~Tsuboi.
\newblock Generalized amalgamation and $n$-simplicity.
\newblock {\em Annals of Pure and Applied Logic}, 155(2):97--114, 2008.

\bibitem{Kunen}
K.~Kunen.
\newblock Ultrafilters and independent sets.
\newblock {\em Transactions of the American Mathematical Society}, 172:299--306, 1972.

\bibitem{KunenSets}
K.~Kunen.
\newblock {\em Set theory}, volume~34 of {\em Studies in Logic}.
\newblock College Publications, London, 2011.

\bibitem{KSComb}
K.~Kuratowski.
\newblock Sur une caract\'{e}risation des alephs.
\newblock {\em Fund. Math.}, (38):14--17, 1951.

\bibitem{NewSimpleTheory}
M.~Malliaris and S.~Shelah.
\newblock An example of a new simple theory.
\newblock Preprint.

\bibitem{InterpOrders}
M.~Malliaris and S.~Shelah.
\newblock A new look at interpretability and saturation.
\newblock Preprint.

\bibitem{InfManyClass}
M.~Malliaris and S.~Shelah.
\newblock Keisler's order has infinitely many classes.
\newblock {\em Israel Journal of Math}, 62(4), 2011.

\bibitem{DividingLine}
M.~Malliaris and S.~Shelah.
\newblock A dividing line within simple unstable theories.
\newblock {\em Advances in Math}, 249:250--288, 2013.

\bibitem{Optimals}
M.~Malliaris and S.~Shelah.
\newblock Existence of optimal ultrafilters and the fundamental complexity of simple theories.
\newblock {\em Advances in Math}, 290:614--681, 2016.

\bibitem{ShelahIso}
S.~Shelah.
\newblock {\em Classification {T}heory}.
\newblock North-Holland, Amsterdam, 1978.

\bibitem{SP}
S.~Shelah and D.~Ulrich.
\newblock $\leq_{SP}$ can have infinitely many classes.
\newblock Preprint.

\bibitem{InterpOrders2Ulrich}
D.~Ulrich.
\newblock Cardinal characteristics of models of set theory.
\newblock Preprint.

\bibitem{BVModelsUlrich}
D.~Ulrich.
\newblock Keisler's {O}rder and {F}ull {B}oolean-{V}alued {M}odels.
\newblock Preprint.

\bibitem{LowDividingLine}
D.~Ulrich.
\newblock Lowness is a {D}ividing {L}ine in {K}eisler's {O}rder.
\newblock Submitted, April 2017.

\bibitem{KeislerNotLinear}
D.~Ulrich.
\newblock Keisler's order is not linear, assuming a supercompact.
\newblock {\em Journal of Symbolic Logic}, 83(2):634---631, 2018.

\end{thebibliography}
\end{document}